\documentclass[a4paper]{article}
\usepackage{authblk}
\usepackage[english]{babel}
\usepackage[utf8]{inputenc}
\usepackage[T1]{fontenc}
\usepackage[a4paper,top=3cm,bottom=2cm,left=3cm,right=3cm,marginparwidth=1.75cm]{geometry}
\usepackage{titlesec}
\usepackage{bm}
\usepackage{amsmath}
\usepackage{mathtools}
\usepackage{amssymb}
\usepackage{mathrsfs}
\usepackage{graphicx}
\usepackage[colorinlistoftodos]{todonotes}
\usepackage[colorlinks=true, allcolors=blue]{hyperref}
\usepackage[numbers,sort&compress]{natbib}
\usepackage[amsmath,thmmarks]{ntheorem}
\usepackage{enumerate}
\usepackage{indentfirst}

\titleformat{\section}[hang]{\large\bfseries}{\thesection\quad}{0pt}{}[]
\titleformat{\subsection}[hang]{\normalsize\bfseries}{\thesubsection\quad}{0pt}{}[]

\theorembodyfont{\upshape}
\newtheorem{theorem}{Theorem}[section]
\newtheorem{lemma}{Lemma}[section]
\newtheorem{proposition}[theorem]{Proposition}

\newtheorem{definition}[theorem]{Definition}
\newtheorem{remark}[theorem]{Remark}
\newtheorem{assumption}[theorem]{Assumption}
\newtheorem{examples}[theorem]{Examples}
\newenvironment{proof}{{\noindent\it Proof}\quad}{\hfill $\square$\par}
\numberwithin{equation}{section}

\begin{document}
	\title{\textbf{General large deviations and functional iterated logarithm law for multivalued McKean-Vlasov stochastic differential equations}}
	
	\author{Lingyan Cheng$^\sharp$, Wei Liu$^\dag$, Huijie Qiao$^\ddag$ and Fengwu Zhu$^\dag$}
	\affil{$^\sharp$  School of Mathematics and Statistics, Nanjing University of Science and Technology,
		Nanjing, Jiangsu 210094, PR China \\
		$^\dag$ School of Mathematics and Statistics, Wuhan University, Wuhan, Hubei 430072, PR China\\
		$^\ddag$  School of Mathematics, Southeast University, Nanjing, Jiangsu 211189, PR China\\
		cly@njust.edu.cn,\ \ \ wliu.math@whu.edu.cn,\ \ \ hjqiaogean@seu.edu.cn,\ \ \ fwzhu$\_$math@whu.edu.cn}
	
	\date{}
	
	\maketitle
	
	\noindent{\bf Abstract}
	\quad In this paper, we present sufficient conditions and criteria to establish general large and moderate deviation principles for multivalued McKean-Vlasov stochastic differential equations (SDEs in short) by means of the weak convergence approach, under non-Lipschit assumptions on the coefficents of the equations. Furthermore, by applying the large deviation estimates we obtain the functional iterated logarithm law for the solutions of multivalued McKean-Vlasov SDEs.
	\vskip0.3cm
	
	\noindent {\bf Keywords}{\quad Multivalued operator, McKean-Vlasov SDE, Large deviations, Weak convergence method, Functional iterated logarithm law.}
	
	\noindent{{\bf Mathematics Subject Classification (2020)} 60F10 · 60H10.}
	\vskip0.3cm
	
	
	
	\section{Introduction}

	Consider the following multivalued McKean-Vlasov stochastic differential equation (SDE in short):
	\begin{equation}\label{mve}
		\begin{cases}
			\mathrm{d}X_{t}\in b\left(X_{t},\mathcal{L}_{X_{t}}\right)\,\mathrm{d}t+\sigma \left(X_{t},\mathcal{L}_{X_{t}}\right)\,\mathrm{d}W_{t}-A(X_{t})\,\mathrm{d}t,\\
			X_{0}=x\in \overline{D(A)},
		\end{cases}
	\end{equation}
	where $A:\mathbb{R} ^{d}\rightarrow 2^{\mathbb{R} ^{d}}$ is a multivalued maximal monotone operator with domain $D(A)$ which will be described in the next section, and $ (W_{t})_{t\geq 0}$ is a $d$-dimensional standard Brownian motion on some filtered probability space $(\Omega , P ,\left( \mathcal{F}_{t}\right)_{t\geq 0})$. Here the coefficients $b$ and $\sigma$ depend not only on the current state $X_{t}$ but also on $\mathcal{L}_{X_t}$ which is the law of $X_t$. Notice that when the maximal monotone operator $A$ is the sub-differential of the indicator of a convex set, equation (\ref{mve}) describes the reflected diffusion process.
	
	
	Note that when $A=0$, equation (\ref{mve}) is just the  classical McKean-Vlasov SDE, which was first suggested by Kac \cite{[Kac]} as a stochastic toy model for the Vlasov kinetic equation of plasma, and then introduced by McKean \cite{[MK]} to model plasma dynamics. These equations correspond to nonlinear Fokker-Planck equations in the field of partial differential equation. They also describe the limiting behaviors of individual particles (and empirical measure equivalently) of mean field interacting particle system when the number of particles goes to infinity (so-called propagation of chaos). For this reason McKean-Vlasov SDEs are also referred as mean-field SDEs. The theory and applications of McKean-Vlasov SDEs and associated interacting particle systems have been extensively studied by a large number of researchers under various settings due to their wide range of applications in several fields, including physics, chemistry, biology, economics, financial mathematics etc., see \cite{[CGM],[GH],[GLWZ1],[GLWZ2],[LW],[LWZ],[MS],[TH]} and the references therein. For the existence and uniqueness of solution to McKean-Vlasov SDE, the reader is referred to \cite{[MK],[G],[MS],[SZ],[MV],[HSS],[RZ]} as well as the references therein.
	
	In the case $A\neq 0$, the equation (\ref{mve}) is called a multivalued SDE when the coefficients $b$ and $\sigma$ do not depend on $\mathcal{L}_{X_t}$. The existence and uniqueness of solution for multivalued SDE were first obtained by C\'epa \cite{cepa,cepa1}. Zhang \cite{ZX} extended C\'epa's results to the infinite-dimensional case and relaxed the Lipschitz assumption on $b$ to the monotone case. Ding and Qiao \cite{DQ} relaxed the assumption on both $b$ and $\sigma$ to the non-Lipschitz case. Xu \cite{xu} showed explicit solutions for multivalued SDEs. Ren et al. \cite{zhwu} obtained the exponential ergodicity of non-Lipschitz multivalued SDEs. The regularity of invariant measure was studied by Ren and Wu \cite{ren}. On the other hand, the equation (\ref{mve}) is referred to as a multivalued McKean-Vlasov SDE in the distribution dependent case. The existence and uniqueness of the solution to equation (\ref{mve}) were first studied by Chi \cite{CHM}. Explicitly, under the global Lipschitz conditions on the coefficients $b$ and $\sigma$, Chi proved that there exists a unique pair of processes $(X_{.}, K_{.})$ so that
	$$X_{t}=x+\int_{0}^{t}b\left( X_{s},\mathcal{L}_{X_{s}}\right)\,\mathrm{d}s+\int_{0}^{t}\sigma\left( X_{s},\mathcal{L}_{X_{s}}\right)\,\mathrm{d}W_{s}-K_{t},$$
	where $K_{t}$ is a finite-variation process which will be explained in (\ref{kd}). Qiao and Gong \cite{QG} extended these results to the one-sided Lipschitz case.
	
	In this paper we consider the following small perturbation of equation (\ref{mve}): for any $\epsilon\in (0,1]$,
	\begin{equation}\label{spmve}
		\begin{cases}
			\mathrm{d}X_{t}^{\epsilon}\in b_{\epsilon}\left( X_{t}^{\epsilon},\mathcal{L}_{X_{t}^{\epsilon}}\right)\mathrm{d}t+\sqrt{\epsilon}\,\sigma_{\epsilon}\left( X_{t}^{\epsilon},\mathcal{L}_{X_{t}^{\epsilon}}\right)\mathrm{d}W_{t}-A_{\epsilon}(X_{t}^{\epsilon})\,\mathrm{d}t,\\
			X^{\epsilon}_{0}=x\in \overline{D(A)}.
		\end{cases}
	\end{equation}
	Denote by $(X^{\epsilon}_{t}, K^{\epsilon}_{t})$ the solution of equation (\ref{spmve}). The main purpose of this paper is to establish the large and moderate deviation principles for the solution $X^{\epsilon}_{t}$ of equation (\ref{spmve}) in the space $\mathbb{S}:=C([0,T],\overline{D(A)})$, i.e. the asymptotic estimates of the probabilities $P(X_{t}^{\epsilon}\in B),\ B\in\mathcal{B}(\mathbb{S})$. Large deviation can provide an exponential estimate for tail probability (or the probability of a rare event) in terms of some explicit rate function. In the case of stochastic processes, the heuristics underlying large deviation theory is to identify a deterministic path around which the diffusion is concentrated with overwhelming probability, so that the stochastic motion can be seen as a small random perturbation of this deterministic path.
	
	For the classic McKean-Vlasov SDEs, i.e. $A=0$, the large deviation principle (LDP in short) was studied by Herrmann et al. \cite{[HIP]} and Dos Reis et al. \cite{[RST]}. The approach in \cite{[HIP]} and \cite{[RST]} is to  first replace the distribution $\mathcal{L}_{X_{t}^{\epsilon}}$ of $X_t^{\epsilon}$ in the coefficients with a Dirac measure $\delta_{X^0_t}$, and then to use discretization, approximation and exponential equivalence arguments. Recently the second author et al.
	\cite{liuw} established the large and moderate deviation principles (MDP in short) for McKean-Vlasov SDEs with jumps by means of the weak convergence method.
	
	For the multivalued SDEs, i.e. the coefficients do not depend on the law of $X_t$, C\'epa \cite{cepa} established the
	LDP in the one-dimensional case by contraction principle. Ren et al. \cite{zhang} established the LDP for multivalued SDEs with monotone drifts (in particular containing a class of SDEs with reflection in a convex domain) by using the weak convergence approach. A slightly more general case was considered by Ren et al. \cite{rwz}, also by the weak convergence approach.
	
	Compared with the results introduced previously, most of the difficulties for dealing with the multivalued McKean-Vlasov SDEs (\ref{spmve}) are threefold. First, the coefficients $b$ and $\sigma$ depend on the {\bf distribution $\mathcal{L}_{X_t}$ of the solution}, which makes the equation (\ref{spmve}) a nonlinear diffusion. For this problem, we adopt the {\bf decoupling strategies} for establishing LDP in the second author et al. \cite{liuw}. The second difficulty comes from the multivalued operator $A$ and the {\bf finite-variation process $K_t$}. One only knows that $K_t$ is continuous and could not prove any further regularity such as H\"older continuity. Therefore, the classical method of time-discretized method is almost not applicable. The proposition \ref{multi} below, which is due to C\'epa \cite{cepa1}, plays an important role. The third difficulty is that the multivalued operators $A_\epsilon$ varying as $\epsilon$. Since $A, A_\epsilon$ are set-valued operators, the difference between $A$ and $A_\epsilon$ can not be directly measured, while there is no such problem for $b_\epsilon$ and $\sigma_\epsilon$. Here we use the {\bf Yosida approximations} introduced in \cite{rwz} to overcome this problem.
	
	The weak convergence approach, developed by Dupuis and Ellis \cite{[DE]} (see also \cite{BD2019,BD2000,BDM2008,BDM2011}), is proved to be a powerful tool to establish large deviation principles for various dynamical systems, see e.g. \cite{zhang,renzhang,zhangren,rwz,hu,liuw} and the references therein. In this paper, we will use the weak convergence approach to establish the large and moderate deviation principles for the multivalued McKean-Vlasov SDEs (\ref{spmve}). Recently Fang and the last three authors \cite{FLQZ} obtained the large and moderate deviation principles for small perturbations of Multi-valued McKean-Vlasov SDEs also by the weak convergence approach. However, it is worth to pointing out that  the two works both use weak convergence approach but in different ways, and the non-lipschitzian assumptions on the coefficients for LDP and MDP results in this paper are more weaker than those in \cite{FLQZ}. On the other hand, the multivalued operators $A_\epsilon$ in the system depends on $\epsilon$. Therefore this paper can be seen as the generalization of the LDP and MDP results in \cite{FLQZ}.
	
	Besides the large and moderate deviation principles, in this paper we also study the (large time and small time) functional iterated logarithm laws for the multivalued McKean-Vlasov SDEs, by applying the LDP obtained by weak convergence approach. The functional iterated logarithm laws were investigated in \cite{[RST]} for the McKean-Vlasov SDEs and \cite{rwz} for the multivalued SDEs. We extend these results in our framework.
	
	The rest of this paper is organized as follows. In Section $2$, we recall some well-known facts about multivalued McKean-Vlasov SDEs and a criterion for Laplace principle. The main results about LDP and MDP as well as the existence and uniqueness of strong solution to equation (\ref{mve}) are presented in Section $3$, and the proofs are given in Section $4$. The last section is devoted to the functional iterated logarithm law.


	\section{Preliminaries}\label{pre}
	
	In this section, we first recall some basic notions and definitions.
	
	\subsection{Maximal Monotone Operators}
	
	Let $ 2^{\mathbb{R} ^{d}} $ denote the set containing all the subsets of $ \mathbb{R} ^{d} $. A map $ A:\mathbb{R} ^{d}\rightarrow 2^{\mathbb{R} ^{d}} $ is called a multivalued operator. Denote
	
	\begin{align*}
		D\left( A\right) &:=\{ x\in \mathbb{R} ^{d}:A\left( x\right) \neq \emptyset \} ,\\	
		Im\left( A\right) &:=\cup_{{x\in D(A)}} A\left( x\right) ,\\
		Gr\left( A\right) &:=\{ \left( x,y\right) \in \mathbb{R} ^{d}\times \mathbb{R} ^{d}:x\in \mathbb{R} ^{d},y \in A\left( x\right) \}.
	\end{align*}
	
	
	\begin{definition}\label{mo}
		
		(1) A multivalued operator $A$ is called \textit{monotone} if
		$$
		\langle x-x',y-y'\rangle \geq 0,\quad\forall \left( x,y\right),\left( x',y'\right) \in Gr\left( A\right).
		$$
		
		(2) A multivalued operator $A$ is called \textit{maximal monotone} if
		$$
		\left( x,y\right) \in Gr\left( A\right) \Leftrightarrow \langle x-x',y-y'\rangle \geq 0,\quad\forall \left( x',y'\right) \in Gr\left( A\right).
		$$
		
	\end{definition}
	
	Let $A$ be a \textit{multivalued maximal monotone operator} on $\mathbb{R}^{d}$. It is known from \cite{BH1973} that Int$(D(A))$ and $\overline{D(A)}$ are convex subsets of $\mathbb{R}^{d}$ and
	\begin{center} 	Int$(D(A))$=Int$(\overline{D(A)}).$ \end{center}
	
	
	Given $T>0$, let
	\begin{equation}\label{kd}
		\mathcal{V}=\left\{K\in C([0,T],\mathbb{R}^{d}):\, K \text{ is of finite variation and }K_{0}=0 \right\}.
	\end{equation}
	For any $K\in \mathcal{V}$ and $t\in [0,T]$, denote the variation of $K$ on $[0,t]$ by $|K|_{0}^{t}$ and  $|K|_{TV}:=|K|_{0}^{T}$. We introduce the following set of functions pair:
	\begin{align*}
		\mathcal{A}:=&\big\{(X,K):X\in C\left([0,T],\overline{D(A)}\right),K\in \mathcal{V}\;\text{and}
		\\
		\phantom{=\;\;}&\text{for all}\;(x,y)\in Gr(A),\langle X_{t}-x,dK_{t}-ydt \rangle\geq0 \big\}.
	\end{align*}
	
	It is known from \cite{CHM,cepa1} that for any pairs of $(X_{t},K_{t}), (\tilde{X}_{t},\tilde{K}_{t})\in \mathcal{A}$,
	\begin{equation}\label{mono}
		\langle X_{t}-\tilde{X}_{t},\mathrm{d}K_{t}-\mathrm{d}\tilde{K}_{t}\rangle\geq 0.
	\end{equation}
	
	We recall the following results due to C\'epa \cite{cepa1}, which will be used in the proofs in Section 4.
	\begin{proposition}\label{multi}
		Assume that Int$(D(A))\neq \emptyset$. For any $a\in$ Int$(D(A))$, there exist three positive constants $\lambda_{1}$, $\lambda_{2}\,\text{and}\, \lambda_{3}$ such that for any $(X,K)\in\mathcal{A}$ and $0\leq s<t\leq T$,
		\begin{equation}
			\int_{s}^{t}\langle X_{r}-a,\mathrm{d}K_{r}\rangle_{\mathbb{R}^{d}}\geq \lambda_{1}|K|_{t}^{s}-\lambda_{2}\int_{s}^{t}|X_{r}-a|\mathrm{d}r-\lambda_{3}(t-s),
		\end{equation}
		where $|K|_{s}^{t}$ denotes the total variation of $K$ on $[s,t]$.
	\end{proposition}

	\begin{proposition}\label{kcon}
		Let $\left\{K^n,\,n\in\mathbb{N}\right\}\in \mathcal{V}$ converge to some $K$ in $C([0,T],\mathbb{R}^d)$. If $\sup_{n\in\mathbb{N}}|K^n|_{TV}\leq \infty$, then $K\in \mathcal{V}$ and for any $X^n\in C([0,T],\mathbb{R}^d)$ converging to $X$ in $C([0,T],\mathbb{R}^d)$,
		\begin{align}
			\lim\limits_{n\to\infty}\int_{0}^{T}\langle X_s^n,\mathrm{d}K_s^n\rangle=\int_{0}^{T}\langle X_s,\mathrm{d}K_s\rangle.
		\end{align}
	\end{proposition}

	\begin{remark}
		If A is a monotone operator and $\alpha>0$, then $(I+\alpha A)^{-1}$ is a Lipschitz function on $\mathbb{R}^d$ with Lipschitz constant equal to 1. Denote $J^\alpha:=(I+\alpha A)^{-1}$ and define the Yosida approximation of the operator $A$ as
		\begin{equation}\label{yosida}
			A^\alpha(x):=\frac{1}{\alpha}\left(x-J^\alpha(x)\right), \quad x\in\mathbb{R}^d.
		\end{equation}
		Then $A^\alpha(x)$ is monotonic and Lipschitz continuous with Lipschitz constant $\frac{1}{\alpha}$. Moreover, for every $x\in\mathbb{R}^d$ we have
		\begin{align}\label{yosidap}
			A^\alpha(x)\in A(J^\alpha(x))
		\end{align}
		and for every $x\in D(A)$, as $\alpha\to 0$,
		\begin{align}\label{yosi0}
			A^\alpha(x)\to A^0(x) \,\,\,\,\text{and}\,\,\,\,|A^\alpha(x)|\to |A^0(x)|,
		\end{align}
		where $A^0(x)$ is called the minimal section of $A$ and $|A^0(x)|<\infty.$
	\end{remark}


	\subsection{Solutions to multivalued Mckean-Vlasov SDEs}
	
	\begin{definition}\label{strong}
		(\textbf{Strong solution}) A pair of $\mathcal{F}_t$-adapted processes $(X,K)$ is said to be a strong solution of equation (\ref{mve}) with the initial value $x$, if $(X,K)$ is defined on a filtered probability space $(\Omega,\mathcal{F},\{\mathcal{F}_t\}_{t\geq0},P)$ such that
		\begin{enumerate} [(i)]
			\item $P(X_0=x)=1$,
			\item $X_t\in \mathcal{F}^W_t$, where $\left\{\mathcal{F}^W_t\right\}_{t\in[0,T]}$ stands for the $\sigma$-field filtration generated by $W$,
			\item $(X_.(\omega),K_.(\omega)\Big)\in\mathcal{A}$$\,\, P\text{-a.s.}$,
			\item it holds that
			\begin{align*}
				\int_{0}^{t}\left(|b(X_{s},\mathcal{L}_{X_{s}})|+\lVert\sigma(X_{s},\mathcal{L}_{X_{s}})\rVert^{2}\right)\mathrm{d}s <\infty, \quad P\text{-a.s.}
			\end{align*}
			and
			\begin{align*}
				X_{t}=x-K_{t}+\int_{0}^{t}b(X_{s},\mathcal{L}_{X_{s}})\mathrm{d}s+\int_{0}^{t}\sigma(X_{s},\mathcal{L}_{X_{s}})\mathrm{d}W_{s}, \quad \forall t\in[0,T].
			\end{align*}
		\end{enumerate}
	\end{definition}

	\begin{definition}\label{solu}
		(\textbf{Weak solution}) We say that (\ref{mve}) admits a weak solution with intial law $\mathcal{L}_{\tilde{X}_0}\in \mathcal{P}(\mathbb{R}^d)$ if there exists a stochastic basis $\tilde{\mathcal{S}}:=(\tilde{\Omega},\tilde{\mathcal{F}},\{\tilde{\mathcal{F}}_t\}_{t\geq0},\tilde{P})$, a d-dimensional standard $\tilde{\mathcal{F}}_t$-Brownian motion $(\tilde{W}_t)_{t\geq0}$ as well as a pair of continuous and $\tilde{\mathcal{F}}_t$-adapted processes $(\tilde{X},\tilde{K})$ defined on $\tilde{\mathcal{S}}$ such that
		\begin{enumerate} [(i)]
			\item $\tilde{X}_0$ has the initial law $\mathcal{L}_{\tilde{X}_0}$ and $\left(\tilde{X}_{.}(\omega),\tilde{K}_{.}(\omega)\right)\in \mathcal{A}$, $\,\,\, \tilde{P}\text{-a.s.}$,
			\item it holds that
			\begin{align*}
				\int_{0}^{t}\left(|b(\tilde{X}_{s},\mathcal{L}_{\tilde{X}_{s}})|+\lVert\sigma(\tilde{X}_{s},\mathcal{L}_{\tilde{X}_{s}})\rVert^{2}\right)\mathrm{d}s <\infty,\quad \tilde{P}\text{-a.s.},
			\end{align*}
			\item
			$$\tilde{X}_{t}=\tilde{X}_{0}-\tilde{K}_{t}+\int_{0}^{t}b(\tilde{X}_{s},\mathcal{L}_{\tilde{X}_{s}})\mathrm{d}s+\int_{0}^{t}\sigma(\tilde{X}_{s},\mathcal{L}_{\tilde{X}_{s}})\mathrm{d}\tilde{W}_{s}, \quad \forall t\in[0,T],$$
			where $\mathcal{L}_{\tilde{X}_{s}}=\tilde{P}\circ \tilde{X}_{s}^{-1}.$
		\end{enumerate}
	\end{definition}

	\begin{definition}\label{lawuni}
		(Uniqueness in Law ) Let $ (S;W,(X,K)) $ and $ (\widetilde{S};\widetilde{W},(\widetilde{X},\widetilde{K})) $ be two weak solutions with the same initial law $ \mathcal{L}_{X_{0}}=\mathcal{L}_{\widetilde{X}_{0}} $. The uniqueness in law is said to hold for (\ref{mve}) if $ (X,K) $ and $ (\widetilde{X},\widetilde{K}) $ have the same law.
	\end{definition}
	
	\begin{definition}\label{pathuni}
		(Pathwise Uniqueness) Let $ (S;W,(X,K)) $ and $ (S;W,(\widetilde{X},\widetilde{K})) $ be two weak solutions with $ X_{0}=\widetilde{X}_{0} $. The pathwise uniqueness is said to hold for (\ref{mve}) if for all $ t\in[0,T] $, $ (X_{t},K_{t})=(\widetilde{X}_{t},\widetilde{K}_{t}) $.
	\end{definition}

	Let $\mathcal{P}_2(\mathbb{R}^d)$ be the set of all probability measures on $(\mathbb{R}^d,\mathcal{B}(\mathbb{R}^d))$ with finite second moments. $\mathcal{P}_2(\mathbb{R}^d)$ is a Polish space equipped with $L^2$ Wasserstein distance defined by
	$$\mathbb{W}_{2}(\mu,\nu):=\inf\limits_{\xi\in \Pi(\mu,\nu)}\left(\iint_{\mathbb{R}^d\times\mathbb{R}^d}|x-y|^{2} \xi(\mathrm{d}x,\mathrm{d}y)\right)^{\frac{1}{2}},\ \ \forall \mu,\ \nu\in \mathcal{P}_2(\mathbb{R}^d)$$
	where $\Pi(\mu,\nu)$ is the set of all couplings of $\mu,\nu$, i.e. the set of all probability measures on $\mathbb{R}^d\times\mathbb{R}^d$ with marginal distributions $\mu$ and $\nu$. It is well known that $\mathbb{W}_{2}$ is a complete metric on $\mathcal{P}_2(\mathbb{R}^d)$.
	
	\begin{remark}\label{W2}
		From the definition above we know that for any $\mathbb{R}^{d}$-valued random variables $X$ and $Y$,
		$$\mathbb{W}_{2}\left(\mathcal{L}_{X},\mathcal{L}_{Y}\right)\leq \left[\mathbb{E}(X-Y)^{2}\right]^{\frac{1}{2}},$$ which will be used in the proofs later.
	\end{remark}
	
	Next we give the notations of martingale solutions. For $t\in[0,T]$, let
	\begin{align*}
		&\mathcal{W}:=C([0,T],\mathbb{R}^d)\times C([0,T],\mathbb{R}^d),\quad \mathbf{W}=\mathcal{B}(\mathcal{W}), \\
		&\mathcal{W}_t:=C([0,t],\mathbb{R}^d)\times C([0,t],\mathbb{R}^d),\quad \overline{\mathbf{W}}_t=\bigcap_{s>t}\mathcal{B}(\mathcal{W}_s).
	\end{align*}
	
	Denote the partial derivative operator as $\nabla:=\Big(\frac{\partial}{\partial x_1},\cdots,\frac{\partial}{\partial x_d}\Big)$ and the second partial derivate operator as $\nabla^2:=\Big(\frac{\partial^2}{\partial x_i\partial x_j}\Big)_{1\leq i,j\leq d}$. Let $C_b^2(\mathbb{R}^d)$ be the collection of all continuous functions which have bounded, continuous partial derivatives of every order up to 2.

	\begin{definition}\label{msolu}
		(\textbf{Martingale solution}) A probability measure $\mathbb{P}$ on $(\mathcal{W},\mathbf{W})$ is called a martingale solution of equation (\ref{mve}) with initial law $\mathcal{L}_{X_0}\in\mathcal{P}_2(\overline{D(A)})$ if $\mathbb{P}(\mathcal{A})=1$ and for each $f\in C_b^2(\mathbb{R}^d)$,
		\begin{align*}
			M_t^f:=&f(X_t)-f(X_0)+\int_{0}^{t}\langle\nabla f(X_s),\mathrm{d}K_s \rangle -\int_{0}^{t}\sum_{i} b_i(X_s,\mathcal{L}_{X_s})\nabla_i f(X_s)\mathrm{d}s \\
			&-\frac{1}{2}\int_0^t\sum_{i,j}( \sigma^*(X_s,\mathcal{L}_{X_s})\sigma(X_s,\mathcal{L}_{X_s}))_{ij}\nabla_{ij}^2f(X_s)\mathrm{d}s.
		\end{align*}
		is a $\overline{\mathbf{W}}_t$-adapted martingale, where $\mathcal{L}_{X_s}:=\mathbb{P}\circ X_s^{-1}$ denotes the law of $X_s$ under $\mathbb{P}$.
	\end{definition}

	\subsection{Large Deviation Principle}\label{ldp}
	
	We first give the definitions of the rate function and LDP. Let $\mathbb{S}$ be a Polish space and $\mathcal{B}(\mathbb{S})$ be the corresponding Borel $\sigma$-field.
	
	\begin{definition} [Rate function]
		A function $I:\mathbb{S}\rightarrow[0,\infty]$ is called a rate function if for every $a<\infty$, the level set $\left\{g\in\mathbb{S}:I(g)\leq a\right\}$ is a closed subset of $\mathbb{S}$.  $I$ is said to be a good rate function if all the level sets defined above are compact subsets of $\mathbb{S}$.
	\end{definition}
	
	\begin{definition}[Large deviation principle]\label{ldpdefin} Let $I$ be a rate function on $\mathbb{S}$. Given a collection $\{\hbar(\epsilon)\}_{\epsilon>0}$ of
		positive reals, a family $\left\{X^{\epsilon}\right\}_{\epsilon>0}$ of $\mathbb{S}$-valued random variables is said to satisfy the LDP on $\mathbb{S}$ with speed $\hbar(\epsilon)$ and rate function $I$ if the following two claims hold.
		\begin{itemize}
			\item[(a)]
			(Upper bound) For each closed subset $C$ of $\mathbb{S},$
			$$
			\limsup_{\epsilon \rightarrow 0} \hbar(\epsilon) \log {P}\left(X^{\epsilon} \in C\right) \leq - \inf_{x \in C} I(x).
			$$
			\item[(b)] (Lower bound) For each open subset $O$ of $\mathbb{S},$
			$$
			\liminf_{\epsilon \rightarrow 0} \hbar(\epsilon) \log {P}\left(X^{\epsilon} \in O\right) \geq - \inf_{x \in O} I(x).
			$$
		\end{itemize}
	\end{definition}	
	
	Next we recall a general criterion for laplace principle, which is equivalent to the LDP with the same rate function through the weak convergence approach (see \cite{BD2000,BDM2008,zhang}).
	
	
	
	For any $m\in \mathbb{N}$ and $T\geq0$, define
	\begin{equation}
		\mathcal{S}_{m}:=\left\{h\in L^{2}([0,T],\mathbb{R}^d):\int_{0}^{T}|h(s)|^{2}\mathrm{d}s\leq m\right\}.
	\end{equation}
	Equipped with the weak topology, $\mathcal{S}_{m}$ is a compact subset of $L^{2}([0,T],\mathbb{R}^d)$.
	
	Let
	\begin{align*}
		\mathcal{D}_{m}:=\bigg\{\varphi: \varphi\text{ is a predictable process and }
		\varphi(\cdot,\omega)\in\mathcal{S}_{m}\text{ for almost all }\omega\bigg\}.
	\end{align*}

	Let $\left\{\mathcal{G}^{\epsilon}\right\}_{\epsilon>0}$ be a family of measurable maps from $C([0,T],\mathbb{R}^d)$ to $\mathbb{S}$, and $X^\epsilon=\mathcal{G}^{\epsilon}(\sqrt{\epsilon}W)$. Next we introduce the following conditions which will be sufficient to establish the large and moderate deviation principles for $\left\{X^{\epsilon}\right\}_{\epsilon>0}$.
	
	\begin{assumption}\label{assumld}
		There exists a measurable map $\mathcal{G}^{0}:C([0,T],\mathbb{R}^d)\mapsto \mathbb{S}$ such that
		\begin{enumerate}
			\item[(\textbf{LD}$_{\bm{1}}$)] for any $m\in \mathbb{N}$ and $\left\{v_{n},n\in\mathbb{N}\right\}\subset \mathcal{S}_{m}$ satisfying that  $v_{n}$ converges to some $v\in \mathcal{S}_{m}$ as $n\rightarrow\infty$;
			$$\mathcal{G}^{0}\left(\int_{0}^{\cdot}v_{n}(s)\mathrm{d}s\right)\rightarrow\mathcal{G}^{0}\left(\int_{0}^{\cdot}v(s)\mathrm{d}s\right),$$
			\item[(\textbf{LD}$_{\bm{2}}$)] for any $m\in \mathbb{N}$ and $\left\{h_{\epsilon},\epsilon>0\right\}\subset\mathcal{D}_{m}$ satisfying that $h_{\epsilon}$  converges in distribution to some $h\in\mathcal{D}_{m}$ as $ \epsilon\rightarrow 0 $,
			$$\mathcal{G}^{\epsilon}\left(\sqrt{\epsilon}W+\int_{0}^{\cdot}h_{\epsilon}(s)\mathrm{d}s\right)\rightarrow \mathcal{G}^{0}\left(\int_{0}^{\cdot}h(s)\mathrm{d}s\right)\quad \quad \text{in distribution}.$$
		\end{enumerate}
	\end{assumption}
	
	For each $x\in\mathbb{S}$, define $\mathbb{V}_{x}=\left\{h\in L^{2}([0,T],\mathbb{R}^d):x=\mathcal{G}^{0}\left(\int_{0}^{\cdot}h(s)\mathrm{d}s\right)\right\}$ and
	\begin{equation}\label{rate}
		I(x):=\frac{1}{2}\inf\limits_{h\in\mathbb{V}_{x}}\int_{0}^{T}|h(s)|^{2}\mathrm{d}s,
	\end{equation}
	with the convention $\inf\{\emptyset\}=\infty$.

	We recall the following result due to \cite{BD2000} (see also \cite{BDM2008,BDM2011}).
	\begin{theorem}\label{ldlap}
		Under (\textbf{LD}$_{\bm{1}}$) and (\textbf{LD}$_{\bm{2}}$), $\left\{X^{\epsilon}\right\}_{\epsilon>0}$ satisfies the Laplace principle with  the rate function $I$. More precisely, for each real bounded continuous function $ f $ on $ \mathbb{S} $,
		\begin{equation}
			\lim\limits_{\epsilon\rightarrow 0}\epsilon\, \log\,\mathbb{E}^{\mu}\left(\exp[-f(X^{\epsilon})/\epsilon]\right)=-\inf\limits_{x\in\mathbb{S}}\left\{f(x)+I(x)\right\}.
		\end{equation}
		In particular, the family of $\left\{X^{\epsilon}\right\}_{\epsilon>0}$ satisfies the LDP in $(\mathbb{S},\mathcal{B}(\mathbb{S}))$ with speed $\epsilon$ and rate function $I$.
	\end{theorem}

	%
	%
	%

	
	\section{Main Results}\label{res}

	Consider the following small perturbation of multivalued Mckean-Vlasov SDEs:
	\begin{equation}\label{per}
		\left\{
		\begin{array}{lr}
			X_{t}^{\epsilon} \in x + \int_{0}^{t} b_{\epsilon}\left(X_{s}^{\epsilon},\mathcal{L}_{X_{s}^{\epsilon}} \right)\,\mathrm{d}s + \sqrt{\epsilon}\int_{0}^{t} \sigma_{\epsilon}\left(X_{s}^{\epsilon},\mathcal{L}_{X_{s}^{\epsilon}} \right)\,\mathrm{d}W_{s} - A_\epsilon(X^{\epsilon}_t), \\
			X_{0}^{\epsilon} = x \in\overline{D(A_\epsilon)},\quad \epsilon \in \left(0,1\right)
		\end{array}
		\right.
	\end{equation}
	where $K_{t}^{\epsilon}\in \mathcal{V}$, and $b_{\epsilon}: \overline{D(A_\epsilon)} \times \mathcal{P}_{2} \rightarrow \mathbb{R}^{d}$, $\sigma_{\epsilon}: \overline{D(A_\epsilon)} \times \mathcal{P}_{2} \rightarrow \mathbb{R}^{d}\otimes \mathbb{R}^{d}$ are measurable maps.	
	
	Throughout this paper we denote $\lVert\cdot\rVert$ by the Hilbert-Schmidt norm of matrix in $\mathbb{R}^{d}\otimes \mathbb{R}^{d}$. Let $|\cdot|$  be the Euclidean norm of vector in $\mathbb{R}^d$. For any $\mu\in\mathcal{P}_2(\mathbb{R}^d)$, let $\lVert\mu\rVert_2^2$ be the second moment of $\mu$.
	\begin{assumption}\label{assum}
		Assume that
		
		\textbf{(H0)}\label{h0}
		$ A $ and $A_{\epsilon}$ are multivalued monotone operators with nonempty interior, i.e. Int$ D(A)\neq\emptyset $, Int$ D(A_{\epsilon})\neq\emptyset $, and they have a common domain $D=\overline{D(A)}=\overline{D(A_{\epsilon})}$ with $0\in \text{Int}(D(A))$. Moreover, $A_{\epsilon}$ is locally bounded at 0 uniformly in $\epsilon$, i.e. there exists some $\gamma>0$ such that
		\begin{align*}
			\sup\left\{|y|; y\in A_{\epsilon}(x),\, \epsilon\in[0,1],\, x\in B_{0}(\gamma):=\left\{x\in\mathbb{R}^{d};|x|\leq\gamma\right\}\right\}<\infty.
		\end{align*}
		
		\textbf{(H1)} $b$ and $b_{\epsilon}$ are continuous functions such that for some $L>0$ and all $x,x'\in\mathbb{R}^d$, $\mu,\nu\in \mathcal{P}_2(\mathbb{R}^d)$,
		\begin{align}\label{h1}
			&\langle x-x',b(x,\mu)-b(x',\nu)\rangle\vee\langle x-x',b_\epsilon(x,\mu)-b_\epsilon(x',\nu)\rangle\leq L[\varrho\left(|x-x'|^2\right)+\varrho\left(\mathbb{W}^2_{2}(\mu,\nu)\right)], \\
			&|b(x,\mu)|\vee|b_\epsilon(x,\mu)|\leq L(1+|x|+\lVert \mu\rVert_2),  \label{hblg}
		\end{align}
		where $\varrho:\mathbb{R}^{+}\to \mathbb{R}^{+} $ is a continuous and non-decreasing concave function with $\varrho(0)=0$, $\varrho(u)>0$ for every $u>0$ and $\int_{0^+}1/\varrho(u)du=\infty$.
		And there exist some nonnegative constants $\rho_{b,\epsilon}$ converging to 0 as $\epsilon\rightarrow0$ such that
		$$\sup\limits_{(x,\mu)\in \mathbb{R}^{d}\times\mathcal{P}_{2}}\lvert b_{\epsilon}(x,\mu)-b(x,\mu)\rvert\leq \rho_{b,\epsilon}.$$
		
		\textbf{(H2)}
		$\sigma$ and $\sigma_{\epsilon}$ are continuous functions and satisfy that for some $L>0$ and all $x,x'\in\mathbb{R}^d$, $\mu,\nu\in \mathcal{P}_2(\mathbb{R}^d)$, 	
		\begin{align}\label{h2}
			&\lVert \sigma(x,\mu)-\sigma(x',\nu)\rVert^2\vee\lVert \sigma_\epsilon(x,\mu)-\sigma_\epsilon(x',\nu)\rVert^2
			\leq\, L[\varrho\left(|x-x'|^2\right)+\varrho\left(\mathbb{W}^2_{2}(\mu,\nu)\right)], \\
			&\lVert \sigma(x,\mu)\rVert^2 \vee \lVert \sigma_\epsilon(x,\mu)\rVert^2\leq L(1+|x|^2+\lVert \mu\rVert_2^2), \label{hslg}
		\end{align}
		and there exist some nonnegative constants $\rho_{\sigma,\epsilon}$ converging to 0 as $\epsilon\rightarrow0$ such that
		$$\sup\limits_{(x,\mu)\in \mathbb{R}^{d}\times\mathcal{P}_{2}}\lVert \sigma_{\epsilon}(x,\mu)-\sigma(x,\mu)\rVert\leq \rho_{\sigma,\epsilon}.$$
		
		
		\textbf{(H3)}\label{h3}
		For any given $\epsilon>0$, once if the distributions $\mathcal{L}_{X^{\epsilon}_{\cdot}}$ are fixed with some $\mu_{\cdot}$ in advance, then the pathwise uniqueness holds for (\ref{per}), i.e. for any two solutions $(X^{\epsilon},K^{\epsilon})$ and $(\tilde{X}^{\epsilon},\tilde{K}^{\epsilon})$ of (\ref{per}),
		$$(X^{\epsilon}_t,K^{\epsilon}_t)=(\tilde{X}^{\epsilon}_t,\tilde{K}^{\epsilon}_t),\quad \forall t\in[0,T].$$
		
		\textbf{(H4)}\label{h4}
		For all $\alpha>0$,
		\begin{equation*}
			\lim\limits_{\epsilon\to 0}(1+\alpha A_{\epsilon})^{-1}y=(1+\alpha A)^{-1}y,
		\end{equation*}
		uniformly for $ y $ in compact sets.
		
		
		\begin{remark}
			It is obvious that \textbf{(H4)} is equivalent to
			\begin{equation*}
				\lim\limits_{\epsilon\to 0}A_\epsilon^\alpha(y)=A^\alpha(y),
			\end{equation*}
		\end{remark}
		uniformly for $ y $ in compacts sets, where $A^\alpha$(resp.$\,A_\epsilon^\alpha$) is the Yosida approximation of $A\,$(resp.$\,A_\epsilon$). This condition is proposed to measure the difference between multivalued operators $A$ and $A_\epsilon$.
		
		\begin{remark}
			It is noted that in \textbf{(H1)} and \textbf{(H2)}, the concave functions for the state variable
			$x$ and the distribution variable $\mu$ need not be the same; this does not affect the conclusions in the following estimates. For convenience, we assume that the concave functions for both variables are identical.
		\end{remark}

		\begin{examples}
			The following are some concrete examples of the function $\varrho$:
			\begin{flalign*}
				&\varrho_1(u)=
				\begin{cases}
					u\log(u^{-1}), &0\leq u\leq \eta;  \\
					\eta\log(\eta^{-1})+\dot{\varrho_2}(\eta-)(u-\eta),  &u>\eta;
				\end{cases}&  \\
				&\varrho_2(u)=
				\begin{cases}
					u\log(u^{-1})\log\log(u^{-1}), &0\leq u\leq \eta;  \\
					\eta\log(\eta^{-1})\log\log(\eta^{-1})+\dot{\varrho_3}(\eta-)(u-\eta),  &u>\eta
				\end{cases}&
			\end{flalign*}
			where $L>0$, $\eta$ is some sufficiently small positive constant and $\dot{\varrho}(\eta-)$ denotes the left derivative of $\varrho$ at $\eta$.
		\end{examples}

	\end{assumption}

	Next we introduce the following Bihari’s inequality from \cite{mao}, which will be frequently used to deal with the assumptions (\hyperref[h1]{\textbf{H1}}) and (\hyperref[h2]{\textbf{H2}}).
	

	\begin{lemma}\label{bhr}
		Let $T>0$ and $c>0$. Let $\varrho:\mathbb{R}^+\mapsto\mathbb{R}^+$ be a continuous and nondecreasing function such that $\varrho(t)>0$ for all $t>0$. Let $u(\cdot)$ be a Borel measurable bounded nonnegative function on $[0,T]$ and  $v(\cdot)$ a nonnegative integrable function on $[0,T]$. If
		\begin{align*}
			u(t)\leq c+\int_{0}^{t}v(s)\varrho(u(s))\mathrm{d}s, \quad \text{for all }\,0\leq t\leq T,
		\end{align*}
		then
		\begin{align*}
			u(t)\leq f^{-1}\left(f(c)+\int_{0}^{t}v(s)\mathrm{d}s\right)
		\end{align*}
		holds for all $t\in[0,T]$ with
		\begin{align*}
			f(c)+\int_{0}^{t}v(s)\mathrm{d}s\in Dom(f^{-1}),
		\end{align*}
		where $f(r)=\int_{1}^{r}1/\varrho(s)\mathrm{d}s$ on $r>0$ and $f^{-1}$ is the inverse function of $f$.
	\end{lemma}

	The existence and uniqueness of the strong solution for equation (\ref{mve}) under one-sided Lipschitz form condition were obtained in \cite{QG}. Under non-Lipschitz conditions, the results for McKean-Vlasov SDEs are  investigated in \cite{DQ}. Now we state the strong solution result for equation (\ref{mve}).
	
	\begin{theorem}\label{strongsolu}
		Assume that (\hyperref[h0]{\textbf{H0}})-(\hyperref[h2]{\textbf{H2}}) hold. If $X_0$ be a $\mathcal{F}_0$-measurable random variable with finite second moment, then there exists a unique strong solution $(X,K)\in\mathcal{A}$ as defined in the Definition \ref{strong} to equation (\ref{mve}).
	\end{theorem}
	
	The proof of this theorem will be presented in the next section. By Theorem \ref{strongsolu} we can easily obtain the following result by taking the diffusion term as zero.
	\begin{proposition}
		Assume that (\hyperref[h0]{\textbf{H0}}) and (\hyperref[h1]{\textbf{H1}}) holds, then there exists a unique process $\left\{(X^{0}_t,K^0_t),t\in[0,T]\right\}$ such that
		\begin{enumerate}
			\item[(1)] $X^{0}\in C\big([0,T],\overline{D(A)}\big)$,
			\item[(2)] $\int_{0}^{T}\lvert b(X_{s}^{0},\mathcal{L}_{X_{s}^{0}})\rvert\mathrm{d}s<\infty$ where $\mathcal{L}_{X_{s}^{0}}$ is just the Dirac measure $\delta_{{X}_{s}^{0}}$ and $K_{t}^{0}\in \mathcal{V},\, \forall t\in[0,T]$,
			\item[(3)] $X^{0}$ satisfies
			\begin{equation}\label{x0}
				X_{t}^{0}=x+\int_{0}^{t}b(X_{s}^{0},\mathcal{L}_{X_{s}^{0}})\mathrm{d}s-K_{t}^{0},\quad\forall t\in[0,T],
			\end{equation}
		\end{enumerate}
		where $(X^{0}_t,K^0_t)$ is called the solution to (\ref{x0}).
	\end{proposition}

	
	To ensure the application of Girsanov Theorem, we need the following lemma which is taken from \cite[Lemma 2.21]{liuw}.
	\begin{lemma}\label{girs}
		Let $X_{t}^{\epsilon}$ be a solution of (\ref{per}) with initial value $X_{0}^{\epsilon}=x\in\overline{D(A)}$. Assume (\hyperref[h3]{\textbf{H3}})  and the following assumptions hold:
		\begin{enumerate}
			\label{a0}\item[\textbf{(A0)}] for any $\mathcal{L}_{X^{\epsilon}}$ fixed, the map $b_{\epsilon}(\cdot,\mathcal{L}_{X^{\epsilon}}):\mathbb{R}^{d}\rightarrow\mathbb{R}^{d}$ is $\mathcal{B}(\mathbb{R}^{d})/\mathcal{B}({\mathbb{R}^{d}})$-measurable and  $\sigma_{\epsilon}(\cdot,\mathcal{L}_{X^{\epsilon}}): \mathbb{R}^{d}\rightarrow\mathbb{R}^{d}\otimes\mathbb{R}^{d}$ is $\mathcal{B}(\mathbb{R}^{d})/\mathcal{B}({\mathbb{R}^{d}\otimes \mathbb{R}^{d}})$-measurable,
			\label{a1}\item[\textbf{(A1)}] (\ref{per}) has a unique solution $X^{\epsilon}$ as stated in Definition \ref{strong}.
		\end{enumerate}
		Then there exists a map $\mathcal{G}_{\mathcal{L}_{X^{\epsilon}}}^{\epsilon}$ such that $X^{\epsilon}=\mathcal{G}_{\mathcal{L}_{X^{\epsilon}}}^{\epsilon}(\sqrt{\epsilon}W)$. Furthermore, for any $m\in(0,\infty)$, $h_{\epsilon}\in\mathcal{D}^{T}_m$, let
		\begin{equation}\label{zep}	Z^{h_{\epsilon}}:=\mathcal{G}_{\mathcal{L}_{X^{\epsilon}}}^{\epsilon}\big(\sqrt{\epsilon}W+\int_{0}^{\cdot}h_{\epsilon}(s)\mathrm{d}s\big),
		\end{equation}
		then $Z^{h_{\epsilon}}$ is the unique stochastic process satisfying that
		\begin{enumerate}
			\item[(1)] $Z^{h_{\epsilon}}$ is $\mathcal{F}_{t}$-adapted,
			\item[(2)]
			\begin{align*}
				\int_{0}^{T} \lvert b_\epsilon\left(Z_{s}^{h_{\epsilon}},\mathcal{L}_{X_{s}^{\epsilon}}\right) \rvert \,\mathrm{d}s + \int_{0}^{T} \lVert \sigma_\epsilon\left(Z_{s}^{h_{\epsilon}},\mathcal{L}_{X_{s}^{\epsilon}}\right) \rVert^{2} \,\mathrm{d}s + \int_{0}^{T} \lvert \sigma_\epsilon\left(Z_{s}^{h_{\epsilon}},\mathcal{L}_{X_{s}^{\epsilon}}\right)h_{\epsilon}\left(s\right) \rvert \,\mathrm{d}s +| K^{h_{\epsilon}} |_{0}^{T} <\infty,
			\end{align*}
			\item[(3)]
			\begin{equation}
				\begin{split}
					Z_{t}^{h_{\epsilon}} = x + \int_{0}^{t} b_{\epsilon}\left(Z_{s}^{h_{\epsilon}},\mathcal{L}_{X_{s}^{\epsilon}} \right)\,\mathrm{d}s + \sqrt{\epsilon}\int_{0}^{t} \sigma_{\epsilon}\left(Z_{s}^{h_{\epsilon}},\mathcal{L}_{X_{s}^{\epsilon}} \right)\,\mathrm{d}W_{s} + \int_{0}^{t} \sigma_{\epsilon}\left(Z_{s}^{h_{\epsilon}},\mathcal{L}_{X_{s}^{\epsilon}} \right)h_{\epsilon}(s)\,\mathrm{d}s - K_{t}^{h_{\epsilon}}.
				\end{split}
			\end{equation}
		\end{enumerate}
	\end{lemma}

	\subsection{Large deviation principle}
	
	\begin{proposition}\label{Yu}Assume that (\hyperref[h0]{\textbf{H0}}), (\hyperref[h1]{\textbf{H1}}) and (\hyperref[h2]{\textbf{H2}}) hold. Then for any $ h\in \mathcal{S}_{m} $, there exists a unique process $ (Y^{h},K^h)=\left\{ (Y_{t}^{h},K_t^h), t\in\left[ 0, T \right] \right\} $ satisfying that
		\begin{enumerate}
			\item[(1)] $Y^{h}\in C\big( \left[ 0, T \right], \overline{D(A)} \big)$,
			\item[(2)] $$\int_{0}^{T} \left| b\left(Y_{s}^{h},\mathcal{L}_{X_{s}^{0}}\right) \right| \,\mathrm{d}s + \int_{0}^{T} \lvert \sigma\left(Y_{s}^{h},\mathcal{L}_{X_{s}^{0}}\right)h\left(s\right) \rvert \,\mathrm{d}s + \left| K^{h} \right|_{0}^{T} <\infty,$$
			\item[(3)]
			\begin{equation}\label{yy}
				Y_{t}^{h}=x+\int_{0}^{t}b\left( Y_{s}^{h},\mathcal{L}_{X_{s}^{0}}\right) \,\mathrm{d}s + \int_{0}^{t}\sigma\left( Y_{s}^{h},\mathcal{L}_{X_{s}^{0}}\right)h\left(s\right) \,\mathrm{d}s-K_{t}^{h},\quad t\in \left[0,T\right].
			\end{equation}
		\end{enumerate}
		
		Moreover, for any $m>0$,
		$$ \sup\limits_{h\in \mathcal{S}_{m}}\sup\limits_{t\in \left[0,T\right]}\left| Y_{t}^{h}\right| <\infty.$$
		
	\end{proposition}
	\begin{proof} This result can be proved by using Proposition \ref{multi} presented in Section 2 and Proposition 3.7 in \cite{liuw}. We omit the tedious proofs here.\end{proof}

	Next we present our main result about the large deviation principle.
	\begin{theorem}\label{thmldp}
		Assume that (\hyperref[h0]{\textbf{H0}})-(\hyperref[h4]{\textbf{H4}}) hold, then the family of processes $ \left\{X_{t}^{\epsilon}, \epsilon >0, t\in[0,T]\right\} $ satisfies the LDP on $C\big(\left[0,T\right],\overline{D(A)}\big)$ with speed $\epsilon$ and the good rate function $I$ given by
		\begin{equation}
			I\left( g\right): = \frac{1}{2}\inf\limits_{\left\{h\in L^{2}([0,T],\mathbb{R}^d): g=Y^{h}\right\}}\int_{0}^{T}\lvert h(s)\rvert^2\,\mathrm{d}s,\quad g\in C\left(\left[0,T\right],\overline{D(A)}\right),
		\end{equation}
		where $Y^{h}$ is the solution to equation (\ref{yy}).
	\end{theorem}
	
	\begin{proof}
		By Proposition \ref{Yu} we can define a map $$ \mathcal{G}^{0}: \mathcal{S}_m\ni h\mapsto Y^{h} \in C\left(\left[0,T\right],\overline{D(A)}\right),$$ where $Y^{h}$ is the unique solution of eqution (\ref{yy}).
		
		For any $\epsilon >0$, $m\in \left(0,\infty\right)$ and $h_{\epsilon}\in \mathcal{D}_{m}$, by Lemma \ref{girs} there exists a unique solution $ \left\{Z_{t}^{\epsilon,h_{\epsilon}} ,t\in \left[0,T\right]\right\} $ to the following equation
		\begin{equation}\label{LDP}
			\mathrm{d}Z_{t}^{\epsilon,h_{\epsilon}}\in b_{\epsilon}\left(Z_{t}^{\epsilon,h_{\epsilon}},\mathcal{L}_{X_{t}^{\epsilon}}\right)\mathrm{d}t+ \sqrt{\epsilon}\sigma_{\epsilon}\left(Z_{t}^{\epsilon,h_\epsilon},\mathcal{L}_{X_{t}^{\epsilon}}\right)\mathrm{d}W_{t} + \sigma_{\epsilon}\left(Z_{t}^{\epsilon,h_\epsilon},\mathcal{L}_{X_{t}^{\epsilon}}\right)h_{\epsilon}\left(t\right)\mathrm{d}t-A_\epsilon\left(Z_{t}^{\epsilon,h_\epsilon}\right)\mathrm{d}t,
		\end{equation}
		where $Z_{0}^{\epsilon,h_{\epsilon}}=x$ is the initial value and $X^{\epsilon}$ is the solution to (\ref{per}).
		
		To establish the LDP by using Theorem \ref{ldlap}, it is sufficient to verify the following two conditions.

		\textbf{(LDP)$\bm{_{1}}$} For any given $m\in\left(0,\infty\right)$, if $\left\{h_{n}, n\in \mathbb{N}\right\}\subset \mathcal{S}_{m}$ weakly converges to $h\in \mathcal{S}_{m}$ as $n\to \infty$, then
		$$
		\lim\limits_{n\to \infty}\sup\limits_{t\in \left[0,T\right]}\lvert \mathcal{G}^{0}\left(h_{n}\right)\left(t\right)-\mathcal{G}^{0}\left(h\right)\left(t\right) \rvert=0;
		$$
		
		\textbf{(LDP)$\bm{_{2}}$} For any given $m\in\left(0,\infty\right)$, let $\left\{h_{\epsilon}, \epsilon>0\right\}\subset \mathcal{D}_{m}$, then
		$$
		\lim\limits_{\epsilon\to 0}\mathbb{E}\left(\sup\limits_{t\in \left[0,T\right]}\lvert Z^{\epsilon,h_\epsilon}_t-\mathcal{G}^{0}\left(h_{\epsilon}\right)\left(t\right)\rvert^{2}\right)=0.
		$$
		
		The verification of \textbf{(LDP)$\bm{_{1}}$} and \textbf{(LDP)$\bm{_{2}}$} will be given later in Section 4.2. Note that \textbf{(LDP)$\bm{_{2}}$} is stronger than (\textbf{LD}$_{\bm{2}}$) stated in Assumption \ref{assumld}.
	\end{proof}
	
	\subsection{Moderate deviation principle}
	
	Lemma \ref{girs} can also be used to establish the moderate deviation principle of $X^\epsilon$ as $\epsilon$ decreases to 0. Assume that $ \lambda(\epsilon)>0$ satisfies
	\begin{align}\label{lame}
		\lambda(\epsilon)\to 0, \quad \frac{\epsilon}{\lambda^2(\epsilon)}\to 0, \quad \text{as}\,\,\epsilon\to 0.
	\end{align}
	
	Define
	\begin{align*}
		M^{\epsilon}_t:=\frac{1}{\lambda(\epsilon)}(\bar{X}_t^\epsilon-X_t^0), \quad t\in[0,T],
	\end{align*}
	where $X^0_t$ solves equation (\ref{x0}), i.e.
	\begin{align*}
		\begin{cases}
			\,\,\mathrm{d}X_t^0\in\,\,b(X^0_t,\mathcal{L}_{X^0_t})\mathrm{d}t-A(X^0_t)\mathrm{d}t, \\
			\,\, X^0_0=x.
		\end{cases}
	\end{align*}
	Consider the following multivalued SDE
	\begin{align}\label{Mep}
		\begin{cases}
			\,\mathrm{d}M^{\epsilon}_t\in\,\frac{1}{\lambda(\epsilon)}\left(b_{\epsilon}(\bar{X}^\epsilon_t,\mathcal{L}_{\bar{X}_t^{\epsilon}})-b(X_t^0,\mathcal{L}_{X_t^0})\right)\mathrm{d}t+\frac{\sqrt{\epsilon}}{\lambda(\epsilon)}\sigma_{\epsilon}(\bar{X}_t^\epsilon,\mathcal{L}_{\bar{X}_t^{\epsilon}})\mathrm{d}W_t-A_\epsilon(M^{\epsilon}_t)\mathrm{d}t, \\
			\,
			M^{\epsilon}_0=0.
		\end{cases}
	\end{align}
	By Theorem \ref{strongsolu},  $(M^{\epsilon}_t,\bar{K}^\epsilon_t)$ is the unique solution to the following equation
	\begin{align}\label{mep}
		\mathrm{d}M^{\epsilon}_t=&\frac{1}{\lambda(\epsilon)}\left(b_{\epsilon}(\lambda(\epsilon)M^{\epsilon}_t+X^0_t,\mathcal{L}_{\bar{X}_t^{\epsilon}})-b(X_t^0,\mathcal{L}_{X_t^0})\right)\mathrm{d}t   \nonumber \\
		&+\frac{\sqrt{\epsilon}}{\lambda(\epsilon)}\sigma_{\epsilon}(\lambda(\epsilon)M^{\epsilon}_t+X_t^0,\mathcal{L}_{\bar{X}_t^{\epsilon}})\mathrm{d}W_t-\mathrm{d}\bar{K}^\epsilon_t.
	\end{align}
	
	
	To show the moderate deviation principle for $\bar{X}^\epsilon$, it is equivalent to prove that $M^\epsilon$ satisfies the LDP. Let $\nabla b(x,\mu)$ denote the gradient of $b(x,\mu)$ with respect to the space variable $x$.
	
	We make the following assumptions.
	
	\textbf{(B0)} There exist $L',q'\geq 0$ such that for all $x_1,x_2\in\mathbb{R}^{d}$, and $\mu_1,\mu_2\in\mathcal{P}_2(\mathbb{R}^d)$,
	\begin{align*}\label{b0}
		&\lVert\nabla b(x_1,\mu_1)-\nabla b(x_2,\mu_1)\rVert\leq L'(1+|x_1|^{q'}+|x_2|^{q'})|x_1-x_2|, \\
		&\langle x_1-x_2, b(x_1,\mu_1)-b(x_2,\mu_1) \rangle \vee \langle x_1-x_2, b_\epsilon(x_1,\mu_1)-b_\epsilon(x_2,\mu_1) \rangle\leq L'|x_1-x_2|^2, \\
		&|b(x_1,\mu_1)-b(x_1,\mu_2)|\leq L'\mathbb{W}_2(\mu_1,\mu_2), \\
		&|b(x_1,\mu_1)|\vee|b_\epsilon(x_1,\mu_1)|\leq L'(1+|x_1|+\lVert \mu_1\rVert_2).
	\end{align*}
	
	\textbf{(B1)} There exists $L_{\nabla b}>0$ such that
	\begin{align*}\label{b1}
		\int_{0}^{T}\lVert\nabla b(X^0_t,\mathcal{L}_{X^0_t})\rVert\mathrm{d}t<L_{\nabla b}.
	\end{align*}
	
	\textbf{(B2)} For $\rho_{b,\epsilon}$ given in (\hyperref[h2]{\textbf{H2}}) and $\lambda(\epsilon)$ satisfying (\ref{lame}), it holds that
	\begin{equation*}\label{b2}
		\begin{aligned}
			\lim\limits_{\epsilon\to 0}\frac{\rho_{b,\epsilon}}{\lambda(\epsilon)}=0.
		\end{aligned}
	\end{equation*}
	
	\textbf{(B3)}  There exists $L'>0$, such that for all $x_1,x_2\in\mathbb{R}^{d}$, and $\mu_1,\mu_2\in\mathcal{P}_2(\mathbb{R}^d)$,
	\begin{align*}\label{b3}
		& \lVert \sigma(x_1,\mu_1)-\sigma(x_2,\mu_2)\rVert^2\vee\lVert \sigma_\epsilon(x_1,\mu_1)-\sigma_\epsilon(x_2,\mu_2)\rVert^2\leq L'(|x_1-x_2|^2+\mathbb{W}^2_2(\mu_1,\mu_2)). \\
		&\lVert \sigma(x_1,\mu_1)\rVert^2 \vee \lVert \sigma_\epsilon(x_1,\mu_1)\rVert^2\leq L'(1+|x_1|^2+\lVert \mu_1\rVert_2^2).
	\end{align*}



	\begin{proposition}\label{nu}
		Assume that (\hyperref[h1]{\textbf{H1}}), (\hyperref[h2]{\textbf{H2}}), (\hyperref[b0]{\textbf{B0}}) and (\hyperref[b1]{\textbf{B1}}) hold. Then for any fixed $m\in(0,\infty)$ and $\psi\in \mathcal{S}_m$, there is a unique solution $(\nu^\psi,\hat{K}^\psi)=\big\{\big(\nu^{\psi}_t,\hat{K}^\psi_t\big), t\in[0,T]\big\}$ to the following equation:
		\begin{equation}\label{mdp1}
			\left\{ \begin{aligned}
				&\mathrm{d}\nu^{\psi}_t=\nabla b(X^0_t,\mathcal{L}_{X^0_t})\nu^{\psi}_t\mathrm{d}t+\sigma(X_t^0,\mathcal{L}_{X_t^0})\psi(t)\mathrm{d}t-\mathrm{d}\hat{K}^\psi_t,    \\
				&\nu^{\psi}_0=0.
			\end{aligned}  \right.
		\end{equation}
		Moreover,
		\begin{align}\label{nulim}
			\sup_{\psi\in \mathcal{S}^{m}}\sup_{t\in[0,T]}|\nu^{\psi}_t|<\infty.
		\end{align}
		
		\begin{proof}
			Since $\big\{\hat{K}^\psi_t, t\in[0,T]\big\}$ is of finite variation with $\psi\in\mathcal{S}_m$, we have
			$|\hat{K}^\psi|_0^T<\infty.$
			
			By (\hyperref[h1]{\textbf{H1}}), (\hyperref[h2]{\textbf{H2}}), Remark \ref{W2} and the fact that $X^0\in C([0,T],\mathbb{R}^d)$ and $\psi\in\mathcal{S}_m$, we have
			\begin{align}\label{slim}
				\int_{0}^{T}\lVert \sigma(X_s^0,\mathcal{L}_{X_s^0})\rVert^2\mathrm{d}s<\infty.
			\end{align}
			Thus
			\begin{align}\label{si}
				\int_{0}^{T}\lvert\sigma(X_t^0,\mathcal{L}_{X_t^0})\psi(t)\rvert\mathrm{d}t  \nonumber
				\leq &\left(\int_{0}^{T}\lVert\sigma(X_t^0,\mathcal{L}_{X_t^0})\rVert^2\mathrm{d}t\right)^{\frac{1}{2}}\left(\int_{0}^{T}|\psi(t)|^2\mathrm{d}t\right)^\frac{1}{2}   \nonumber\\
				\leq &\left(\int_{0}^{T}\lVert\sigma(X_t^0,\mathcal{L}_{X_t^0})\rVert^2\mathrm{d}t\right)^{\frac{1}{2}}(2m)^{\frac{1}{2}}  \nonumber
				<\infty.
			\end{align}
			
			Due to (\hyperref[b1]{\textbf{B1}}) and the estimates above we can easily prove that the equation (\ref{mdp1}) has a unique solution $\big\{\big(\nu^{\psi}_t,\hat{K}^\psi_t\big), t\in[0,T]\big\}$.
			By Gronwall's inequality we have
			\begin{align}
				\nu^{\psi}_t\leq  e^{\int_{0}^{t}\nabla b(X_s^0,\mathcal{L}_{X_s^0})\mathrm{d}s}\int_{0}^{t}\sigma(X_s^0,\mathcal{L}_{X_s^0})\psi(s)\mathrm{d}s,
			\end{align}
			which implies (\ref{nulim}).
			
		\end{proof}
		
	\end{proposition}

	\begin{theorem}
		Assume that (\hyperref[h0]{\textbf{H0}}), (\hyperref[h4]{\textbf{H4}}), (\hyperref[b0]{\textbf{B0}}), (\hyperref[b1]{\textbf{B1}}) and (\hyperref[b2]{\textbf{B2}}) hold, then $\left\{M^{\epsilon},\epsilon>0\right\}$ satisfies the LDP on $C([0,T],\overline{D(A)})$ with speed ${\epsilon}/{\lambda^2(\epsilon)}$ and rate function $I$ given by
		\begin{align}
			I(g):=\frac{1}{2}\inf_{\left\{\psi\in L^2([0,T],\mathbb{R}^d), \nu^{\psi}=g\right\}}\int_{0}^{T}|\psi(s)|^2\mathrm{d}s, \quad g\in C\left([0,T],\overline{D(A)}\right)
		\end{align}
		where  $\big(\nu^{\psi},\hat{K}^\psi\big)$ is the unique solution of (\ref{mdp1}) with $\psi\in L^2([0,T],\mathbb{R}^d)$.
	\end{theorem}
	
	\begin{proof}
		By proposition \ref{nu}, we can define a map
		\begin{align}
			\Gamma^0:L^2([0,T],\mathbb{R}^d)\ni \psi \mapsto \nu^\psi\in C([0,T],\overline{D(A)})
		\end{align}
		such that $\nu^{\psi}_.=\Gamma^{0}(\psi)(\cdot)$, where $\nu^\psi$ is the unique solution of (\ref{nu}). Let
		\begin{align*}	\Gamma^\epsilon_{\mathcal{L}_{X^0}}(\cdot):=\frac{1}{\lambda(\epsilon)}\left(\mathcal{G}^\epsilon_{\mathcal{L}_{X^0}}(\cdot)-X^0\right),
		\end{align*}
		then we have
		\begin{enumerate}
			\item[(a)] $\Gamma^\epsilon_{\mathcal{L}_{X^0}}$ is a measurable map from $C([0,T],\mathbb{R}^d)\mapsto C([0,T],\overline{D(A)}) $ such that
			\begin{align*}
				M^\epsilon=\Gamma^\epsilon_{\mathcal{L}_{X^0}}\left(\sqrt{\epsilon}W_\cdot\right),
			\end{align*}
			\item[(b)] for any $m\in(0,\infty)$, $\psi_\epsilon\in\mathcal{D}_m^T$, let
			\begin{align}\label{mpsi}
				M^{\epsilon,\psi_\epsilon}:=\Gamma^\epsilon_{\mathcal{L}_{X^0}}\left(\sqrt{\epsilon}W_\cdot+\lambda(\epsilon)\int_{0}^{\cdot}\psi_\epsilon(s)\mathrm{d}s\right).
			\end{align}
		\end{enumerate}
		
		For any $\epsilon>0$, by Proposition \ref{strongsolu} and Girsanov Theorem, it holds that   $\left\{\left(M^{\epsilon,\psi_\epsilon}_t,\hat{K}^{\epsilon,\psi_\epsilon}_t\right),t\in[0,T]\right\}$ is the unique solution to the following controlled multivalued McKean-Vlasov  SDE:
		\begin{equation}\label{mdp2m}
			\left\{\begin{aligned}
				\mathrm{d}M^{\epsilon,\psi_\epsilon}_t\in&\frac{1}{\lambda(\epsilon)}\left(b_\epsilon(\lambda(\epsilon)M^{\epsilon,\psi_{\epsilon}}_t+X_t^0,\mathcal{L}_{X_t^\epsilon})-b(X_t^0,\mathcal{L}_{X_t^0})\right)\mathrm{d}t  \\
				&+\frac{\sqrt{\epsilon}}{\lambda(\epsilon)}\sigma_{\epsilon}(\lambda(\epsilon)M^{\epsilon,\psi_{\epsilon}}+X^0_t,\mathcal{L}_{X_t^\epsilon})\mathrm{d}W_t  \\
				&+\sigma_{\epsilon}(\lambda(\epsilon)M^{\epsilon,\psi_{\epsilon}}+X^0_t,\mathcal{L}_{X_t^\epsilon})\psi_{\epsilon}(t)\mathrm{d}t-A_\epsilon(M^{\epsilon,\psi_\epsilon}_t)\mathrm{d}t,  \\
				\,\,\,\, M^{\epsilon,\psi_{\epsilon}}_0=&\,\,0.
			\end{aligned}\right.
		\end{equation}

		By Theorem \ref{ldlap}, it is sufficient to verify the following two claims.
		
		\noindent\textbf{(MDP)}$\bm{_{1}}$ For any given $m\in(0,\infty)$, let $\left\{\psi_n, n\in\mathbb{N}\right\}, \psi\in\mathcal{S}_m$ be such that $\psi_n\to\psi$ in $\mathcal{S}_m$ as $n\to\infty$, then
		$$\lim\limits_{n\to\infty}\sup_{t\in[0,T]}|\Gamma^0(\psi_n)(t)-\Gamma^0(\psi)(t)|=0.$$
		
		\noindent\textbf{(MDP)}$\bm{_{2}}$ For any given $m\in(0,\infty)$, let $\left\{\psi_{\epsilon},\epsilon>0\right\}\in\mathcal{D}_m$, then for any $\xi>0$,
		$$\lim\limits_{\epsilon\to0}P\big(\sup_{t\in[0,T]}|M^{\epsilon,\psi_{\epsilon}}_t-\Gamma^0(\psi_{\epsilon})(t)|>\xi\big)=0.$$
		
		
		The verifications of \textbf{(MDP)}$\bm{_{1}}$ and \textbf{(MDP)}$\bm{_{2}}$ will be given in Section 4.4.

	\end{proof}

	\section{Proof of main results}
	
	\subsection{Proof of Theorem 3.4}
	
	First, we introduce the relationship between martingale solution and weak solution.
	
	\begin{proposition}\label{vv}
		The existence of martingale solutions implies the existence of weak solutions and vice versa.
	\end{proposition}
	
	Since the Proposition \ref{vv} above can be proved by the same way in \cite[Proposition 3.5]{Zalinescu}, we omit the detailed proofs here for brevity. To study the existence of martingale solution to equation (\ref{mve}), we first present the following lemmas. The first one is taken from \citep[Lemma 4.1]{CHM}.
	\begin{lemma}\label{kfcon}
		Let $\{\mathbb{P}_n\in\mathcal{P}(\mathcal{W})\}_{n\geq 1}$ be a sequence of probability measures converging weakly to $\mathbb{P}_0\in\mathcal{P}(\mathcal{W})$. Assume that $\mathbb{P}_n(\mathcal{A})=1$ and $\sup_{n\geq 1}\mathbb{E}^{\mathbb{P}_n}|K|^2_{TV}<\infty$, then $\mathbb{P}_0(\mathcal{A})=1$ and $\mathbb{E}^{\mathbb{P}_0}|K|^2_{TV}<\infty$. Moreover, for any bounded continuous functional $G$ on $\mathcal{W}$ and $f\in C^1_b(\mathbb{R}^d)$,
		\begin{align*}
			\lim\limits_{n\to\infty}\mathbb{E}^{\mathbb{P}_n}\left(G\int_{0}^{t}\langle\nabla f(X_s),\mathrm{d}K_s\rangle\right)=\mathbb{E}^{\mathbb{P}_0}\left(G\int_{0}^{t}\langle\nabla f(X_s),\mathrm{d}K_s\rangle\right).
		\end{align*}
	\end{lemma}
	
	 Denote $\tilde{\mathbb{E}}(\cdot):=\mathbb{E}^{\tilde{P}}(\cdot)$ the expectation under the probability measure $\tilde{P}$ introduced in the Definition \ref{solu}. We now state the second lemma.
	
	\begin{lemma}\label{l42}
		Assume that $b(x,\mu)$ and $\sigma(x,\mu)$ satisfy (\hyperref[h0]{\textbf{H0}}), (\ref{hblg}) and (\ref{hslg}). If $(\tilde{\mathcal{ S }};\tilde{W},(\tilde{X},\tilde{K}))$ is a weak solution to the equation (\ref{mve}), $\mathcal{L}_{\tilde{X}_0}\in\mathcal{P}({\overline{D(A)}})$ and $\tilde{\mathbb{E}}|\tilde{X}_0|^{2p}<\infty$ for any $p>1$, then it follows that
		\begin{align}\label{l421}
			\tilde{\mathbb{E}}(\sup_{t\in[0,T]}|\tilde{X}_t|^{2p}+(|\tilde{K}|_0^T)^{p})<\infty
		\end{align}
and
		\begin{align}\label{l422}
			\tilde{\mathbb{E}}(|\tilde{X}_t-\tilde{X}_s|^{2p}+(|\tilde{K}|_s^t)^{p})\leq C(t-s)^p,\quad 0\leq s<t\leq T,
		\end{align}
		where $C>0$ is some constant depending on $T,p,L$.
	\end{lemma}
	
	\begin{proof}
		By using the truncation approach to $\tilde{X}_t$ and applying Fatou's Lemma, we may assume that $\tilde{X}_t$ is bounded without loss of generality.
		
		For the equation (\ref{mve}), since $0\in \text{Int}D(A)$, by Proposition \ref{multi} and It\^o's fomula we have
		\begin{align*}
			|\tilde{X}_t|^{2}=&\,|\tilde{X}_0|^{2}-2\int_{0}^{t}\langle \tilde{X}_s, \mathrm{d}\tilde{K}_s\rangle + 2\int_{0}^{t}\langle \tilde{X}_s, b(\tilde{X}_s,\mathcal{L}_{\tilde{X}_s})\rangle\mathrm{d}s \\
			&\,+ 2\int_{0}^{t}\langle \tilde{X}_s, \sigma(\tilde{X}_s,\mathcal{L}_{\tilde{X}_s})\mathrm{d}\tilde{W}_s\rangle + \int_{0}^{t} \lVert \sigma(\tilde{X}_s,\mathcal{L}_{\tilde{X}_s}) \rVert^2\mathrm{d}s \\
			\leq &\, |\tilde{X}_0|^{2}-2\lambda_{1}|\tilde{K}|_0^t + C\int_{0}^{t}\left(1+|\tilde{X}_s|^2+\mathbb{E}|\tilde{X}_s|^2\right)\mathrm{d}s + 2\int_{0}^{t}\langle \tilde{X}_s, \sigma(\tilde{X}_s,\mathcal{L}_{\tilde{X}_s})\mathrm{d}\tilde{W}_s\rangle.
		\end{align*}
		By taking expectations on the both sides of the inequality above, we get
		\begin{align*}
			\tilde{\mathbb{E}}|\tilde{X}_t|^{2}\leq\, \tilde{\mathbb{E}}|\tilde{X}_0|^{2}+C\int_{0}^{t}\left(1+\tilde{\mathbb{E}}|\tilde{X}_s|^2\right)\mathrm{d}s,
		\end{align*}
		which implies $\sup_{t\in[0,T]}\tilde{\mathbb{E}}|\tilde{X}_t|^{2}<\infty$ by Gronwall's inequality.
		
		By BDG inequality and Young's inequality, for any $t\in[0,T]$, it holds that
		\begin{align*}
			&\tilde{\mathbb{E}}\left(\sup_{s\in[0,t]}|\tilde{X}_s|^{2p}+(|\tilde{K}|_0^t)^p\right)\\
\leq &\, C\tilde{\mathbb{E}}|\tilde{X}_0|^{2p}+C\tilde{\mathbb{E}}\left(\int_{0}^{t}\left(1+|\tilde{X}_s|^2+\tilde{\mathbb{E}}|\tilde{X}_s|^2\right)\mathrm{d}s\right)^p  +C\tilde{\mathbb{E}}\left(\int_{0}^{t}\langle \tilde{X}_s, \sigma(\tilde{X}_s,\mathcal{L}_{\tilde{X}_s})\mathrm{d}\tilde{W}_s\rangle\right)^p \\
			\leq&\, C\tilde{\mathbb{E}}|\tilde{X}_0|^{2p}+C+Ct^{p-1}\int_{0}^{t}\tilde{\mathbb{E}}|\tilde{X}_s|^{2p}\mathrm{d}s+C\tilde{\mathbb{E}}\left(\int_{0}^{t}|\tilde{X}_s|^2\lVert\sigma(\tilde{X}_s,\mathcal{L}_{\tilde{X}_s}) \rVert^2\mathrm{d}s\right)^{\frac{p}{2}} \\
			\leq&\, C+Ct^p+\frac{1}{2}\tilde{\mathbb{E}}\sup_{s\in[0,t]}|\tilde{X}_s|^{2p}+Ct^{p-1}\int_{0}^{t}\tilde{\mathbb{E}}|\tilde{X}_s|^{2p}\mathrm{d}s,
		\end{align*}
		where $C$ depends on $T,p,L$. This implies (\ref{l421}) by directly using Gronwall's inequality.
		
		
		By the same deduction, for any $t\in[0,T]$ we get
		\begin{align*}
			&\tilde{\mathbb{E}}|\tilde{X}_t-\tilde{X}_s|^{2p}+\tilde{\mathbb{E}}(|\tilde{K}|_s^t)^{p}\\
\leq&\, C\tilde{\mathbb{E}}\left(\int_{s}^{t}\left(1+|\tilde{X}_r|^2+\mathbb{E}|\tilde{X}_r|^2\right)\mathrm{d}r\right)^p+C\tilde{\mathbb{E}}\left(\int_{s}^{t}\langle \tilde{X}_r, \sigma(\tilde{X}_r,\mathcal{L}_{\tilde{X}_r})\mathrm{d}\tilde{W}_r\rangle\right)^p \\
			\leq&\, C(t-s)^p,
		\end{align*}
		which completes the proof.
		
	\end{proof}
	
	Next we will show the existence of the martingale solution.
	
	\begin{proposition}\label{magl}
		Assume that the assumptions (\hyperref[h0]{\textbf{H0}}), (\ref{hblg}) and (\ref{hslg}) hold. If $\mathcal{L}_{X_0}\in\mathcal{P}({\overline{D(A)}})$ and $\mathbb{E}|X_0|^{2p}<\infty$ with $p>1$, then  there exists a martingale solution to the equation (\ref{mve}).
	\end{proposition}
	
	\begin{proof}
		For any fixed $n\in\mathbb{N}$, consider the following  Euler-Maruyama approximation equation:
		\begin{align*}
			\mathrm{d}X_t^n\in -A(X_t^n)+b(X^n_{t_n},\mathcal{L}_{X^n_{t_n}})\mathrm{d}t+\sigma(X^n_{t_n},\mathcal{L}_{X^n_{t_n}})\mathrm{d}W_t,\quad X^n_0=X_0,
		\end{align*}
		where $t_n=2^{-n}[2^nt]$ and $[a]$ denotes the integer part of a real number $a$. By solving a deterministic Skorohod problem, this equation can be solved step by step. That implies that there exists $(X^n,K^n)\in \mathcal{A}$ such that
		\begin{align}\label{emsolu}
			X_t^n=X_0-K_t^n+\int_{0}^{t}b(X_{s_n}^n,\mathcal{L}_{X_{s_n}^n})\mathrm{d}s+\int_{0}^{t}\sigma(X_{s_n}^n,\mathcal{L}_{X_{s_n}^n})\mathrm{d}W_s.
		\end{align}
		By Lemma \ref{l42} and the assumption $\mathbb{E}|X_0|^{2p}<\infty$, we have
		\begin{align}\label{l4b}
			&\sup_{n\in\mathbb{N}}\mathbb{E}(\sup_{t\in[0,T]}|X_t^n|^{2p})+\sup_{n\in\mathbb{N}}\mathbb{E}(|K^n|_0^T)^{p}< \infty, \\
			&\sup_{n\in\mathbb{N}}\mathbb{E}|X_t^n-X_s^n|^{2p}\leq C(t-s)^p,\quad 0\leq s<t\leq T.
		\end{align}
		
		By \citep[theorem 3.2]{Zalinescu} we deduce that $(X^n,K^n)_{n\in\mathbb{N}}$ is tight in $\mathcal{W}$. Let $\mathbb{P}_n:=\tilde{P}\circ(X^n,K^n)^{-1}$, then there exists a subsequence (still denoted by $\mathbb{P}_n$) and $\mathbb{P}_0\in \mathcal{P}(\mathcal{W})$ such that $\mathbb{P}_n$ weakly converges to $\mathbb{P}_0$. By Lemma \ref{kfcon}, we have
		\begin{align*}
			\mathbb{P}_0(\mathcal{A})=1.
		\end{align*}
		To complete the proof, it suffices to show that for each $f\in C_b^2(\mathbb{R}^d)$, $M_t^f$ is a continuous $\overline{\mathbf{W}}_t$-martingale under $\mathbb{P}_0$. Specifically, we only need to verify that for any $0\leq s<t\leq T$ and bounded $\overline{\mathbf{W}}_s$-measurable functional $G$ on $\mathcal{W}$,
		\begin{align*}
			\mathbb{E}^{\mathbb{P}_0}\left( (M_t^f(X,K)-M_s^f(X,K))G\right)=0.
		\end{align*}
		
		Let
		\begin{align*}
			M_t^{n,f}:=&f(X_t)-f(X_0)+\int_{0}^{t}\langle \nabla f(X_s),\mathrm{d}K_s\rangle-\sum_{i}\int_{0}^{t}b_i(X_{s_n},\mathcal{L}^n_{X_{s_n}})\nabla_i f(X_s)\mathrm{d}s \\
			&-\frac{1}{2}\sum_{i,j}\int_{0}^{t}(\sigma^*(X_{s_n},\mathcal{L}^n_{X_{s_n}})\sigma(X_{s_n},\mathcal{L}^n_{X_{s_n}}))_{ij}\nabla_{ij}^2(f(X_s))\mathrm{d}s,
		\end{align*}
		where $\mathcal{L}^n_{X_{s_n}}=\mathbb{P}_n\circ X_{s_n}^{-1}$.
		
		Since the equation (\ref{emsolu}) admits a weak solution $(X^n,K^n)$, there exists a martingale solution $\mathbb{P}_n$ on $(\mathcal{W},\mathbf{W})$ for (\ref{emsolu}) due to Proposition \ref{vv}. This implies that $M_t^{n,f}$ is a continuous $\overline{\mathbf{W}}_t$-adapted martingale under $\mathbb{P}_n$. Thus we have
		\begin{align}
			\mathbb{E}^{\mathbb{P}_n}\left( (M_t^{n,f}(X,K)-M_s^{n,f}(X,K))G\right)=0.
		\end{align}
		
		Since $\mathbb{P}_n$ weakly converges to $\mathbb{P}_0$ as $n\to\infty$, we obtain
		\begin{align*}
			\lim\limits_{n\to\infty}\mathbb{E}^{\mathbb{P}_n}\left((f(X_t)-f(X_s))G\right)=\mathbb{E}^{\mathbb{P}_0}\left((f(X_t)-f(X_s))G\right).
		\end{align*}
		Furthermore, by applying Lemma \ref{kfcon} we get
		\begin{align*}
			\lim\limits_{n\to\infty}\mathbb{E}^{\mathbb{P}_n}\left(G\int_{s}^{t}\left\langle \nabla f(X_r),\mathrm{d}K_r\right\rangle\right)=\mathbb{E}^{\mathbb{P}_0}\left(G\int_{s}^{t}\langle \nabla f(X_r),\mathrm{d}K_r\rangle\right).
		\end{align*}
		
		We now prove the following two equations: 
		\begin{align}\label{48b}
			\lim\limits_{n\to\infty}\mathbb{E}^{\mathbb{P}_n}\left(G\int_{0}^{t} b_i(X_{s_n},\mathcal{L}^n_{X_{s_n}})\nabla_if(X_s)\mathrm{d}s\right)=\mathbb{E}^{\mathbb{P}_0}\left(G\int_{0}^{t}b_i(X_{s},\mathcal{L}_{X_{s}})\nabla_if(X_s)\mathrm{d}s\right)
		\end{align}
		and
		\begin{align}\label{49s}
			&\lim\limits_{n\to\infty}\mathbb{E}^{\mathbb{P}_n}\left(G\int_{0}^{t} (\sigma^*(X_{s_n},\mathcal{L}^n_{X_{s_n}})\sigma(X_{s_n},\mathcal{L}^n_{X_{s_n}}))_{ij}\nabla_{ij}^2(f(X_s))\mathrm{d}s\right)  \nonumber\\
			=&\mathbb{E}^{\mathbb{P}_0}\left(G\int_{0}^{t}(\sigma^*(X_{s},\mathcal{L}_{X_{s}})\sigma(X_{s},\mathcal{L}_{X_{s}}))_{ij}\nabla_{ij}^2(f(X_s))\mathrm{d}s\right).
		\end{align}
		
		By Skorohod's representation theorem, there exists a probability space  $(\tilde{\Omega},\tilde{\mathcal{F}},\tilde{\mathbb{P}})$ and a sequence of $\mathcal{W}$-valued random variables $\left\{(\tilde{X}^n,\tilde{K}^n)\right\}_{n\in\mathbb{N}}$ and $(\tilde{X},\tilde{K})$ defined on this probability space such that
		\begin{enumerate}[(i)]
			\item the law of $(\tilde{X}^n,\tilde{K}^n)$ is $\mathbb{P}_n$ for each $n\in\mathbb{N}$ and the law of $(\tilde{X},\tilde{K})$ is $\mathbb{P}_0$,
			\item $(\tilde{X}^n,\tilde{K}^n)\to(\tilde{X},\tilde{K}),\,\, \tilde{\mathbb{P}}$-a.s. as  $n\to\infty$.
		\end{enumerate}
		
		Thus, it suffices to verify
		\begin{align}\label{410b}
			\lim\limits_{n\to\infty}\mathbb{E}^{\mathbb{P}_n}\left(G\int_{0}^{t} b_i(\tilde{X}^n_{s_n},\mathcal{L}^n_{\tilde{X}^n_{s_n}})\nabla_if(\tilde{X}^n_s)\mathrm{d}s\right)=\mathbb{E}^{\mathbb{P}_0}\left(G\int_{0}^{t}b_i(\tilde{X}_{s},\mathcal{L}_{\tilde{X}_{s}})\nabla_if(\tilde{X}_s)\mathrm{d}s\right)
		\end{align}
		and
		\begin{align}\label{411s}
			\lim\limits_{n\to\infty}&\mathbb{E}^{\mathbb{P}_n}\left(G\int_{0}^{t} (\sigma^*(\tilde{X}^n_{s_n},\mathcal{L}^n_{\tilde{X}^n_{s_n}})\sigma(\tilde{X}^n_{s_n},\mathcal{L}^n_{\tilde{X}^n_{s_n}}))_{ij}\nabla_{ij}^2(f(\tilde{X}^n_s))\mathrm{d}s\right)  \nonumber\\
			=&\mathbb{E}^{\mathbb{P}_0}\left(G\int_{0}^{t}(\sigma^*(\tilde{X}_{s},\mathcal{L}_{\tilde{X}_{s}})\sigma(\tilde{X}_{s},\mathcal{L}_{\tilde{X}_{s}}))_{ij}\nabla_{ij}^2(f(\tilde{X}_s))\mathrm{d}s\right).
		\end{align}
		
		By (\ref{l4b}), we have
		\begin{align*}
			\sup_{n\in\mathbb{N}}\mathbb{E}^{\tilde{\mathbb{P}}}\left(\sup_{t\in[0,T]}|\tilde{X}^n_t|^{2p}\right)< \infty.
		\end{align*}
		
Combining the estimate above with (ii), we deduce that
		\begin{align}
			\lim\limits_{n\to\infty}\mathbb{W}_2\left(\mathcal{L}_{\tilde{X}^n}^n,\mathcal{L}_{\tilde{X}}\right)= 0.
		\end{align}
		
		Moreover, by (\hyperref[h2]{\textbf{H2}}), (\ref{l4b}) and (i), we obtain
		\begin{align*}
			\mathbb{E}^{\tilde{\mathbb{P}}}|b(\tilde{X}^n_{s_n},\mathcal{L}^n_{\tilde{X}^n_{s_n}})|=\mathbb{E}|b(X^n_{s_n},\mathcal{L}^n_{X^n_{s_n}})|\leq \left(\mathbb{E}|b(X^n_{s_n},\mathcal{L}^n_{X^n_{s_n}})|^{2p}\right)^{\frac{1}{2p}}
			\leq C(1+\mathbb{E}|X^n_{s_n}|^{2p})^{\frac{1}{2p}}
			< \infty.
		\end{align*}
		
		 Thus (\ref{410b}) is verified by using dominated convergence theorem and the continuity of $b$.  The equation (\ref{411s}) follows similarly.

	\end{proof}

	By Proposition \ref{vv} and Proposition \ref{magl}, the equation (\ref{mve}) has a weak solution. Then we only need to prove that pathwise uniqueness holds for the equation (\ref{mve}).
	
	\begin{proposition}
		Assume that (\hyperref[h0]{\textbf{H0}}) and (\hyperref[h2]{\textbf{H2}}) hold. Then the pathwise uniqueness holds for equation (\ref{mve}).
	\end{proposition}
	
	\begin{proof}
		Let $(\tilde{\mathcal{ S }};\tilde{W},(\tilde{X}^1,\tilde{K}^1))$ and $(\tilde{\mathcal{ S }};\tilde{W},(\tilde{X}^2,\tilde{K}^2))$ be two weak solutions to (\ref{mve}) with $\tilde{X}^1_0=\tilde{X}^2_0$. Let $Y_t:=\tilde{X}^1_t-\tilde{X}^2_t$. By It\^o's formula, we have
		\begin{align}\label{pro43}
			|Y_t|^2=&2\int_{0}^{t}\langle Y_s, b(\tilde{X}^1_s,\mathcal{L}_{\tilde{X}^1_s})-b(\tilde{X}^2_s,\mathcal{L}_{\tilde{X}^2_s})\rangle\mathrm{d}s+\int_{0}^{t}\lVert\sigma(\tilde{X}^1_s,\mathcal{L}_{\tilde{X}^1_s})-\sigma(\tilde{X}^2_s,\mathcal{L}_{\tilde{X}^2_s})\rVert^2\mathrm{d}s \nonumber\\
			&+2\int_{0}^{t}\langle Y_s, \sigma(\tilde{X}^1_s,\mathcal{L}_{\tilde{X}^1_s})-\sigma(\tilde{X}^2_s,\mathcal{L}_{\tilde{X}^2_s})\mathrm{d}\tilde{W}_s\rangle-2\int_{0}^{t}\langle Y_s,\mathrm{d}\tilde{K}_s^1-\mathrm{d}\tilde{K}_s^2\rangle.
		\end{align}
		
		Taking the expectation on both two sides of the equation (\ref{pro43}), by H\"older's inequality, (\hyperref[h3]{\textbf{H3}}), Remark \ref{W2} and (\ref{mono}), we get
		\begin{align*}
			\tilde{\mathbb{E}}|Y_t|^2\leq&2\tilde{\mathbb{E}}\int_{0}^{T}\langle Y_s,b(\tilde{X}^1_s,\mathcal{L}_{\tilde{X}^1_s})-b(\tilde{X}^2_s,\mathcal{L}_{\tilde{X}^2_s})\rangle\mathrm{d}s +\tilde{\mathbb{E}}\int_{0}^{T}\lVert\sigma(\tilde{X}^1_s,\mathcal{L}_{\tilde{X}^1_s})-\sigma(\tilde{X}^2_s,\mathcal{L}_{\tilde{X}^2_s})\rVert^2\mathrm{d}s \\
			&+2\tilde{\mathbb{E}}\int_{0}^{T}\langle Y_s, \sigma(\tilde{X}^1_s,\mathcal{L}_{\tilde{X}^1_s})-\sigma(\tilde{X}^2_s,\mathcal{L}_{\tilde{X}^2_s})\mathrm{d}\tilde{W}_s\rangle  \\
			\leq& C\tilde{\mathbb{E}}\int_{0}^{t}(\varrho(|Y_s|^2)+\varrho(\mathbb{W}^2_2(\mathcal{L}_{\tilde{X}_s^1},\mathcal{L}_{\tilde{X}_s^2})))\mathrm{d}s \\
			\leq& C\tilde{\mathbb{E}}\int_{0}^{t}[\varrho(|Y_s|^2)+\varrho(\tilde{\mathbb{E}}|Y_s|^2)]\mathrm{d}s.
		\end{align*}
		
		Applying Jensen's inequality, we obtain
		\begin{align*}
			\tilde{\mathbb{E}}|Y_t|^2\leq C\int_{0}^{t}\varrho(\tilde{\mathbb{E}}|Y_s|^2)\mathrm{d}s,
		\end{align*}
		where $C$ depends on $L_2$ and $T$. By Lemma \ref{bhr}, we get $\tilde{\mathbb{E}}|Y_t|^2=0$, which implies that for any $t\in[0,T]$, $Y_t=0\,\,\,$a.s.. Therefore we get the desired result.
	\end{proof}

	\subsection{Yosida Approximations}
	
	Consider the following controlled  multivalued McKean-Vlasov SDE:
	\begin{equation}\label{zeue}
		\begin{cases}
			\mathrm{d}Z^{\epsilon,h_\epsilon}_t\in b_{\epsilon}(Z^{\epsilon,h_\epsilon}_t,\mathcal{L}_{X^{\epsilon}_t})\mathrm{d}t+\sigma_\epsilon(Z^{\epsilon,h_\epsilon}_t,\mathcal{L}_{X^{\epsilon}_t})h_\epsilon(t)\mathrm{d}t+\sqrt{\epsilon}\sigma_\epsilon(Z^{\epsilon,h_\epsilon}_t,\mathcal{L}_{X^{\epsilon}_t})\mathrm{d}W_t-A_\epsilon(Z^{\epsilon,h_\epsilon}_t)\mathrm{d}t , \\
			Z^{\epsilon,h_\epsilon}_0=x\in\overline{D(A)}.
		\end{cases}
	\end{equation}
	
	By Lemma \ref{girs} we know that the above equation has a unique solution $(Z^{\epsilon,h_\epsilon}_t,K^{\epsilon,h_\epsilon}_t)$. Next we consider the following Yosida Approximation:
	\begin{equation}
		\begin{cases}
			\mathrm{d}Z^{\epsilon,h_\epsilon,\alpha}_t= b_{\epsilon}(Z^{\epsilon,h_\epsilon,\alpha}_t,\mathcal{L}_{X^{\epsilon}_t})\mathrm{d}t+\sigma_\epsilon(Z^{\epsilon,h_\epsilon,\alpha}_t,\mathcal{L}_{X^{\epsilon}_t})h_\epsilon(t)\mathrm{d}t+\sqrt{\epsilon}\sigma_\epsilon(Z^{\epsilon,h_\epsilon,\alpha}_t,\mathcal{L}_{X^{\epsilon}_t})\mathrm{d}W_t-A_\epsilon^\alpha(Z^{\epsilon,h_\epsilon,\alpha}_t)\mathrm{d}t,  \\
			Z^{\epsilon,h_\epsilon,\alpha}_0=x\in\overline{D(A)}.
		\end{cases}
	\end{equation}
	
	
	\begin{lemma}\label{yozfin}
		Assume that (\hyperref[h0]{\textbf{H0}})-(\hyperref[h4]{\textbf{H4}}) hold. For any $p\geq 1$, $h_\epsilon\in\mathcal{D}_m$ and $x\in\overline{D(A)}$, there exists some constant $C>0$ such that for any $\epsilon\in(0,1)$ and $\alpha>0$,
		\begin{equation*}
			\mathbb{E}\sup_{t\in[0,T]}|Z^{\epsilon,h_\epsilon,\alpha}_t|^{2p}\leq C,
		\end{equation*}
		where C may be dependent of $p,\,T,\,m$ and $x$, but independent of $\epsilon$ and $\alpha$.
		
		\begin{proof}
			Since $A^\alpha_\epsilon$ is monotone, it follows that
			\begin{align*}
				-&\langle Z^{\epsilon,h_\epsilon,\alpha}_s,A^\alpha_\epsilon\left(Z^{\epsilon,h_\epsilon,\alpha}_s\right)\rangle  =-\langle Z^{\epsilon,h_\epsilon,\alpha}_s,A^\alpha_\epsilon\left(Z^{\epsilon,h_\epsilon,\alpha}_s\right)-A^\alpha_\epsilon(0)+A^\alpha_\epsilon(0)\rangle  \leq -\langle Z^{\epsilon,h_\epsilon,\alpha}_s,A^\alpha_\epsilon(0)\rangle.
			\end{align*}
			Noting that $\sup_\epsilon|A_\epsilon^\alpha(0)|\leq\sup_\epsilon|A_\epsilon^0(0)|$, we have
			\begin{align*}
				-\langle Z^{\epsilon,h_\epsilon,\alpha}_s,A^\alpha_\epsilon\left(Z^{\epsilon,h_\epsilon,\alpha}_s\right)\rangle\leq \frac{1}{2}|Z^{\epsilon,h_\epsilon,\alpha}_s|^2+C.
			\end{align*}
			By It$\hat{\text{o}}$'s formula we get
			\begin{align*}
				&|Z^{\epsilon,h_\epsilon,\alpha}_t|^{2p} \\
				=&|x|^{2p}+2p\int_{0}^{t}|Z^{\epsilon,h_\epsilon,\alpha}_s|^{2p-2}\langle Z^{\epsilon,h_\epsilon,\alpha}_s,b_\epsilon(Z^{\epsilon,h_\epsilon,\alpha}_s,\mathcal{L}_{X^\epsilon_s})\rangle\mathrm{d}s \\
				&+2p\int_{0}^{t}|Z^{\epsilon,h_\epsilon,\alpha}_s|^{2p-2}\langle Z^{\epsilon,h_\epsilon,\alpha}_s,\sigma_\epsilon(Z^{\epsilon,h_\epsilon,\alpha}_s,\mathcal{L}_{X^\epsilon_s})h_\epsilon(s)\rangle\mathrm{d}s \\
				&+2p\sqrt{\epsilon}\int_{0}^{t}|Z^{\epsilon,h_\epsilon,\alpha}_s|^{2p-2}\langle Z^{\epsilon,h_\epsilon,\alpha}_s,\sigma_\epsilon(Z^{\epsilon,h_\epsilon,\alpha}_s,\mathcal{L}_{X^\epsilon_s})\mathrm{d}W_s\rangle \\
				&+p\epsilon\int_{0}^{t}|Z^{\epsilon,h_\epsilon,\alpha}_s|^{2p-2}\lVert \sigma_\epsilon(Z^{\epsilon,h_\epsilon,\alpha}_s,\mathcal{L}_{X_s^\epsilon})\rVert^2\mathrm{d}s \\
				&+2p(p-1)\epsilon\int_{0}^{t}|Z^{\epsilon,h_\epsilon,\alpha}_s|^{2p-4}\langle Z^{\epsilon,h_\epsilon,\alpha}_s,\sigma_\epsilon(Z^{\epsilon,h_\epsilon,\alpha}_s,\mathcal{L}_{X^\epsilon_s})\sigma^\ast_\epsilon(Z^{\epsilon,h_\epsilon,\alpha}_s,\mathcal{L}_{X^\epsilon_s})Z^{\epsilon,h_\epsilon,\alpha}_s\rangle\mathrm{d}s \\
				&-2p\int_{0}^{t}|Z^{\epsilon,h_\epsilon,\alpha}_s|^{2p-2}\langle Z^{\epsilon,h_\epsilon,\alpha}_s,A_\epsilon^\alpha(Z^{\epsilon,h_\epsilon,\alpha}_s)\rangle\mathrm{d}s.
			\end{align*}
			Then by (\hyperref[h0]{\textbf{H0}})-(\hyperref[h2]{\textbf{H2}}) we have
			\begin{align*}
				|Z^{\epsilon,h_\epsilon,\alpha}_t|^{2p} \leq &C+C\int_{0}^{t}|Z^{\epsilon,h_\epsilon,\alpha}_s|^{2p}\mathrm{d}s+2p\sqrt{\epsilon}\int_{0}^{t}|Z^{\epsilon,h_\epsilon,\alpha}_s|^{2p-2}\langle Z^{\epsilon,h_\epsilon,\alpha}_s,\sigma_\epsilon(Z^{\epsilon,h_\epsilon,\alpha}_s,\mathcal{L}_{X^\epsilon_s})\mathrm{d}W_s\rangle \\
				&+2p\int_{0}^{t}|Z^{\epsilon,h_\epsilon,\alpha}_s|^{2p}\lvert h_\epsilon(s)\rvert\mathrm{d}s+2p\int_{0}^{t}|Z^{\epsilon,h_\epsilon,\alpha}_s|^{2p-1}|\sigma_\epsilon(Z^{\epsilon,h_\epsilon,\alpha}_s,\mathcal{L}_{X^\epsilon_s})h_\epsilon(s)|\mathrm{d}s.
			\end{align*}
			Let
			\begin{equation*}
				f(t):=\mathbb{E}\sup_{s\in[0,t]}|Z^{\epsilon,h_\epsilon,\alpha}_s|^{2p}.
			\end{equation*}
			
			By the same deduction in Lemma \ref{l42} we get $\sup_{t\in[0,T]}\mathbb{E}|X^\epsilon_t|^{2p}<C$, where $C$ depends on $x$, $p$ and $T$. Then by BDG's inequality, H\"{o}lder's inequality, Young's inequality, \hyperref[h1]{\textbf{(H1)}} and \hyperref[h2]{\textbf{(H2)}}, we have
			\begin{align*}
				&\mathbb{E}\sup_{t_1\in[0,t]}2p\sqrt{\epsilon}\int_{0}^{t_1}|Z^{\epsilon,h_\epsilon,\alpha}_s|^{2p-2}\langle Z^{\epsilon,h_\epsilon,\alpha}_s,\sigma_\epsilon(Z^{\epsilon,h_\epsilon,\alpha}_s,\mathcal{L}_{X^\epsilon_s})\mathrm{d}W_s\rangle  \\
				\leq& C\mathbb{E}\left[\left(\int_{0}^{t}|Z^{\epsilon,h_\epsilon,\alpha}_s|^{4p-2}\lVert\sigma_\epsilon(Z^{\epsilon,h_\epsilon,\alpha}_s,\mathcal{L}_{X^\epsilon_s})\rVert^2\mathrm{d}s\right)^{\frac{1}{2}}\right] \\
				\leq& \frac{1}{8}\mathbb{E}\sup_{s\in[0,t]}|Z^{\epsilon,h_\epsilon,\alpha}_s|^{2p}+C\int_{0}^{t}|Z^{\epsilon,h_\epsilon,\alpha}_s|^{2p}\mathrm{d}s+Ct^{\frac{1}{p}}\left(\int_{0}^{t}|Z^{\epsilon,h_\epsilon,\alpha}_s|^{2p}\mathrm{d}s\right)^{\frac{2p-2}{2p}} \\
				\leq& \frac{1}{8}f(t)+C\int_{0}^{t}\mathbb{E}|Z^{\epsilon,h_\epsilon,\alpha}_s|^{2p}\mathrm{d}s+Ct^{\frac{1}{p}}.
			\end{align*}
			
			By Young's inequality again we have
			\begin{align*}
				&\mathbb{E}\left[\int_{0}^{t}|Z^{\epsilon,h_\epsilon,\alpha}_s|^{2p}\lvert h_\epsilon(s)\rvert\mathrm{d}s\right] \\
				\leq& C\mathbb{E}\left[\left(\int_{0}^{t}|Z^{\epsilon,h_\epsilon,\alpha}_s|^{4p}\mathrm{d}s\right)^{\frac{1}{2}}\right]\left(\int_{0}^{t}|h_\epsilon(s)|^2\mathrm{d}s\right) \\
				\leq& C\mathbb{E}\left[\left(\int_{0}^{t}|Z^{\epsilon,h_\epsilon,\alpha}_s|^{4p}\mathrm{d}s\right)^{\frac{1}{2}}\right]\leq \frac{1}{4}f(t)+C\int_{0}^{t}\mathbb{E}|Z^{\epsilon,h_\epsilon,\alpha}_s|^{2p}\mathrm{d}s
			\end{align*}
			and
			\begin{align*}
				&\mathbb{E}\left[\int_{0}^{t}|Z^{\epsilon,h_\epsilon,\alpha}_s|^{2p-1}|\sigma_\epsilon(Z^{\epsilon,h_\epsilon,\alpha}_s,\mathcal{L}_{X^\epsilon_s})h_\epsilon(s)|\mathrm{d}s\right]  \\
				\leq& C\mathbb{E}\left[\left(\int_{0}^{t}|Z^{\epsilon,h_\epsilon,\alpha}_s|^{4p-2}\lVert \sigma_\epsilon(Z^{\epsilon,h_\epsilon,\alpha}_s,\mathcal{L}_{X^\epsilon_s})\rVert^2\mathrm{d}s\right)^{\frac{1}{2}}\right]+C \\
				\leq& \frac{1}{8}f(t)+C\int_{0}^{t}\mathbb{E}|Z^{\epsilon,h_\epsilon,\alpha}_s|^{2p}\mathrm{d}s+Ct^{\frac{1}{p}}.
			\end{align*}
			
			By combing the estimates above together  we get
			\begin{align*}
				f(t)\leq C\int_{0}^{t}\mathbb{E}|Z^{\epsilon,h_\epsilon,\alpha}_s|^{2p}\mathrm{d}s+Ct^{\frac{1}{p}},
			\end{align*}
			which completes the proof by Gronwall's inequality.
		\end{proof}
	\end{lemma}
	
	The following result can be proved in a similar way as above by Proposition \ref{multi}.
	
	\begin{lemma}\label{yosifin}
		Assume that (\hyperref[h0]{\textbf{H0}})-(\hyperref[h4]{\textbf{H4}}) hold. For any $p\geq 1$ and $x\in\overline{D(A)}$, there exists a constant $C>0$ such that for any $\epsilon\in(0,1)$,
		\begin{equation*}
			\mathbb{E}\sup_{t\in[0,T]}|Z^{\epsilon,h_\epsilon}_t|^{2p}+\mathbb{E}|K^{\epsilon,h_{\epsilon}}|^T_0\leq C,
		\end{equation*}
		where $C$ may be dependent of $p,T,m$ and $x$, but independent of $\epsilon$.
	\end{lemma}

	\begin{lemma}\label{afin}
		Assume that (\hyperref[h0]{\textbf{H0}})-(\hyperref[h4]{\textbf{H4}}) hold. For any $x\in\overline{D(A)}$, we have
		\begin{equation*}
			\lim\limits_{\alpha\to 0}\mathbb{E}\sup_{t\in[0,T]}|Z^{\epsilon,h_\epsilon,\alpha}_t-Z^{\epsilon,h_\epsilon}_t|^2=0
		\end{equation*}
		uniformly in $\epsilon>0$.
		
		\begin{proof}
			Firstly, let
			\begin{align*}
				N_\epsilon(t):=\int_{0}^{t}b_\epsilon(Z^{\epsilon,h_\epsilon}_s,\mathcal{L}_{X^\epsilon_s})\mathrm{d}s+\int_{0}^{t}\sigma_\epsilon(Z^{\epsilon,h_\epsilon}_s,\mathcal{L}_{X^\epsilon_s})h_\epsilon(s)\mathrm{d}s+\sqrt{\epsilon}\int_{0}^{t}\sigma_\epsilon(Z^{\epsilon,h_\epsilon}_s,\mathcal{L}_{X^\epsilon_s})\mathrm{d}W_s.
			\end{align*}
			
			Using BDG's inequality, (\hyperref[h1]{\textbf{H1}}), (\hyperref[h2]{\textbf{H2}}) and Lemma \ref{afin}, we have
			\begin{equation}\label{Nfin}
				\mathbb{E}\sup_{t\in[0,T]}|N_\epsilon(t)|^2<\infty.
			\end{equation}
			For any $ p\geq 2$ and $\Delta>0$, we also have
			\begin{equation}\label{Npfin}
				\mathbb{E}\sup_{0\leq|t-s|\leq\Delta}|N_\epsilon(t)-N_\epsilon(s)|^p\leq C_{p,m}\Delta^{\frac{p}{2}-1},
			\end{equation}
			where $C_{p,m}$ is independent of $\epsilon$ and $\Delta$.
			
			Now consider the following Yosida approximation:
			\begin{equation}\label{brez}
				\begin{cases}
					\mathrm{d}\breve{Z}^{\epsilon,h_\epsilon,\alpha}_t=\mathrm{d}N_\epsilon(t)-A^\alpha_\epsilon(\breve{Z}^{\epsilon,h_\epsilon,\alpha}_t)\mathrm{d}t, \\
					\breve{Z}^{\epsilon,h_\epsilon,\alpha}_0=x\in\overline{D(A)}.
				\end{cases}
			\end{equation}
			Since $A^\alpha_\epsilon$ is Lipschitz, the above equation admits a unique solution.
			
			First we will use the argument in \cite{rwz} to show
			\begin{equation}\label{approx}
				\lim\limits_{\alpha\to 0}\mathbb{E}\sup_{t\in[0,T]}|\breve{Z}^{\epsilon,h_\epsilon,\alpha}_t-Z^{\epsilon,h_\epsilon}_t|^2=0
			\end{equation}
			uniformly in $\epsilon> 0$.
			
			Let $N^n_\epsilon(t)$ be the $C^\infty$-approximation of $N_\epsilon(t)$ defined by
			\begin{align*}
				N^n_\epsilon(t):=n\int_{t-\frac{1}{n}}^{t+\frac{1}{n}}N_\epsilon(s-\frac{1}{n})\varphi(n(t-s))\mathrm{d}s,
			\end{align*}
			where $\varphi$ is a mollifier with $supp\,\varphi\subset(-1,1)$, $\varphi\in C^\infty$ and $\int_{-1}^{1}\varphi(s)\mathrm{d}s=1$.
			
			By (\ref{Nfin}) and (\ref{Npfin}) we can deduce that
			\begin{align*}
				&N_\epsilon^n(0)=0, \\
				&\lim\limits_{n\to\infty}\mathbb{E}\sup_{t\in[0,T]}|N^n_\epsilon(t)-N_\epsilon(t)|^2=0, \\
				&\sup_{t\in[0,T]}|N^n_\epsilon(t)|\leq\sup_{t\in[0,T]}|N_\epsilon(t)| \quad \text{a.s.}, \\
				&\sup_{|t-s|\leq\Delta}|N^n_\epsilon(t)-N^n_\epsilon(s)|\leq \sup_{|t-s|\leq\Delta}|N_\epsilon(t)-N_\epsilon(s)| \quad\text{a.s.,}\quad\forall \Delta>0, \\
				&\mathbb{E}\sup_{t\in[0,T]}|\dot{N}^n_\epsilon(t)|^2+\mathbb{E}\sup_{t\in[0,T]}|\ddot{N}^n_\epsilon(t)|^2\leq L_n,
			\end{align*}
			where $L_n$ is a positive constant which is independent of $\epsilon$ and $\alpha$.
			
			Let $(Z^{\epsilon,h_\epsilon,n}_t,K^{\epsilon,h_\epsilon,n}_t)$ be the solution of the following Multivalued McKean-Vlasov SDE
			\begin{equation}\label{zeuen}
				\begin{cases}
					\mathrm{d}Z^{\epsilon,h_\epsilon,n}_t\in \mathrm{d}N^n_\epsilon(t)-A_\epsilon(Z^{\epsilon,h_\epsilon,n}_t)\mathrm{d}t, \\
					Z^{\epsilon,h_\epsilon,n}_0=x,
				\end{cases}
			\end{equation}
			and $(\breve{Z}^{\epsilon,h_\epsilon,\alpha,n}_t,K^{\epsilon,h_\epsilon,\alpha,n}_t)$ be the solution of  the following differential equation
			\begin{equation*}
				\begin{cases}
					\mathrm{d}\breve{Z}^{\epsilon,h_\epsilon,\alpha,n}_t= \mathrm{d}N^n_\epsilon(t)-A^\alpha_\epsilon(\breve{Z}^{\epsilon,h_\epsilon,\alpha,n}_t)\mathrm{d}t, \\
					\breve{Z}^{\epsilon,h_\epsilon,\alpha,n}_0=x.
				\end{cases}
			\end{equation*}
			
			We divide the proof of (\ref{approx}) into three steps.
			
			\textbf{Step 1}: First we prove
			\begin{equation}\label{zns1}
				\lim\limits_{n\to\infty}\mathbb{E}\sup_{t\in[0,T]}|Z^{\epsilon,h_\epsilon,n}_t-Z^{\epsilon,h_\epsilon}_t|=0
			\end{equation}
			uniformly in $\epsilon>0$. By \cite[Proposition 4.3]{cepa1}, we have
			\begin{align*}
				&\sup_{t\in[0,T]}|Z^{\epsilon,h_\epsilon,n}_t-Z^{\epsilon,h_\epsilon}_t|^2  \\
				\leq&\sup_{t\in[0,T]}|N^n_\epsilon(t)-N_\epsilon(t)|^{\frac{1}{2}}\left(\sup_{t\in[0,T]}|N^n_\epsilon(t)-N_\epsilon(t)|+4|K^{\epsilon,h_\epsilon,n}|^0_T+4|K^{\epsilon,h_\epsilon}|^0_T\right)^\frac{1}{2}.
			\end{align*}
			
			By Lemma \ref{yosifin}, we have
			\begin{equation*}
				\mathbb{E}|K^{\epsilon,h_\epsilon}|^0_T\leq C,
			\end{equation*}
			where $C$ is independent of $\epsilon$. Then it is left to verify
			\begin{align*}
				\mathbb{E}|K^{\epsilon,h_\epsilon,n}|^0_T\leq C,
			\end{align*}
			where $C$ is independent of $\epsilon$ and $n$. By Proposition \ref{multi} we have for any $a\in$ Int$(D(A))$
			\begin{align*}
				&|Z^{\epsilon,h_\epsilon,n}_t-a|^2-|Z^{\epsilon,h_\epsilon,n}_s-a|^2 \\
				\leq&|N^n_\epsilon(t)+h-a|^2+|N^n_\epsilon(s)+h-a|^2-2\lambda_1|K^{\epsilon,h_\epsilon,n}|^s_t+\lambda_{2}\int_{s}^{t}|Z^{\epsilon,h_\epsilon,n}_r-a|\mathrm{d}r \\
				&+\lambda_{3}(t-s)+2\int_{s}^{t}\langle N^n_\epsilon(r)-N_\epsilon(r), \mathrm{d}K^{\epsilon,h_\epsilon,n}_r\rangle+2\langle Z^{\epsilon,h_\epsilon,n}_t-h-N^n_\epsilon(t), N^n_\epsilon(t)-N^n_\epsilon(s)\rangle \\
				\leq&C\left(1+\sup_{t\in [0,T]}|N_\epsilon(t)|^2+\sup_{t\in [0,T]}|Z^{\epsilon,h_\epsilon,n}_t-a|\right)+2\left(\sup_{|t'-s'|\leq|t-s|}|N^n_\epsilon(t')-N^n_\epsilon(s')|-\lambda_{1}\right)|K^{\epsilon,h_\epsilon,n}|^s_t,
			\end{align*}
			then we get
			\begin{align*}
				|Z^{\epsilon,h_\epsilon,n}_t-a|^2-|Z^{\epsilon,h_\epsilon,n}_s-a|^2+\lambda_{1}|K^{\epsilon,h_\epsilon,n}|^s_t
				\leq C\left(1+\sup_{t\in [0,T]}|N_\epsilon(t)|^2+\sup_{t\in [0,T]}|Z^{\epsilon,h_\epsilon,n}_t-a|\right),
			\end{align*}
			on the set
			\begin{equation*}
				\Theta_{\epsilon,n,\Delta}=\left\{\sup_{|t'-s'|\leq\Delta}|N^n_\epsilon(t')-N^n_\epsilon(s')|\leq\frac{\lambda_1}{2}\right\}.
			\end{equation*}
			By the same way in \citep[Proposition 4.9]{cepa1}, we have
			\begin{align*}
				|K^{\epsilon,h_\epsilon,n}|^0_TI_{\Theta_{\epsilon,n,\Delta}}\leq \frac{C\left(1+\sup_{t\in[0,T]}|N_\epsilon(t)|^2\right)}{\Delta}.
			\end{align*}
			Notice that for $p>6$, we can deduce that
			\begin{align*}
				\mathbb{E}|K^{\epsilon,h_\epsilon,n}|^0_T=&\sum_{k=1}^{\infty}\mathbb{E}\left[|K^{\epsilon,h_\epsilon,n}|^0_TI_{\Theta_{\epsilon,n,\frac{1}{k+1}}\backslash\Theta_{\epsilon,n,\frac{1}{k}}}\right] \\
				\leq&C\sum_{k=1}^{\infty}(k+1)P\left(\Theta_{\epsilon,n,\frac{1}{k}}^c\right)
				\leq C\sum_{k=1}^{\infty}(k+1)\mathbb{E}\sup_{|t-s|\leq\frac{1}{k}}|N_\epsilon(t)-N_\epsilon(s)|^p \\ \leq &C\sum_{k=1}^{\infty}(k+1)\left(\frac{1}{k}\right)^{\frac{p}{2}-1}<\infty.
			\end{align*}
			Hence, we obtain (\ref{zns1}).
			
			\textbf{Step 2}: By using the same procedure in Step $1$, we also have
			\begin{align}
				\lim\limits_{n\to\infty}\mathbb{E}\sup_{t\in[0,T]}|\breve{Z}^{\epsilon,h_\epsilon,\alpha,n}_t-\breve{Z}^{\epsilon,h_\epsilon,\alpha}_t|=0
			\end{align}
			uniformly in $\epsilon$ and $\alpha$.
			
			\textbf{Step 3}: Let $\alpha,\beta>0$. Since by Remark \ref{yosidap} we have $A^\alpha_\epsilon(x)\in A_\epsilon(J^\alpha(x))$, then for $\forall x,y\in\mathbb{R}^d$
			\begin{equation*}
				\langle A^\alpha_\epsilon(x)-A^\beta_\epsilon(y),x-y\rangle\leq-(\alpha+\beta)\langle A^\alpha_\epsilon(x),A^\beta_\epsilon(y)\rangle.
			\end{equation*}
			Hence
			\begin{align*}
				&\sup_{t\in[0,T]}|\breve{Z}^{\epsilon,h_\epsilon,\alpha,n}_t-\breve{Z}^{\epsilon,h_\epsilon,\beta,n}_t|^2 \\
				=&\sup_{t\in[0,T]}|\int_{0}^{t}\langle A^\alpha_\epsilon(\breve{Z}^{\epsilon,h_\epsilon,\alpha,n}_s)-A^\beta_\epsilon(\breve{Z}^{\epsilon,h_\epsilon,\beta,n}_s),\breve{Z}^{\epsilon,h_\epsilon,\alpha,n}_s-\breve{Z}^{\epsilon,h_\epsilon,\beta,n}_s\rangle\mathrm{d}s| \\
				\leq&(\alpha+\beta)\int_{0}^{T}|A^\alpha_\epsilon(\breve{Z}^{\epsilon,h_\epsilon,\alpha,n}_s)||A^\beta_\epsilon(\breve{Z}^{\epsilon,h_\epsilon,\beta,n}_s)|\mathrm{d}s \\
				\leq&(\alpha+\beta)T\sup_{s\in[0,T]}|A^\alpha_\epsilon(\breve{Z}^{\epsilon,h_\epsilon,\alpha,n}_s)|\sup_{s\in[0,T]}|A^\beta_\epsilon(\breve{Z}^{\epsilon,h_\epsilon,\beta,n}_s)|.
			\end{align*}
			By the estimation in \citep[Proposition 4.7]{cepa1} and property (\ref{yosi0}), we can deduce that
			\begin{align*}
				|\frac{d}{dt}\breve{Z}^{\epsilon,h_\epsilon,\alpha,n}|\leq |A^0_\epsilon(x)|+|\dot{N}_\epsilon^n(t)|+\int_{0}^{t}|\ddot{N}^n_\epsilon(s)|\mathrm{d}s
			\end{align*}
			and thus we have
			\begin{align*}
				\sup_{t\in[0,T]}|A^\alpha_\epsilon(\breve{Z}^{\epsilon,h_\epsilon,\alpha,n}_t)|\leq|A^0_\epsilon(x)|+2\sup_{t\in[0,T]}|\dot{N}^n_\epsilon(t)|+T\sup_{t\in[0,T]}|\ddot{N}^n_\epsilon(t)|.
			\end{align*}
			Combining the above estimations we have
			\begin{equation*}
				\mathbb{E}\sup_{t\in[0,T]}|\breve{Z}^{\epsilon,h_\epsilon,\alpha,n}_t-\breve{Z}^{\epsilon,h_\epsilon,\beta,n}_t|^2\leq\,L_n(\alpha+\beta),
			\end{equation*}
			where $L_n$ is a constant independent of $\epsilon$, $\alpha$ and $\beta$. Then by \citep[Proposition 4.7]{cepa1} we know that
			\begin{equation*}
				\sup_{t\in[0,T]}|\breve{Z}^{\epsilon,h_\epsilon,\alpha,n}_t-Z^{\epsilon,h_\epsilon,n}_t|\to 0\,\,\,a.s.\quad \text{as }\,\alpha\to 0.
			\end{equation*}
			
			Thus we obtain the result of the first part
			\begin{equation}\label{yosida1}
				\lim\limits_{\alpha\to 0}\mathbb{E}\sup_{t\in[0,T]}|\breve{Z}^{\epsilon,h_\epsilon,\alpha,n}_t-Z^{\epsilon,h_\epsilon,n}_t|^2= 0
			\end{equation}
			uniformly in $\epsilon>0$.

			Secondly, by It{\^o}'s formula and monotonicity of $A_\epsilon$, we can deduce that
			\begin{align*}
				&|\breve{Z}^{\epsilon,h_\epsilon,\alpha}_t-Z^{\epsilon,h_\epsilon,\alpha}_t|^2  \\
				\leq &\,\,2\int_{0}^{t}\langle \breve{Z}^{\epsilon,h_\epsilon,\alpha}_s-Z^{\epsilon,h_\epsilon}_s,\,b_\epsilon(Z^{\epsilon,h_\epsilon}_s,\mathcal{L}_{X^\epsilon_s})-b_\epsilon(Z^{\epsilon,h_\epsilon,\alpha}_s,\mathcal{L}_{X^\epsilon_s})\rangle\mathrm{d}s \\
				&+2\int_{0}^{t}\langle Z^{\epsilon,h_\epsilon}_s-Z^{\epsilon,h_\epsilon,\alpha}_s,\,b_\epsilon(Z^{\epsilon,h_\epsilon}_s,\mathcal{L}_{X^\epsilon_s})-b_\epsilon(Z^{\epsilon,h_\epsilon,\alpha}_s,\mathcal{L}_{X^\epsilon_s})\rangle\mathrm{d}s \\
				&+2\int_{0}^{t}\langle \breve{Z}^{\epsilon,h_\epsilon,\alpha}_s-Z^{\epsilon,h_\epsilon}_s,\,\sigma_\epsilon(Z^{\epsilon,h_\epsilon}_s,\mathcal{L}_{X^\epsilon_s})h_\epsilon(s)-\sigma_\epsilon(Z^{\epsilon,h_\epsilon,\alpha}_s,\mathcal{L}_{X^\epsilon_s})h_\epsilon(s)\rangle\mathrm{d}s \\
				&+2\int_{0}^{t}\langle Z^{\epsilon,h_\epsilon}_s-Z^{\epsilon,h_\epsilon,\alpha}_s,\,\sigma_\epsilon(Z^{\epsilon,h_\epsilon}_s,\mathcal{L}_{X^\epsilon_s})h_\epsilon(s)-\sigma_\epsilon(Z^{\epsilon,h_\epsilon,\alpha}_s,\mathcal{L}_{X^\epsilon_s})h_\epsilon(s)\rangle\mathrm{d}s \\
				&+2\sqrt{\epsilon}\int_{0}^{t}\langle \breve{Z}^{\epsilon,h_\epsilon,\alpha}_s-Z^{\epsilon,h_\epsilon}_s,\,\left(\sigma_\epsilon(Z^{\epsilon,h_\epsilon}_s,\mathcal{L}_{X^\epsilon_s})-\sigma_\epsilon(Z^{\epsilon,h_\epsilon,\alpha}_s,\mathcal{L}_{X^\epsilon_s})\right)\mathrm{d}W_s\rangle  \\
				&+2\sqrt{\epsilon}\int_{0}^{t}\langle Z^{\epsilon,h_\epsilon}_s-Z^{\epsilon,h_\epsilon,\alpha}_s,\,\left(\sigma_\epsilon(Z^{\epsilon,h_\epsilon}_s,\mathcal{L}_{X^\epsilon_s})-\sigma_\epsilon(Z^{\epsilon,h_\epsilon,\alpha}_s,\mathcal{L}_{X^\epsilon_s})\right)\mathrm{d}W_s\rangle  \\
				&+\frac{\epsilon}{2}\int_{0}^{t}\lVert \sigma_\epsilon(Z^{\epsilon,h_\epsilon}_s,\mathcal{L}_{X^\epsilon_s})-\sigma_\epsilon(Z^{\epsilon,h_\epsilon,\alpha}_s,\mathcal{L}_{X^\epsilon_s})\rVert^2\mathrm{d}s \\
				:=&\,\,I^{\epsilon,\alpha}_1(t)+I^{\epsilon,\alpha}_2(t)+I^{\epsilon,\alpha}_3(t)+I^{\epsilon,\alpha}_4(t)+I^{\epsilon,\alpha}_5(t)+I^{\epsilon,\alpha}_2(t)+I^{\epsilon,\alpha}_6(t)+I^{\epsilon,\alpha}_7(t).
			\end{align*}
			
			For $I^{\epsilon,\alpha}_1(t)$, by H{\"o}lder's inequality, Young's inequality, (\hyperref[h1]{\textbf{H1}})-(\hyperref[h4]{\textbf{H4}}), Lemma \ref{yozfin} and Lemma \ref{yosifin}, we have
			\begin{align}\label{431}
				&\mathbb{E}\sup_{t\in[0,T]}|I^{\epsilon,\alpha}_1(t)| \nonumber\\
				\leq& C\left(1+\mathbb{E}\sup_{t\in[0,T]}|Z^{\epsilon,h_\epsilon}_t|+\mathbb{E}\sup_{t\in[0,T]}|Z^{\epsilon,h_\epsilon,\alpha}_t|+(\sup_{t\in[0,T]}\mathbb{E}|X^\epsilon_t|^2)^{\frac{1}{2}}\right)\left(\mathbb{E}\sup_{t\in[0,T]}|\breve{Z}^{\epsilon,h_\epsilon,\alpha}_t-Z^{\epsilon,h_\epsilon}_t|^2\right)^\frac{1}{2} \nonumber\\
				\leq& C\left(\mathbb{E}\sup_{t\in[0,T]}|\breve{Z}^{\epsilon,h_\epsilon,\alpha}_t-Z^{\epsilon,h_\epsilon}_t|^2\right)^\frac{1}{2}.
			\end{align}
			
			By the similar way, for $I^{\epsilon,\alpha}_2(t)$, $I^{\epsilon,\alpha}_3(t)$ and $I^{\epsilon,\alpha}_4(t)$, we have
			\begin{align}\label{432}
				\mathbb{E}\sup_{t\in[0,T]}|I^{\epsilon,\alpha}_2(t)|
				\leq
				C\mathbb{E}\int_{0}^{T}\varrho\left[\sup_{s\in[0,t]}|Z^{\epsilon,h_\epsilon}_s-Z^{\epsilon,h_\epsilon,\alpha}_s|^2\right]\mathrm{d}t,
			\end{align}
			\begin{align}\label{433}
				&\mathbb{E}\sup_{t\in[0,T]}|I^{\epsilon,\alpha}_3(t)| \nonumber\\
				\leq& C\mathbb{E}\left[\left(\int_{0}^{T}|\breve{Z}^{\epsilon,h_\epsilon,\alpha}_t-Z^{\epsilon,h_\epsilon}_t|^2\lVert\sigma_\epsilon(Z^{\epsilon,h_\epsilon}_t,\mathcal{L}_{X^\epsilon_t})-\sigma_\epsilon(Z^{\epsilon,h_\epsilon,\alpha}_t,\mathcal{L}_{X^\epsilon_t})\rVert^2\mathrm{d}t\right)^\frac{1}{2}\left(\int_{0}^{T}|h_\epsilon(t)|^2\mathrm{d}t\right)^{\frac{1}{2}}\right] \nonumber\\
				\leq& C\left(1+\mathbb{E}\sup_{t\in[0,T]}|Z^{\epsilon,h_\epsilon}_t|+\mathbb{E}\sup_{t\in[0,T]}|Z^{\epsilon,h_\epsilon,\alpha}_t|+(\sup_{t\in[0,T]}\mathbb{E}|X^\epsilon_t|^2)^{\frac{1}{2}}\right)\left(\mathbb{E}\sup_{t\in[0,T]}|\breve{Z}^{\epsilon,h_\epsilon,\alpha}_t-Z^{\epsilon,h_\epsilon}_t|^2\right)^{\frac{1}{2}} \nonumber\\
				\leq&
				C\left(\mathbb{E}\sup_{t\in[0,T]}|\breve{Z}^{\epsilon,h_\epsilon,\alpha}_t-Z^{\epsilon,h_\epsilon}_t|^2\right)^{\frac{1}{2}}
			\end{align}
			and
			\begin{align}\label{434}
				&\mathbb{E}\sup_{t\in[0,T]}|I^{\epsilon,\alpha}_4(t)| \nonumber\\
				\leq& C\mathbb{E}\left[\left(\int_{0}^{T}|Z^{\epsilon,h_\epsilon}_t-Z^{\epsilon,h_\epsilon,\alpha}_t|^2\lVert\sigma_\epsilon(Z^{\epsilon,h_\epsilon}_t,\mathcal{L}_{X^\epsilon_t})-\sigma_\epsilon(Z^{\epsilon,h_\epsilon,\alpha}_t,\mathcal{L}_{X^\epsilon_t})\rVert^2\mathrm{d}t\right)^\frac{1}{2}\left(\int_{0}^{T}|h_\epsilon(t)|^2\mathrm{d}t\right)^{\frac{1}{2}}\right] \nonumber\\
				\leq& \frac{1}{4}\mathbb{E}\sup_{t\in[0,T]}|Z^{\epsilon,h_\epsilon}_t-Z^{\epsilon,h_\epsilon,\alpha}_t|^2+C\mathbb{E}\int_{0}^{T}\varrho\left[\sup_{s\in[0,t]}|Z^{\epsilon,h_\epsilon}_s-Z^{\epsilon,h_\epsilon,\alpha}_s|^2\right]\mathrm{d}t.
			\end{align}
			
			By BDG's inequality, H{\"o}lder's inequality, Young's inequality, (\hyperref[h1]{\textbf{H1}}), Lemma \ref{yozfin} and Lemma \ref{yosifin}, we obtain
			\begin{align}\label{435}
				&\mathbb{E}\sup_{t\in[0,T]}|I^{\epsilon,\alpha}_5(t)| \nonumber\\
				\leq& C\mathbb{E}\left[\left(\int_{0}^{T}|\breve{Z}^{\epsilon,h_\epsilon,\alpha}_t-Z^{\epsilon,h_\epsilon}_t|^2\lVert\sigma_\epsilon(Z^{\epsilon,h_\epsilon}_t,\mathcal{L}_{X^\epsilon_t})-\sigma_\epsilon(Z^{\epsilon,h_\epsilon,\alpha}_t,\mathcal{L}_{X^\epsilon_t})\rVert^2\mathrm{d}t\right)^\frac{1}{2}\right] \nonumber\\
				\leq& C\mathbb{E}\left[\sup_{t\in[0,T]}|\breve{Z}^{\epsilon,h_\epsilon,\alpha}_t-Z^{\epsilon,h_\epsilon}_t|\sup_{t\in[0,T]}\lVert\sigma_\epsilon(Z^{\epsilon,h_\epsilon}_t,\mathcal{L}_{X^\epsilon_t})-\sigma_\epsilon(Z^{\epsilon,h_\epsilon,\alpha}_t,\mathcal{L}_{X^\epsilon_t})\rVert\right] \nonumber\\
				\leq& C\left(\mathbb{E}\sup_{t\in[0,T]}|\breve{Z}^{\epsilon,h_\epsilon,\alpha}_t-Z^{\epsilon,h_\epsilon}_t|^2\right)^\frac{1}{2}\left(\mathbb{E}\sup_{t\in[0,T]}\lVert\sigma_\epsilon(Z^{\epsilon,h_\epsilon}_t,\mathcal{L}_{X^\epsilon_t})-\sigma_\epsilon(Z^{\epsilon,h_\epsilon,\alpha}_t,\mathcal{L}_{X^\epsilon_t})\rVert^2\right)^\frac{1}{2} \nonumber\\
				\leq& C\left(1+\mathbb{E}\sup_{t\in[0,T]}|Z^{\epsilon,h_\epsilon}_t|+\mathbb{E}\sup_{t\in[0,T]}|Z^{\epsilon,h_\epsilon,\alpha}_t|+(\sup_{t\in[0,T]}\mathbb{E}|X^\epsilon_t|^2)^{\frac{1}{2}}\right)\left(\mathbb{E}\sup_{t\in[0,T]}|\breve{Z}^{\epsilon,h_\epsilon,\alpha}_t-Z^{\epsilon,h_\epsilon}_t|^2\right)^\frac{1}{2} \nonumber\\
				\leq& C\left(\mathbb{E}\sup_{t\in[0,T]}|\breve{Z}^{\epsilon,h_\epsilon,\alpha}_t-Z^{\epsilon,h_\epsilon}_t|^2\right)^\frac{1}{2}
			\end{align}
			and
			\begin{align}\label{436}
				&\mathbb{E}\sup_{t\in[0,T]}|I^{\epsilon,\alpha}_6(t)| \nonumber\\
				\leq& C\mathbb{E}\left[\left(\int_{0}^{T}|Z^{\epsilon,h_\epsilon,\alpha}_t-Z^{\epsilon,h_\epsilon}_t|^2\lVert\sigma_\epsilon(Z^{\epsilon,h_\epsilon}_t,\mathcal{L}_{X^\epsilon_t})-\sigma_\epsilon(Z^{\epsilon,h_\epsilon,\alpha}_t,\mathcal{L}_{X^\epsilon_t})\rVert^2\mathrm{d}t\right)^\frac{1}{2}\right] \nonumber\\
				\leq& C\left(\mathbb{E}\sup_{t\in[0,T]}|Z^{\epsilon,h_\epsilon,\alpha}_t-Z^{\epsilon,h_\epsilon}_t|\right)\left(\mathbb{E}\int_{0}^{T}\lVert\sigma_\epsilon(Z^{\epsilon,h_\epsilon}_t,\mathcal{L}_{X^\epsilon_t})-\sigma_\epsilon(Z^{\epsilon,h_\epsilon,\alpha}_t,\mathcal{L}_{X^\epsilon_t})\rVert^2\mathrm{d}t\right)^\frac{1}{2} \nonumber\\
				\leq& \frac{1}{4}\mathbb{E}\sup_{t\in[0,T]}|Z^{\epsilon,h_\epsilon,\alpha}_t-Z^{\epsilon,h_\epsilon}_t|^2+C\mathbb{E}\int_{0}^{T}\varrho\left[\sup_{s\in[0,t]}|Z^{\epsilon,h_\epsilon}_s-Z^{\epsilon,h_\epsilon,\alpha}_s|^2\right]\mathrm{d}t.
			\end{align}
			
			For the last term, we have
			\begin{align}\label{437}
				\mathbb{E}\sup_{t\in[0,T]}|I^{\epsilon,\alpha}_7(t)|\leq C\mathbb{E}\int_{0}^{T}\varrho\left[\sup_{s\in[0,t]}|Z^{\epsilon,h_\epsilon}_s-Z^{\epsilon,h_\epsilon,\alpha}_s|^2\right]\mathrm{d}t.
			\end{align}
			
			Denote
			\begin{equation*}
				R(t):=|Z^{\epsilon,h_\epsilon,\alpha}_t-Z^{\epsilon,h_\epsilon}_t|^2,
			\end{equation*}
			
			then by combining (\ref{431})-(\ref{437}) together and using Jensen's inequality we have
			\begin{align}\label{rt}
				\mathbb{E}\sup_{t\in[0,T]}R(t)\leq&\,\, 2\mathbb{E}\sup_{t\in[0,T]}|\breve{Z}^{\epsilon,h_\epsilon,\alpha}_t-Z^{\epsilon,h_\epsilon,\alpha}_t|^2+2\mathbb{E}\sup_{t\in[0,T]}|\breve{Z}^{\epsilon,h_\epsilon,\alpha}_t-Z^{\epsilon,h_\epsilon}_t|^2\nonumber \\
				\leq&\,\, \frac{1}{2}\mathbb{E}\sup_{t\in[0,T]}R(t)+C\mathbb{E}\int_{0}^{T}\varrho(\sup_{s\in[0,t]}R(s))\mathrm{d}t+C\left(\mathbb{E}\sup_{t\in[0,T]}|\breve{Z}^{\epsilon,h_\epsilon,\alpha}_t-Z^{\epsilon,h_\epsilon}_t|^2\right)^{\frac{1}{2}} \nonumber\\
				&+2\mathbb{E}\sup_{t\in[0,T]}|\breve{Z}^{\epsilon,h_\epsilon,\alpha}_t-Z^{\epsilon,h_\epsilon}_t|^2 \nonumber\\
				=:&\frac{1}{2}\mathbb{E}\sup_{t\in[0,T]}R(t)+C\mathbb{E}\int_{0}^{T}\varrho(\sup_{s\in[0,t]}R(s))\mathrm{d}t+O_1(\alpha).
			\end{align}
			Since the $L$ in (\ref{rt}) is independent of $\epsilon$ and $\alpha$, by (\ref{yosida1}) we have
			\begin{align*}
				\lim\limits_{\alpha\to 0}O_1(\alpha)=0.
			\end{align*}
			Setting $f(t)=\int_{1}^{t}1/\varrho(s)\mathrm{d}s$, by Lemma \ref{bhr}, we have
			\begin{align*}\label{ad1}
				\mathbb{E}\sup_{t\in[0,T]}R(t)\leq f^{-1}\left(f(O_1(\alpha))+2LT\right).
			\end{align*}
			Since $f$ and $f^{-1}$ are strictly increasing functions and $\int_{0^+}1/\varrho(s)\mathrm{d}s=\infty$, we obtain
			\begin{align*}\label{ad2}
				\lim\limits_{\alpha\to 0}f^{-1}\left(f(O_1(\alpha))+2LT\right)= 0.
			\end{align*}
			Combining (\ref{ad1}) and (\ref{ad2}) together, we get
			\begin{equation*}
				\lim\limits_{\alpha\to 0}\mathbb{E}\sup_{t\in[0,T]}|Z^{\epsilon,h_\epsilon,\alpha}_t-Z^{\epsilon,h_\epsilon}_t|^2=0
			\end{equation*}
			uniformly in $\epsilon>0$, which completes the proof.
			
		\end{proof}
	\end{lemma}
	
	Next we consider the Yosida approximation of the  equation (\ref{yy}):
	\begin{align*}
		\begin{cases}
			\mathrm{d}Y^{h,\alpha}_t=\,b(Y^{h,\alpha}_t,\mathcal{L}_{X^0_t})\mathrm{d}t+\sigma(Y^{h,\alpha}_t,\mathcal{L}_{X^0_t})h(t)\mathrm{d}t-A^\alpha(Y^{h,\alpha}_t)\mathrm{d}t, \\
			Y^{h,\alpha}_0=x\in\overline{D(A)}.
		\end{cases}
	\end{align*}
	
	Similar to Lemma \ref{afin}, we have the following result. For the sake of brevity, we omit the detailed proofs here.
	\begin{lemma}\label{ayfin}
		Assume that \hyperref[h0]{(\textbf{H0})}-\hyperref[h4]{\textbf{(H4)}} hold. For any $p\geq 1$ and $x\in\overline{D(A)}$, there exists $C>0$ such that for any $\alpha>0$,
		\begin{equation*}
			\sup_{t\in[0,T]}|Y^{h,\alpha}_t|^{2p}\leq C,
		\end{equation*}
		where $C$ may be dependent of $p,T,m$ and $x$, but independent of $\epsilon$ and $\alpha$.
		
		Moreover, for any $x\in\overline{D(A)}$, we have
		\begin{equation*}
			\lim\limits_{\alpha\to 0}\sup_{t\in[0,T]}|Y^{h,\alpha}_t-Y^h_t|^2=0.
		\end{equation*}
	\end{lemma}

	\subsection{Proof of Theorem 3.6}
	Next we prove the main result on the LDP. We will verify the conditions \textbf{(LDP)$\bm{_{1}}$} and \textbf{(LDP)$\bm{_{2}}$} separately.
	\begin{proposition}\label{ld1}[\textbf{(LDP)$\bm{_{1}}$}]
		For any given $m\in\left(0,\infty\right)$, if $\left\{h_{n}, n\in \mathbb{N}\right\}\subset \mathcal{S}_{m}$ weakly converges to $h\in \mathcal{S}_{m}$ as $n\to \infty$, then
		$$
		\lim\limits_{n\to \infty}\sup\limits_{t\in \left[0,T\right]}\lvert \mathcal{G}^{0}\left(h_{n}\right)\left(t\right)-\mathcal{G}^{0}\left(h\right)\left(t\right) \rvert=0.
		$$
	\end{proposition}	
	
	\begin{proof}
		Let $Y^h$ be the solution of (\ref{yy}) and $Y^{h}_n$ be the solution of (\ref{yy}) with $h$ replaced by $h_n$. By the definition of $\mathcal{G}^0$, we have $Y^h=\mathcal{G}^0(h)$ and $Y^{h_n}=\mathcal{G}^0(h_n)$, then it is left to verify that
		\begin{align*}
			\lim\limits_{n\to\infty}\sup_{t\in[0,T]}|Y^{h_n}_t-Y^h_t|=0.
		\end{align*}
		
		Note that $Y^h,Y^{h_n}\in C([0,T],\overline{D(A)})$, $\forall n\in\mathbb{N}$. By Proposition \ref{Yu} we know that both $Y^h$ and $\left\{Y_{h_n}\right\}_{n\geq 1}$ are uniformly bounded, i.e. there exists some constant $C>0$ such that
		\begin{align}\label{yf}
			\max\left\{\sup_{n\geq 1}\sup_{t\in[0,T]}|Y^{h_n}_t|,\sup_{t\in[0,T]}|Y^h_t|\right\}\leq C.
		\end{align}
		
		Set $\varsigma_{n}(t):=Y_{t}^{h_{n}}-Y_{t}^{h}$, then
		\begin{align*}
			\varsigma_{n}(t)=&-(K_{t}^{h_{n}}-K_{t}^{h}) +\int_{0}^{t}\left[b\left(Y_{r}^{h_{n}},\mathcal{L}_{X_{r}^{0}}\right)-b\left(Y_{r}^{h},\mathcal{L}_{X_{r}^{0}}\right)\right]\,\mathrm{d}r \\
			&+\int_{0}^{t}\left[\sigma\left(Y_{r}^{h_{n}},\mathcal{L}_{X_{r}^{0}}\right)h_{n}(r)-\sigma\left(Y_{r}^{h},\mathcal{L}_{X_{r}^{0}}\right)h(r)\right]\,\mathrm{d}r.
		\end{align*}
		
		It follows from It$\hat{\text{o}}$'s formula that
		\begin{align*}
			|\varsigma_{n}(t)|^{2}= &2\int_{0}^{t}\left\langle\mathrm{d}K_{r}^{h}-\mathrm{d}K_{r}^{h_{n}},\varsigma_{n}(r)\right\rangle+2\int_{0}^{t}\left\langle b\left(Y_{r}^{h_{n}},\mathcal{L}_{X_{r}^{0}}\right)-b\left(Y_{r}^{h},\mathcal{L}_{X_{r}^{0}}\right),\varsigma_{n}(r)\right\rangle\mathrm{d}r \\
			&+2\int_{0}^{t}\left\langle \sigma\left(Y_{r}^{h},\mathcal{L}_{X_{r}^{0}}\right)\left[h_{n}(r)-h(r)\right],\varsigma_{n}(r)\right\rangle\mathrm{d}r \\
			&+2\int_{0}^{t}\left\langle \left[\sigma\left(Y_{r}^{h_{n}},\mathcal{L}_{X_{r}^{0}}\right)-\sigma\left(Y_{r}^{h},\mathcal{L}_{X_{r}^{0}}\right)\right]h_{n}(r),\varsigma_{n}(r)\right\rangle\mathrm{d}r \\
			=:&I_{1}^{n}(t)+I_{2}^{n}(t)+I_{3}^{n}(t)+I_{4}^{n}(t).
		\end{align*}
		
		By the Assumption (\hyperref[h1]{\textbf{H1}}) and Proposition \ref{multi}, there exists some constant $L>0$ such that
		\begin{align}\label{361}
			I_{1}^{n}(t)\leq0
		\end{align}
		and
		\begin{align}\label{362}
			|I_{2}^{n}(t)|\leq 2L\int_{0}^{t}\varrho(\sup_{s\in[0,r]}\lvert\varsigma_{n}(s)\rvert^{2})\mathrm{d}r.
		\end{align}			
		By Young's inequalities, the definition of $ \mathcal{S}_{m} $ and Lemma 2.1, we get
		\begin{align}\label{363}
			|I_{4}^{n}(t)|=
			&2\int_{0}^{t}\left\langle \left[\sigma\left(Y_{r}^{h_{n}},\mathcal{L}_{X_{r}^{0}}\right)-\sigma\left(Y_{r}^{h},\mathcal{L}_{X_{r}^{0}}\right)\right]h_{n}(r),\varsigma_{n}(r)\right\rangle\mathrm{d}r \nonumber\\
			&\leq 2\sqrt{L}\int_{0}^{t}\lvert h_{n}(r)\rvert \lvert\varsigma_{n}(r)\rvert\left(\varrho(|\varsigma_n(r)|^2)\right)^{\frac{1}{2}}\mathrm{d}r \nonumber\\
			&\leq 2\sqrt{L}\left(\int_{0}^{T}\lvert h_{n}(r)\rvert^{2}\mathrm{d}r\right)^{\frac{1}{2}}\left(\int_{0}^{t}\lvert \varsigma_{n}(r)\rvert^{2}\varrho(|\varsigma_n(r)|^2)\mathrm{d}r\right)^{\frac{1}{2}} \nonumber\\
			&\leq 2\sqrt{Lm}\left[\left(\sup_{r\in[0,t]}\lvert\varsigma_{n}(r)\rvert^{2}\right)\left(\int_{0}^{t}\varrho(|\varsigma_n(r)|^2)\mathrm{d}r\right)\right]^{\frac{1}{2}} \nonumber\\
			&\leq \frac{1}{2}\sup_{r\in[0,t]}\lvert\varsigma_{n}(r)\rvert^{2}+C\int_{0}^{t}\varrho(\sup_{s\in[0,r]}\lvert\varsigma_{n}(s)\rvert^{2})\mathrm{d}r.
		\end{align}
		
		

		Combining (\ref{361}), (\ref{362}) and (\ref{363}) together, we obtain
		\begin{equation*}
			\sup\limits_{r\in[0,t]}|\varsigma_{n}(r)|^{2}\leq \frac{1}{2}\sup\limits_{r\in[0,t]}|\varsigma_{n}(r)|^{2}+C\int_{0}^{t}\varrho(\sup\limits_{s\in[0,r]}|\varsigma_{n}(s)|^{2})\mathrm{d}r+\sup\limits_{r\in[0,t]}|I^n_3(r)|,
		\end{equation*}
		which implies that
		\begin{equation}\label{y11}
			\sup\limits_{r\in[0,t]}|\varsigma_{r}|^{2}\leq C\int_{0}^{t}\varrho(\sup\limits_{s\in[0,r]}|\varsigma_{n}(s)|^{2})\mathrm{d}r+\sup\limits_{r\in[0,t]}|I^n_3(r)|.
		\end{equation}

		For $I^n_3(t)$, since $h,\,h_n\in \mathcal{S}_m$, we set $$g_n(t):=\int_{0}^{t}\sigma(Y_{r}^{h},\mathcal{L}_{X_{r}^{0}})\left[h_n(r)-h(r)\right]\mathrm{d}r.$$
		
		Then we need to show that $g_n(t)$ tends to 0 in $C([0,T],\mathbb{R}^d)$. To begin with, we show that
		\begin{enumerate}
			\item[(i)] $\sup_{n\geq 1}\sup_{r\in[0,T]}|g_n(r)|<\infty$;
			\item[(ii)] $t\mapsto g_n(t), n\geq 1$ is equi-continuous.
		\end{enumerate}
		
		Firstly, for $0\leq s<t\leq T$, by \hyperref[h2]{(\textbf{H2})} we have
		\begin{align*}
			|g_n(t)-g_n(s)|\leq &\,\left\lvert\int_{s}^{t}\sigma(Y_{r}^{h},\mathcal{L}_{X_{r}^{0}})\left[h_n(r)-h(r)\right]\mathrm{d}r\right\rvert \\
			\leq &\,\int_{s}^{t}\lVert\sigma(Y_{r}^{h},\mathcal{L}_{X_{r}^{0}})\rVert\lvert h_n(r)-h(r)\rvert\mathrm{d}r \\
			\leq &\,\left(\int_{s}^{t}\lVert\sigma(Y_{r}^{h},\mathcal{L}_{X_{r}^{0}})\rVert^2\mathrm{d}r\right)^{\frac{1}{2}}\left(\int_{s}^{t}|h_n(r)-h(r)|^2\mathrm{d}r\right)^{\frac{1}{2}} \\
			\leq&\, 2\sqrt{m}\left(\int_{s}^{t}\lVert\sigma(Y_{r}^{h},\mathcal{L}_{X_{r}^{0}})\rVert^2\mathrm{d}r\right)^{\frac{1}{2}} \\
			\leq&\, 2\sqrt{Lm}\left(\int_{s}^{t}(1+|Y_{r}^{h}|^2+\lVert \mathcal{L}_{X_{r}^{0}}\rVert_2^2)\mathrm{d}r\right)^{\frac{1}{2}}.
		\end{align*}
		
		Let $s=0$. By (\ref{yf}) and Remark \ref{W2}, we have
		\begin{align*}
			|g_n(t)|\leq 2\sqrt{Lm}(C+\mathbb{E}\sup_{r\in[0,t]}|X_r^0|^2)<\infty
		\end{align*}
		where $C$ is independent of $n$, and then we obtain (i).
		By similar deduction we have
		\begin{align*}
			|g_n(t)-g_n(s)|\leq 2\sqrt{Lm}(C+\mathbb{E}\sup_{r\in[0,t]}|X_r^0|^2)\sqrt{t-s},
		\end{align*}
		which implies (ii).
		
		By (i), (ii), and Ascoli-Arzela lemma we obtain that $\left\{g_n(t),n\geq1\right\}$ is pre-compact in $C([0,T],\mathbb{R}^d)$.
		
		Note that
		\begin{equation*}
			\int_{0}^{t}\lVert \sigma(Y_{r}^{h},\mathcal{L}_{X_{r}^{0}})\rVert^2\mathrm{d}r<\infty
		\end{equation*}
		and $h_n(t)$ converges to $h(t)$ weakly in $L^2([0,T],\mathbb{R}^d)$ as $n\to\infty$, which implies that
		\begin{align}\label{gf}
			\lim\limits_{n\to\infty}\sup_{t\in[0,T]}|g_n(t)|=0.
		\end{align}
		
		By using Taylor formula for $\langle \varsigma_n(s ),g_n(s)\rangle$, we can obtain that
		\begin{align*}
			\frac{1}{2}I^n_3(t)=&\langle g_n(t),\varsigma_{n}(t)\rangle-\int_{0}^{t}\langle \mathrm{d}K^h_s-\mathrm{d}K^{h_n}_s, g_n(s)\rangle -\int_{0}^{t}\langle b(Y_{s}^{h_n},\mathcal{L}_{X_{s}^{0}})-b(Y_{s}^{h},\mathcal{L}_{X_{s}^{0}}), g_n(s)\rangle\mathrm{d}s \\
			&-\int_{0}^{t}\langle \sigma(Y_{s}^{h_n},\mathcal{L}_{X_{s}^{0}})h_n(s)-\sigma(Y_{s}^{h},\mathcal{L}_{X_{s}^{0}})h(s), g_n(s)\rangle\mathrm{d}s \\
			=:& I^n_{31}(t)+I^n_{32}(t)+I^n_{33}(t)+I^n_{34}(t).
		\end{align*}
		
		For $I^n_{31}(t)$, since
		\begin{align*}
			\sup_{t\in[0,T]}|I^n_{31}(t)|\leq \sup_{t\in[0,T]}|g_n(t)|\,\sup_{t\in[0,T]}|\varsigma_n(t)|,
		\end{align*}
		by (\ref{yf}) and (\ref{gf}) we have
		\begin{align}\label{3631}
			\lim\limits_{n\to\infty}\sup_{t\in[0,T]}|I^n_{31}(t)|=0.
		\end{align}
		Similarly, by (\ref{mono}) and (\ref{yf}) we obtain
		\begin{align}\label{3632}
			\lim\limits_{n\to\infty}\sup_{t\in[0,T]}|I^n_{32}(t)|=0.
		\end{align}
		
		For $I^n_{33}(t)$, note that
		\begin{align*}
			\sup_{t\in[0,T]}|I^n_{33}(t)|\leq T\sup_{t\in[0,T]}\left(|b(Y_{t}^{h_n},\mathcal{L}_{X_{t}^{0}})|+|b(Y_{t}^{h},\mathcal{L}_{X_{t}^{0}})| \right)\,\sup_{t\in[0,T]}|g_n(t)|,
		\end{align*}
		then by \hyperref[h1]{\textbf{(H1)}} and (\ref{gf}) we get
		\begin{align}\label{3633}
			\lim\limits_{n\to\infty}\sup_{t\in[0,T]}|I^n_{33}(t)|=0.
		\end{align}
		By \hyperref[h2]{\textbf{(H2)}} and $h_n,\,h\in \mathcal{ S }_m$, we have
		\begin{align}\label{3634}
			\lim\limits_{n\to\infty}\sup_{t\in[0,T]}|I^n_{34}(t)|=0.
		\end{align}
		
		Let $f(t)=\int_{1}^{t}1/\varrho(s)\mathrm{d}s$. By (\ref{3631})-(\ref{3634}) and applying Lemma \ref{bhr} to the inequality (\ref{y11}), we get
		\begin{align*}
			\sup\limits_{r\in[0,t]}|\varsigma_{r}|^{2}\leq f^{-1}\left(f(\sup\limits_{r\in[0,t]}|I^n_3(r)|)+CLT\right),
		\end{align*}
		similar to the proof of Lemma \ref{afin}, we obtain
		\begin{align*}
			\lim\limits_{n\to\infty}\sup_{t\in[0,T]}|Y^{h_n}_t-Y^h_t|=0,
		\end{align*}
		which is the desired result.
	\end{proof}



	To verify \textbf{(LDP)$\bm{_{2}}$}, we need the following lemma.
	
	\begin{lemma}\label{ep0}
		Under (\hyperref[h0]{\textbf{H0}}), (\hyperref[h1]{\textbf{H1}}) and (\hyperref[h2]{\textbf{H2}}), it holds that
		$$\lim\limits_{\epsilon\to 0}\mathbb{E}\left(\sup_{t\in[0,T]}\lvert X_{t}^{\epsilon}-X_{t}^{0}\rvert^{2}\right)=0.$$
	\end{lemma}
	
	\begin{proof}
		By applying It$\hat{\text{o}}$'s formula we have
		\begin{align*}
			&\lvert X_{t}^{\epsilon}-X_{t}^{0}\rvert^{2} \\
			=&2\int_{0}^{t}\left\langle b_{\epsilon}\left(X_{r}^{\epsilon},\mathcal{L}_{X_{r}^{\epsilon}}\right)-b\left(X_{r}^{0},\mathcal{L}_{X_{r}^{0}}\right), X_{r}^{\epsilon}-X_{r}^{0}\right\rangle\mathrm{d}r
			+2\sqrt{\epsilon}\int_{0}^{t}\left\langle\sigma_{\epsilon}\left(X_{r}^{\epsilon},\mathcal{L}_{X_{r}^{\epsilon}}\right), X_{r}^{\epsilon}-X_{r}^{0}\right\rangle\mathrm{d}W(s) \\
			&+\epsilon\int_{0}^{t}\lVert\sigma_{\epsilon}\left(X_{r}^{\epsilon},\mathcal{L}_{X_{r}^{\epsilon}}\right)\rVert^{2}\mathrm{d}r+\int_{0}^{t}\left\langle X_{r}^{\epsilon}-X_{r}^{0},\mathrm{d}K_{r}^{0}-\mathrm{d}K_{r}^{\epsilon}\right\rangle \\
			=:&J_{1}^{\epsilon}(t)+J_{2}^{\epsilon}(t)+J_{3}^{\epsilon}(t)+J_{4}^{\epsilon}(t).
		\end{align*}
		Next we estimate the four items above.
		
		For $J_{1}^{\epsilon}(t)$,
		by Jensen's inequality,  (\hyperref[h1]{\textbf{H1}}),(\hyperref[h2]{\textbf{H2}}) and Remark \ref{W2},
		\begin{align*}
			\sup_{t\in [0,T]}\lvert J_{1}^{\epsilon}(t)\rvert
			\leq& \,2\int_{0}^{T}\lvert\left\langle X_{r}^{\epsilon}-X_{r}^{0},  b_{\epsilon}\left(X_{r}^{\epsilon},\mathcal{L}_{X_{r}^{\epsilon}}\right)-b\left(X_{r}^{\epsilon},\mathcal{L}_{X_{r}^{\epsilon}}\right)\right\rangle\rvert\mathrm{d}r \\
			&+2\int_{0}^{T}\lvert\left\langle X_{r}^{\epsilon}-X_{r}^{0},  b\left(X_{r}^{\epsilon},\mathcal{L}_{X_{r}^{\epsilon}}\right)-b\left(X_{r}^{0},\mathcal{L}_{X_{r}^{0}}\right)\right\rangle\rvert\mathrm{d}r \\
			\leq& \,2\rho_{b,\epsilon}\int_{0}^{T}\lvert X_{r}^{\epsilon}-X_{r}^{0}\rvert\mathrm{d}r+2L\int_{0}^{T}\varrho\left[\lvert X_{r}^{\epsilon}-X_{r}^{0}\rvert^{2}+\mathbb{W}_2^{2}\left(\mathcal{L}_{X_{r}^{\epsilon}},\mathcal{L}_{X_{r}^{0}}\right)\right]\mathrm{d}r  \\
			\leq& \,\rho_{b,\epsilon}\int_{0}^{T}\lvert X_{r}^{\epsilon}-X_{r}^{0}\rvert^{2}\,\mathrm{d}r+2L\int_{0}^{T}\varrho(|X_r^\epsilon-X_r^0|^2)\mathrm{d}r+2L\int_{0}^{T}
			\mathbb{E}\left[\varrho(\lvert X_{r}^{\epsilon}-X_{r}^{0}\rvert^{2})\right]\,\mathrm{d}r+T\rho_{b,\epsilon}.
		\end{align*}
		
		Thus we have
		\begin{align}\label{451}
			\mathbb{E}\left(\sup_{t\in [0,T]}\lvert J_{1}^{\epsilon}(t)\rvert\right)\leq T\rho_{b,\epsilon}\mathbb{E}\left(\sup_{r\in[0,T]}\lvert X_{r}^{\epsilon}-X_{r}^{0}\rvert^{2}\right)+4L\mathbb{E}\int_{0}^{T}\varrho(\sup_{s\in [0,r]}\lvert X_{s}^{\epsilon}-X_{s}^{0}\rvert^{2})\,\mathrm{d}r+T\rho_{b,\epsilon}.
		\end{align}
		
		For $J_{3}^{\epsilon}(t)$,
		since $\varrho(\cdot)$ is concave and increasing, there exists a positive number $a$ such that
		\begin{align}\label{lin}
			\varrho(u)\leq a(1+u).
		\end{align}
		By (\hyperref[h1]{\textbf{H1}}), (\hyperref[h2]{\textbf{H2}}) and Remark \ref{W2} again, we get
		\begin{align}\label{453}
			\mathbb{E}\left(\sup_{t\in [0,T]}\lvert J_{3}^{\epsilon}(t)\rvert\right)
			=&\, \epsilon\,\mathbb{E}\int_{0}^{T}\lVert\sigma_{\epsilon}\left(X_{r}^{\epsilon},\mathcal{L}_{X_{r}^{\epsilon}}\right)\rVert^{2}\,\mathrm{d}r \nonumber\\
			\leq&\, C\epsilon\,\mathbb{E}\int_{0}^{T}\lVert\sigma_{\epsilon}\left(X_{r}^{\epsilon},\mathcal{L}_{X_{r}^{\epsilon}}\right)-\sigma\left(X_{r}^{\epsilon},\mathcal{L}_{X_{r}^{\epsilon}}\right)\rVert^{2}\,\mathrm{d}r +C\epsilon\int_{0}^{T}\lVert\sigma\left(X_{r}^{0},\mathcal{L}_{X_{r}^{0}}\right)\rVert^{2}\,\mathrm{d}r \nonumber\\
			&+\, C\epsilon\,\mathbb{E}\int_{0}^{T}\lVert\sigma\left(X_{r}^{\epsilon},\mathcal{L}_{X_{r}^{\epsilon}}\right)-\sigma\left(X_{r}^{0},\mathcal{L}_{X_{r}^{0}}\right)\rVert^{2}\,\mathrm{d}r \nonumber\\
			\leq&\, CT\epsilon\rho_{\sigma,\epsilon}^{2}+CL\epsilon\,\mathbb{E}\int_{0}^{T}\varrho\left[\lvert X_{r}^{\epsilon}-X_{r}^{0}\rvert^{2}+\mathbb{W}_2^{2}\left(\mathcal{L}_{X_{r}^{\epsilon}},\mathcal{L}_{X_{r}^{0}}\right)\right]\mathrm{d}r \nonumber\\
			&+\,C\epsilon\int_{0}^{T}\lVert\sigma\left(X_{r}^{0},\mathcal{L}_{X_{r}^{0}}\right)\rVert^{2}\,\mathrm{d}r \nonumber\\
			\leq&\, CT\epsilon\rho_{\sigma,\epsilon}^{2}+CL\epsilon \int_{0}^{T}\,\mathbb{E}\left[\varrho(\sup_{s\in[0,r]}\lvert X_{s}^{\epsilon}-X_{s}^{0}\rvert^{2})\right]\mathrm{~d}r +\,CLT\epsilon\nonumber\\
			\leq&\, CT\epsilon\rho_{\sigma,\epsilon}^2+\,CLT\epsilon \,\mathbb{E}\left(\sup_{r\in[0,T]}\lvert X_{r}^{\epsilon}-X_{r}^{0}\rvert^{2}\right)+\,CLT(1+a)\epsilon.
		\end{align}
		
		For $J_{2}^{\epsilon}(t)$,
		by BDG's inequality and Young's inequality, we have
		\begin{align}\label{452}
			\mathbb{E}\left(\sup_{t\in [0,T]}\lvert J_{2}^{\epsilon}(t)\rvert\right)
			\leq& C\sqrt{\epsilon}\,\mathbb{E}\left[\int_{0}^{T}\lVert\sigma_{\epsilon}\left(X_{r}^{\epsilon},\mathcal{L}_{X_{r}^{\epsilon}}\right)\rVert^{2}\cdot\lvert X_{r}^{\epsilon}-X_{r}^{0}\rvert^{2}\,\mathrm{d}r\right]^{\frac{1}{2}} \nonumber\\
			\leq& \frac{1}{4}\,\mathbb{E}\left(\sup_{r\in[0,T]}\lvert X_{r}^{\epsilon}-X_{r}^{0}\rvert^{2}\right)+C\epsilon\,\mathbb{E}\int_{0}^{T}\lVert\sigma_{\epsilon}\left(X_{r}^{\epsilon},\mathcal{L}_{X_{r}^{\epsilon}}\right)\rVert^{2}\,\mathrm{d}r \nonumber\\
			\leq& \frac{1}{4}\,\mathbb{E}\left(\sup_{r\in[0,T]}\lvert X_{r}^{\epsilon}-X_{r}^{0}\rvert^{2}\right)+CL\epsilon\int_{0}^{T}\mathbb{E}\left[\varrho(\sup_{s\in[0,r]}\lvert X_{s}^{\epsilon}-X_{s}^{0}\rvert^{2})\right]\mathrm{d}r+CT\epsilon\rho_{\sigma,\epsilon}^{2} \nonumber\\
			\leq& \left(\frac{1}{4}+CLT\epsilon\right)\,\mathbb{E}\left(\sup_{r\in[0,T]}\lvert X_{r}^{\epsilon}-X_{r}^{0}\rvert^{2}\right)+CT\epsilon\rho_{\sigma,\epsilon}^{2}+\,CLT\epsilon.
		\end{align}
		
		For $J_{4}^{\epsilon}(t)$, by (\ref{mono}) we have
		\begin{align}\label{454}
			\sup_{t\in [0,T]} J_{4}^{\epsilon}(t)\leq 0.
		\end{align}

		Combining (\ref{451})-(\ref{454}) together, we obtain
		\begin{align*}
			&\left(\frac{3}{4}-CLT\epsilon-T\rho_{b,\epsilon}\right)\,\,\mathbb{E}\left(\sup_{r\in[0,T]}\lvert X_{r}^{\epsilon}-X_{r}^{0}\rvert^{2}\right) \\
			\leq& CL\mathbb{E}\int_{0}^{T}\varrho(\sup_{s\in[0,r]}\lvert X_{s}^{\epsilon}-X_{s}^{0}\rvert^{2})\,\mathrm{d}r+T\rho_{b,\epsilon}+CT\epsilon\rho_{\sigma,\epsilon}^{2} \\
			=:&CL\mathbb{E}\int_{0}^{T}\varrho(\sup_{s\in[0,r]}\lvert X_{s}^{\epsilon}-X_{s}^{0}\rvert^{2})\,\mathrm{d}r+O_2(\epsilon).
		\end{align*}
		By (\hyperref[h2]{\textbf{H2}}) we can deduce that
		\begin{align*}
			\lim\limits_{\epsilon\to0}O_2(\epsilon)=0.
		\end{align*}
		
		Since $X^{0}\in C\left([0,T],\mathbb{R}^{d}\right)$ and  by (\hyperref[h1]{\textbf{H1}}),  there exists $\epsilon_{0}>0$ small enough such that for any $\epsilon\in \left(0,\epsilon_{0}\right]$,
		$$\frac{3}{4}-CLT\epsilon-T\rho_{b,\epsilon}\geq\frac{1}{4}.$$
		By Lemma \ref{bhr}, we have
		\begin{align*}
			\mathbb{E}\left(\sup_{r\in[0,T]}\lvert X_{r}^{\epsilon}-X_{r}^{0}\rvert^{2}\right)\leq f^{-1}\left(f(O_2(\epsilon))+CLT\right),
		\end{align*}
		where $f(t)=\int_{1}^{t}1/\varrho(s)\mathrm{d}s$. Similar to the proof of Proposition \ref{ld1}, we obtain
		$$\lim\limits_{\epsilon\to 0}\mathbb{E}\left(\sup_{t\in[0,T]}\lvert X_{t}^{\epsilon}-X_{t}^{0}\rvert^{2}\right)=0,$$
		which is the desired result.
	\end{proof}

	Now we are in the position to verify  \textbf{(LDP)$\bm{_{2}}$}.
	
	\begin{proposition}\label{ld2}[\textbf{(LDP)$\bm{_{2}}$}]
		For any given $m\in\left(0,\infty\right)$, let $\left\{h_{\epsilon}, \epsilon>0\right\}\subset \mathcal{D}_{m}$, then
		$$
		\lim\limits_{\epsilon\to 0}\mathbb{E}\left(\sup\limits_{t\in \left[0,T\right]}\lvert Z_{t}^{\epsilon,h_{\epsilon}}-\mathcal{G}^{0}\left(h_{\epsilon}\right)\left(t\right)\rvert^{2}\right)=0.
		$$
		\begin{proof}
			Let $Y_{t}^{h_{\epsilon}}$ be the solution of (\ref{yy}) with $h$ replaced by $h_{\epsilon}$, then $\mathcal{G}^{0}(h_{\epsilon})=Y^{h_{\epsilon}}$.
			Since
			\begin{align}
				\mathbb{E}\sup_{t\in[0,T]}|Z^{\epsilon,h_\epsilon}_t-Y^{h_\epsilon}_t|^2\leq &\,3\,\mathbb{E}\sup_{t\in[0,T]}|Z^{\epsilon,h_\epsilon}_t-Z^{\epsilon,h_\epsilon,\alpha}_t|^2 \\
				&\,+3\,\mathbb{E}\sup_{t\in[0,T]}|Z^{\epsilon,h_\epsilon,\alpha}_t-Y^{h_\epsilon,\alpha}_t|^2 \\
				&\,+3\,\mathbb{E}\sup_{t\in[0,T]}|Y^{h_\epsilon,\alpha}_t-Y^{h_\epsilon}_t|^2.
			\end{align}
			
			By Lemma \ref{afin} and Lemma \ref{ayfin}, for any $\delta>0$, there exists some constant $\alpha$ independent of $\epsilon$ such that
			\begin{equation}\label{fin21}
				\mathbb{E}\sup_{t\in[0,T]}|Z^{\epsilon,h_\epsilon}_t-Z^{\epsilon,h_\epsilon,\alpha}|^2\leq\frac{\delta}{9}
			\end{equation}
			and
			\begin{equation}\label{fin23}
				\mathbb{E}\sup_{t\in[0,T]}|Y^{h_\epsilon,\alpha}_t-Y^{h_\epsilon}_t|^2\leq\frac{\delta}{9}.
			\end{equation}
			
			For this $\alpha$, we have
			\begin{align*}
				\phi_{t}^{\epsilon,h_{\epsilon},\alpha}\coloneqq&\,Z_{t}^{\epsilon,h_{\epsilon},\alpha}-Y_{t}^{h_{\epsilon},\alpha}\\
				=&\int_{0}^{t}\left[A^\alpha(Y^{h_\epsilon,\alpha}_r)-A_\epsilon^\alpha(Z_{r}^{\epsilon,h_{\epsilon},\alpha})\right]\mathrm{d}r +\int_{0}^{t}\left[b_{\epsilon}\left(Z_{r}^{\epsilon,h_{\epsilon},\alpha},\mathcal{L}_{X_{r}^{\epsilon}}\right)-b\left(Y_{r}^{h_{\epsilon},\alpha},\mathcal{L}_{X_{r}^{0}}\right)\right]\,\mathrm{d}r \\
				&+\int_{0}^{t}\left[\sigma_{\epsilon}\left(Z_{r}^{\epsilon,h_{\epsilon},\alpha},\mathcal{L}_{X_{r}^{\epsilon}}\right)h_{\epsilon}(r)-\sigma\left(Y_{r}^{h_{\epsilon},\alpha},\mathcal{L}_{X_{r}^{0}}\right)h_{\epsilon}(r)\right]\,\mathrm{d}r +\sqrt{\epsilon}\int_{0}^{t}\sigma_{\epsilon}\left(Z_{r}^{\epsilon,h_{\epsilon},\alpha},\mathcal{L}_{X_{r}^{\epsilon}}\right)\,\mathrm{d}W_r.
			\end{align*}
			By It$\hat{\text{o}}$'s formula we get
			\begin{align*}
				\lvert\phi_{t}^{\epsilon,h_{\epsilon},\alpha}\rvert^{2}=&\,2\int_{0}^{t}\left\langle \phi_{r}^{\epsilon,h_{\epsilon},\alpha},A^\alpha(Y_r^{h_\epsilon,\alpha})-A^\alpha(Z_{r}^{\epsilon,h_{\epsilon},\alpha})\right\rangle\mathrm{d}r  \\
				&+2\int_{0}^{t}\left\langle \phi_{r}^{\epsilon,h_{\epsilon},\alpha},A^\alpha(Z_{r}^{\epsilon,h_{\epsilon},\alpha})-A_\epsilon^\alpha(Z_{r}^{\epsilon,h_{\epsilon},\alpha})\right\rangle\mathrm{d}r  \\ &+2\int_{0}^{t}\left\langle\phi_{r}^{\epsilon,h_{\epsilon},\alpha},b_{\epsilon}\left(Z_{r}^{\epsilon,h_{\epsilon},\alpha},\mathcal{L}_{X_{r}^{\epsilon}}\right)-b\left(Y_{r}^{h_{\epsilon},\alpha},\mathcal{L}_{X_{r}^{0}}\right)\right\rangle\,\mathrm{d}r \\
				&+2\int_{0}^{t}\left\langle\phi_{r}^{\epsilon,h_{\epsilon},\alpha},\sigma_{\epsilon}\left(Z_{r}^{\epsilon,h_{\epsilon},\alpha},\mathcal{L}_{X_{r}^{\epsilon}}\right)h_{\epsilon}(r)-\sigma\left(Y_{r}^{h_{\epsilon},\alpha},\mathcal{L}_{X_{r}^{0}}\right)h_{\epsilon}(r)\right\rangle\,\mathrm{d}r \\
				&+2\sqrt{\epsilon}\int_{0}^{t}\left\langle\phi_{r}^{\epsilon,h_{\epsilon},\alpha},\sigma_{\epsilon}\left(Z_{r}^{\epsilon,h_{\epsilon},\alpha},\mathcal{L}_{X_{r}^{\epsilon}}\right)\,\mathrm{d}W_r\right\rangle +\epsilon\int_{0}^{t}\lVert \sigma_{\epsilon}\left(Z_{r}^{\epsilon,h_{\epsilon},\alpha},\mathcal{L}_{X_{r}^{\epsilon}}\right)\rVert^{2}\mathrm{d}r \\
				=:&I_{1}^{\epsilon,\alpha}(t)+I_{2}^{\epsilon,\alpha}(t)+I_{3}^{\epsilon,\alpha}(t)+I_{4}^{\epsilon,\alpha}(t)+I_{5}^{\epsilon,\alpha}(t)+I_{6}^{\epsilon,\alpha}(t).
			\end{align*}
			
			Next we estimate the six items in the equality above.
			
			(1) $I_{1}^{\epsilon,\alpha}(t)$.
			By (\ref{mono}) and the monotonicity of $A^\alpha$ we immediately get
			\begin{align}\label{421}
				I_{1}^{\epsilon,\alpha}(t)\leq 0.
			\end{align}
			
			(2) $I_{2}^{\epsilon,\alpha}(t)$. For any $\tau>0$ we have
			\begin{align}\label{422}
				\sup_{t\in [0,T]}\lvert I_{2}^{\epsilon,\alpha}(t)\rvert\leq \tau\mathbb{E}\left(\sup_{t\in[0,T]}|\phi_{t}^{\epsilon,h_\epsilon,\alpha}|^2\right)+C_\tau\int_{0}^{T}\mathbb{E}|A^\alpha(Z^{\epsilon,h_\epsilon,\alpha}_t)-A^\alpha_\epsilon(Z^{\epsilon,h_\epsilon,\alpha}_t)|^2\mathrm{d}t.
			\end{align}
			Next we show that as $\epsilon\to 0$
			\begin{equation}\label{Psii}
				\Psi_\epsilon:=\int_{0}^{T}\mathbb{E}|A^\alpha(Z^{\epsilon,h_\epsilon,\alpha}_t)-A^\alpha_\epsilon(Z^{\epsilon,h_\epsilon,\alpha}_t)|^2\mathrm{d}t\to 0.
			\end{equation}
			For any $j>0$, we have
			\begin{align*}
				\Psi_\epsilon=&\int_{0}^{T}\mathbb{E}\left[|A^\alpha(Z^{\epsilon,h_\epsilon,\alpha}_t)-A^\alpha_\epsilon(Z^{\epsilon,h_\epsilon,\alpha}_t)|^2\mathbb{I}_{\left\{\sup_{t\in [0,T]}|Z^{\epsilon,h_\epsilon,\alpha}_t|\leq j\right\}}\right]\mathrm{d}t \\
				&+\int_{0}^{T}\mathbb{E}\left[|A^\alpha(Z^{\epsilon,h_\epsilon,\alpha}_t)-A^\alpha_\epsilon(Z^{\epsilon,h_\epsilon,\alpha}_t)|^2\mathbb{I}_{\left\{\sup_{t\in [0,T]}|Z^{\epsilon,h_\epsilon,\alpha}_t|> j\right\}}\right]\mathrm{d}t \\
				=:&J_1^{\epsilon,j}+J_2^{\epsilon,j}.
			\end{align*}
			
			By \hyperref[h4]{\textbf{(H4)}} we have $\lim\limits_{\epsilon\to 0}J_1^{\epsilon,j}=0$. For $J_2^{\epsilon,j}$,			
			by Lipschitz continuity of Yosida approximation and \hyperref[h0]{(\textbf{H0})} we have
			\begin{align*}
				|A^\alpha(Z^{\epsilon,h_\epsilon,\alpha}_t)|+|A^\alpha_\epsilon(Z^{\epsilon,h_\epsilon,\alpha}_t)|\leq C_\alpha\left(1+|Z^{\epsilon,h_\epsilon,\alpha}_t|\right),
			\end{align*}
			
			and hence
			\begin{align*}
				J_2^{\epsilon,j}\leq &C_\alpha\mathbb{E}\left[\left(1+\sup_{t\in[0,T]}|Z^{\epsilon,h_\epsilon,\alpha}_t|^2\right)\mathbb{I}_{\left\{\sup_{t\in [0,T]}|Z^{\epsilon,h_\epsilon,\alpha}_t|> j\right\}}\right] \\
				\leq&C_\alpha\mathbb{E}\left[\left(1+\sup_{t\in[0,T]}|Z^{\epsilon,h_\epsilon,\alpha}_t|^2\right)\frac{\sup_{t\in [0,T]}|Z^{\epsilon,h_\epsilon,\alpha}_t|^2}{j^2}\right] \\
				\leq&C_\alpha\sup_{\epsilon}\frac{\mathbb{E}\left[\sup_{t\in [0,T]}|Z^{\epsilon,h_\epsilon,\alpha}_t|^2+\sup_{t\in [0,T]}|Z^{\epsilon,h_\epsilon,\alpha}_t|^4\right]}{j^2}\leq C_\alpha \cdot C/j^2,
			\end{align*}
			where the last inequality follows from Lemma \ref{yozfin} and $C$ is independent of $\epsilon$ and $\alpha$. This implies $\lim\limits_{j\to\infty}J_2^{\epsilon,j}=0$ for any $\epsilon$, and therefore (\ref{Psii}).
			
			(3) $I_{3}^{\epsilon,\alpha}(t)$.
			By Assumptions (\hyperref[h1]{\textbf{H1}}) and Remark \ref{W2}, we have
			\begin{align*}
				\sup_{t\in [0,T]}\lvert I_{3}^{\epsilon,\alpha}(t)\rvert
				\leq& \,2\int_{0}^{T}\lvert\left\langle \phi_{r}^{\epsilon,h_{\epsilon},\alpha},  b_{\epsilon}\left(Z_{r}^{\epsilon,h_{\epsilon},\alpha},\mathcal{L}_{X_{r}^{\epsilon}}\right)-b\left(Y_{r}^{h_{\epsilon},\alpha},\mathcal{L}_{X_{r}^{0}}\right)\right\rangle\rvert\mathrm{d}r \\
				\leq& \,2\int_{0}^{T}\lvert\left\langle \phi_{r}^{\epsilon,h_{\epsilon},\alpha},  b_{\epsilon}\left(Z_{r}^{\epsilon,h_{\epsilon},\alpha},\mathcal{L}_{X_{r}^{\epsilon}}\right)-b\left(Z_{r}^{\epsilon,h_{\epsilon},\alpha},\mathcal{L}_{X_{r}^{\epsilon}}\right)\right\rangle\rvert\mathrm{d}r \\
				&+2\int_{0}^{T}\lvert\left\langle \phi_{r}^{\epsilon,h_{\epsilon},\alpha},  b\left(Z_{r}^{\epsilon,h_{\epsilon},\alpha},\mathcal{L}_{X_{r}^{\epsilon}}\right)-b\left(Y_{r}^{h_{\epsilon},\alpha},\mathcal{L}_{X_{r}^{0}}\right)\right\rangle\rvert\mathrm{d}r \\
				\leq&\,2\rho_{b,\epsilon}\int_{0}^{T}\lvert\phi_{r}^{\epsilon,h_{\epsilon},\alpha}\rvert\mathrm{d}r+2L\int_{0}^{T}\varrho(\lvert\phi_{r}^{\epsilon,h_{\epsilon},\alpha}\rvert^{2})\mathrm{d}r+2L\int_{0}^{T}\varrho\left(\mathbb{W}^2_{2}(\mathcal{L}_{X_{r}^{\epsilon}},\mathcal{L}_{X_{r}^{0}})\right)\mathrm{d}r \\
				\leq&\,T\rho_{b,\epsilon}+\rho_{b,\epsilon}
				\int_{0}^{T}\lvert\phi_{r}^{\epsilon,h_{\epsilon},\alpha}\rvert^{2}\mathrm{d}r+2L\int_{0}^{T}\varrho(\lvert\phi_{r}^{\epsilon,h_{\epsilon},\alpha}\rvert^{2})\mathrm{d}r+2L\int_{0}^{T}\mathbb{E}\left[\varrho\left(\lvert X_{r}^{\epsilon}-X_{r}^{0}\rvert^{2}\right)\right]\mathrm{d}r \\
				\leq&T\rho_{b,\epsilon}+\,T\rho_{b,\epsilon}\sup_{r\in [0,T]}\lvert\phi_{r}^{\epsilon,h_{\epsilon},\alpha}\rvert^{2}+2L\int_{0}^{T}\varrho(\lvert\phi_{r}^{\epsilon,h_{\epsilon},\alpha}\rvert^{2})\mathrm{d}r+LT\mathbb{E}\left[\varrho(\sup_{r\in[0,T]}\lvert X_{r}^{\epsilon}-X_{r}^{0}\rvert^{2})\right].	
			\end{align*}
			Hence
			\begin{align}\label{423}
				\mathbb{E}\left(\sup_{r \in[0, T]}\lvert I_{3}^{\epsilon,\alpha}(t)\rvert\right)\leq& T\rho_{b,\epsilon}+\,T\rho_{b,\epsilon}\mathbb{E}\left(\sup_{r\in [0,T]}\lvert\phi_{r}^{\epsilon,h_{\epsilon},\alpha}\rvert^{2}\right)+2L\int_{0}^{T}\mathbb{E}\left[\varrho(\sup_{s\in[0,r]}\lvert\phi_{s}^{\epsilon,h_{\epsilon},\alpha}\rvert^{2})\right]\mathrm{d}r \nonumber\\
				&+LT\mathbb{E}\left[\varrho(\sup_{r\in[0,T]}\lvert X_{r}^{\epsilon}-X_{r}^{0}\rvert^{2})\right].
			\end{align}
			
			(4) $I_{4}^{\epsilon,\alpha}(t)$. By (\hyperref[h2]{\textbf{H2}}),  we have
			\begin{align*}
				&\mathbb{E}\left(\sup_{t\in [0,T]}\lvert I_{4}^{\epsilon,\alpha}(t)\rvert\right) \nonumber\\
				\leq& \, 2\mathbb{E}\int_{0}^{T}\lvert\left\langle \phi_{r}^{\epsilon,h_{\epsilon},\alpha}, \left[\sigma_{\epsilon}\left(Z_{r}^{\epsilon,h_{\epsilon},\alpha},\mathcal{L}_{X_{r}^{\epsilon}}\right)-\sigma\left(Z_{r}^{\epsilon,h_{\epsilon},\alpha},\mathcal{L}_{X_{r}^{\epsilon}}\right)\right]h_{\epsilon}(r)\right\rangle\rvert\,\mathrm{d}r \nonumber\\
				&+ 2\mathbb{E}\int_{0}^{T}\lvert\left\langle \phi_{r}^{\epsilon,h_{\epsilon},\alpha}, \left[\sigma\left(Z_{r}^{\epsilon,h_{\epsilon},\alpha},\mathcal{L}_{X_{r}^{\epsilon}}\right)-\sigma\left(Y_{r}^{h_{\epsilon},\alpha},\mathcal{L}_{X_{r}^{0}}\right)\right]h_{\epsilon}(r)\right\rangle\rvert\,\mathrm{d}r  \nonumber\\
				\leq& \,2\rho_{\sigma,\epsilon}\mathbb{E}\int_{0}^{T}\lvert\phi_{r}^{\epsilon,h_{\epsilon},\alpha}\rvert\lvert h_{\epsilon}(r)\rvert\,\mathrm{d}r+2\mathbb{E}\int_{0}^{T}\lvert\phi_{r}^{\epsilon,h_{\epsilon},\alpha}\rvert\cdot\lVert\sigma\left(Z_{r}^{h_{\epsilon}},\mathcal{L}_{X_{r}^{\epsilon}}\right)-\sigma\left(Y_{r}^{h_{\epsilon}},\mathcal{L}_{X_{r}^{0}}\right)\rVert\cdot\lvert h_{\epsilon}(r)\rvert\,\mathrm{d}r \nonumber\\
				\leq& \, T\rho_{\sigma,\epsilon}\mathbb{E}\left(\sup_{r\in [0,T]}\lvert\phi_{r}^{\epsilon,h_{\epsilon},\alpha}\rvert^{2}\right)+\rho_{\sigma,\epsilon}\int_{0}^{T}\lvert h_{\epsilon}(r)\rvert^{2}\,\mathrm{d}r \nonumber\\
				&+2\mathbb{E}\left(\sup_{r\in[0,T]}\lvert\phi_{r}^{\epsilon,h_{\epsilon},\alpha}\rvert\int_{0}^{T}\lVert\sigma\left(Z_{r}^{\epsilon,h_{\epsilon},\alpha},\mathcal{L}_{X_{r}^{\epsilon}}\right)-\sigma\left(Y_{r}^{h_{\epsilon},\alpha},\mathcal{L}_{X_{r}^{0}}\right)\rVert\cdot|h_\epsilon|\,\mathrm{d}r\right).
			\end{align*}
			Furthermore, for any $\tau>0$, by Young's inequality we have
			\begin{align}\label{424}
				\mathbb{E}\left(\sup_{t\in [0,T]}\lvert I_{4}^{\epsilon,\alpha}(t)\rvert\right)\leq& \, T\rho_{\sigma,\epsilon}\mathbb{E}\left(\sup_{r\in [0,T]}\lvert\phi_{r}^{\epsilon,h_{\epsilon},\alpha}\rvert^{2}\right)+m\rho_{\sigma,\epsilon}+\tau\mathbb{E}\left(\sup_{r\in[0,T]}\lvert\phi_{r}^{\epsilon,h_{\epsilon},\alpha}\rvert^{2}\right) \nonumber\\
				&+C\int_{0}^{T}\left(\mathbb{E}(\varrho(\sup_{s\in[0,r]}\lvert\phi_{s}^{\epsilon,h_{\epsilon},\alpha}\rvert^{2}))+\mathbb{E}\left[\varrho(\lvert X_{r}^{\epsilon}-X_{r}^{0}\rvert^{2})\right]\right)\,\mathrm{d}r \nonumber\\
				=&(\tau+T\rho_{\sigma,\epsilon})\mathbb{E}\left(\sup_{r\in[0,T]}\lvert\phi_{r}^{\epsilon,h_{\epsilon},\alpha}\rvert^{2}\right)+C\int_{0}^{T}\mathbb{E}\left[\varrho(\sup_{s\in[0,r]}\lvert\phi_{s}^{\epsilon,h_{\epsilon},\alpha}\rvert^{2})\right]\mathrm{d}r \nonumber\\
				&+C\mathbb{E}\left[\varrho(\sup_{r\in[0,T]}\lvert X_{r}^{\epsilon}-X_{r}^{0}\rvert^{2})\right]+m\rho_{\sigma,\epsilon},
			\end{align}
			where $C$ depends on $m,T,\tau$ and $L$.
			
			(5) $I_{5}^{\epsilon,\alpha}(t)$ and $I_{6}^{\epsilon,\alpha}(t)$. By (\hyperref[h2]{\textbf{H2}}), Young's inequlity and (\ref{lin}) we have
			\begin{align}\label{425}
				&\mathbb{E}\left(\sup_{t\in [0,T]}\lvert I_{5}^{\epsilon,\alpha}(t)\rvert\right)+\mathbb{E}\left(\sup_{t\in [0,T]}\lvert I_{6}^{\epsilon,\alpha}(t)\rvert\right) \nonumber\\
				\leq&\, C\sqrt{\epsilon}\,\mathbb{E}\left[\int_{0}^{T}\lVert\sigma_{\epsilon}\left(Z_{r}^{\epsilon,h_{\epsilon},\alpha},\mathcal{L}_{X_{r}^{\epsilon}}\right)\rVert^{2}\cdot\lvert\phi_{r}^{\epsilon,h_{\epsilon},\alpha}\rvert^{2}\,\mathrm{d}r\right]^{\frac{1}{2}}+\epsilon\,\mathbb{E}\int_{0}^{T}\lVert\sigma_{\epsilon}\left(Z_{r}^{\epsilon,h_{\epsilon},\alpha},\mathcal{L}_{X_{r}^{\epsilon}}\right)\rVert^{2}\,\mathrm{d}r \nonumber\\
				\leq& \,\tau\mathbb{E}\left(\sup_{r\in [0,T]}\lvert\phi_{r}^{\epsilon,h_{\epsilon},\alpha}\rvert^{2}\right)+C\epsilon\,\mathbb{E}\int_{0}^{T}\lVert\sigma_{\epsilon}\left(Z_{r}^{\epsilon,h_{\epsilon},\alpha},\mathcal{L}_{X_{r}^{\epsilon}}\right)\rVert^{2}\,\mathrm{d}r \nonumber\\
				\leq& \, \tau\mathbb{E}\left(\sup_{r\in [0,T]}\lvert\phi_{r}^{\epsilon,h_{\epsilon},\alpha}\rvert^{2}\right)+C\epsilon\,\mathbb{E}\int_{0}^{T}\lVert\sigma_{\epsilon}\left(Z_{r}^{\epsilon,h_{\epsilon},\alpha},\mathcal{L}_{X_{r}^{\epsilon}}\right)-\sigma\left(Z_{r}^{\epsilon,h_{\epsilon},\alpha},\mathcal{L}_{X_{r}^{\epsilon}}\right)\rVert^{2}\,\mathrm{d}r \nonumber\\
				&+C\epsilon\,\mathbb{E}\int_{0}^{T}\lVert\sigma\left(Z_{r}^{\epsilon,h_{\epsilon},\alpha},\mathcal{L}_{X_{r}^{\epsilon}}\right)-\sigma\left(Y_{r}^{h_{\epsilon},\alpha},\mathcal{L}_{X_{r}^{0}}\right)\rVert^{2}\,\mathrm{d}r+C\epsilon\,\mathbb{E}\int_{0}^{T}\lVert\sigma\left(Y_{r}^{h_{\epsilon},\alpha},\mathcal{L}_{X_{r}^{0}}\right)\rVert^{2}\,\mathrm{d}r \nonumber\\
				\leq& \, \tau\mathbb{E}\left(\sup_{r\in [0,T]}\lvert\phi_{r}^{\epsilon,h_{\epsilon},\alpha}\rvert^{2}\right)+C\rho_{\sigma,\epsilon}^{2}\epsilon+C\epsilon\,\int_{0}^{T}\mathbb{E}\left[\varrho(\sup_{r\in [0,s]}\lvert\phi_{s}^{\epsilon,h_{\epsilon},\alpha}\rvert^{2})\right]\,\mathrm{d}r \nonumber\\
				&+ C\epsilon\,\int_{0}^{T}\mathbb{E}\left[\varrho(\lvert X_{r}^{\epsilon}-X_{r}^{0}\rvert^{2})\right]\,\mathrm{d}r+C\epsilon\sup_{h_{\epsilon}\in \mathcal{S}_{m}}\left(\sup_{r\in [0,T]}\lvert Y_{r}^{h_{\epsilon}}\rvert^{2}\right)+C\epsilon\sup_{r\in [0,T]}|X_{r}^{0}|^{2}+C\epsilon \nonumber\\
				\leq&\, \left(\tau+\,Ca\epsilon\right)\mathbb{E}\left(\sup_{r\in [0,T]}\lvert\phi_{r}^{\epsilon,h_{\epsilon},\alpha}\rvert^{2}\right)+Ca\epsilon\,\mathbb{E}\left(\sup_{r\in[0,T]}\lvert X_{r}^{\epsilon}-X_{r}^{0}\rvert^{2}\right)+C(1+a)\epsilon+C\rho_{\sigma,\epsilon}^{2}\epsilon,
			\end{align}
			where $C$  depends on $\tau,L,T,Y^{h_\epsilon},X^{0},m$.
			
			Taking $\tau=\frac{1}{8}$ and combining (\ref{421})-(\ref{425}) together, we get for any $\epsilon\in \left(0,\epsilon_{0}\right]$,
			\begin{align*}
				&\left(\frac{5}{8}-Ca\epsilon-T\rho_{\sigma,\epsilon}-T\rho_{b,\epsilon}\right)\mathbb{E}\left(\sup_{r\in [0,T]}\lvert\phi_{r}^{\epsilon,h_{\epsilon},\alpha}\rvert^{2}\right)    \nonumber\\
				\leq&\, CL\int_{0}^{T}\mathbb{E}\left[\varrho(\sup_{s\in[0,r]}|\phi_{s}^{\epsilon,h_\epsilon}|^2)\right]\mathrm{d}r+CTa\epsilon\,\mathbb{E}\left(\sup_{r\in[0,T]}\lvert X_{r}^{\epsilon}-X_{r}^{0}\rvert^{2}\right)+C\mathbb{E}\left[\varrho(\sup_{r\in [0,T]}\lvert X^\epsilon_r-X^0_r\rvert^2)\right]  \nonumber\\ &+C\left(\epsilon+\rho_{b,\epsilon}+\rho_{\sigma,\epsilon}+\epsilon\rho_{\sigma,\epsilon}^{2}+\Psi_\epsilon\right) \\
				=:&CL\int_{0}^{T}\mathbb{E}\left[\varrho(\sup_{s\in[0,r]}|\phi_{s}^{\epsilon,h_\epsilon}|^2)\right]\mathrm{d}r+O_3(\epsilon).
			\end{align*}
			By Lemma \ref{ep0} and (\ref{Psii}) we get
			\begin{align*}
				\lim\limits_{\epsilon\to 0}O_3(\epsilon)=0.
			\end{align*}
			Since there exists a constant $\epsilon_1\leq \epsilon_{0}$ small enough such that for any $\epsilon\in(0,\epsilon_{1}],$
			$$ \frac{5}{8}-Ca\epsilon-T\rho_{\sigma,\epsilon}-T\rho_{b,\epsilon}\geq \frac{1}{4}. $$
			Setting $f(t)=\int_{1}^{t}1/\varrho(s)\mathrm{d}s$, then by Lemma \ref{bhr}, we obtain
			\begin{align*}
				\mathbb{E}\left(\sup_{t\in[0,T]}\lvert Z_{t}^{\epsilon,h_{\epsilon},\alpha}-Y_{t}^{h_{\epsilon},\alpha}\rvert^{2}\right)\leq f^{-1}\left(f(O_3(\epsilon)+CLT)\right).
			\end{align*}
			Since $f$ and $f^{-1}$ are strictly increasing functions and $\int_{0^+}1/\varrho(s)\mathrm{d}s=\infty$, we have
			$$\lim\limits_{\epsilon\to 0}\mathbb{E}\left(\sup_{t\in[0,T]}\lvert Z_{t}^{\epsilon,h_{\epsilon},\alpha}-Y_{t}^{h_{\epsilon},\alpha}\rvert^{2}\right)=0.$$
			
			Hence for fixed $\alpha$ we can choose $\epsilon$ such that
			\begin{align}\label{fin22}
				\mathbb{E}\left(\sup_{t\in[0,T]}\lvert Z_{t}^{\epsilon,h_{\epsilon},\alpha}-Y_{t}^{h_{\epsilon},\alpha}\rvert^{2}\right)\leq\frac{\delta}{9}.
			\end{align}
			By taking (\ref{fin21}) and (\ref{fin23}) into account, we get
			$$\lim\limits_{\epsilon\to 0}\mathbb{E}\left(\sup_{t\in[0,T]}\lvert Z_{t}^{\epsilon,h_{\epsilon}}-Y_{t}^{h_{\epsilon}}\rvert^{2}\right)=0,$$
			which is the desired result.
			
		\end{proof}
	\end{proposition}

	\subsection{An Example}
	Consider the following small perturbation of a Multivalued McKean-Vlasov SDE:
	\begin{equation}
		\begin{cases}
			\mathrm{d}X^\epsilon_t\in b_\epsilon(X^\epsilon_t,\mathcal{L}_{X^\epsilon_t})\mathrm{d}t+\sigma_\epsilon(X^\epsilon_t,\mathcal{L}_{X^\epsilon_t})\mathrm{d}W_t-\partial\Phi_\epsilon(X^\epsilon_t)\mathrm{d}t, \\
			X^\epsilon_0=h\in\overline{D(\Phi)},\quad \epsilon\in(0,1],
		\end{cases}
	\end{equation}
	where $\Phi_\epsilon$ and $\Phi$ are lower semicontinuous convex functions with common domain $D$ whose interior is non-empty. Suppose $\lim_{\epsilon\downarrow0}\Phi_\epsilon(x)=\Phi(x)$ for every $x\in D$. Define the Moreau-Yosida approximation of $\Phi$ by
	\begin{equation*}
		\Phi^\alpha(x)=\inf_{y\in{\mathbb{R}^d}}\left\{\Phi(y)+|y-x|^2/(2\alpha)\right\},
	\end{equation*}
	and that of $\Phi_\epsilon$ in the same way. Then, by \cite[Ch.\uppercase\expandafter{\romannumeral2}]{AGM}, $\Phi_\epsilon^\alpha$ and $\Phi^\alpha$ are differentiable convex functions defined on the whole space and it is not difficult to prove that
	\begin{equation*}
		\lim\limits_{\epsilon\downarrow0}\Phi_\epsilon^\alpha(x)=\Phi^\alpha(x),
	\end{equation*}
	for every $\alpha$. Let
	\begin{equation*}
		A:=\partial\Phi,\quad A_\epsilon:=\partial\Phi_\epsilon,
		A^\alpha:=\partial\Phi^\alpha,\quad A^\alpha_\epsilon:=\partial\Phi^\alpha_\epsilon,
	\end{equation*}
	then by \cite[Ch.D, Corollary 6.2.7 and Corollary 6.2.8]{JBHU}, it is easy to verify that the assumptions \hyperref[h0]{\textbf{(H0)}} and \hyperref[h4]{\textbf{(H4)}} hold in this case. Furthermore, if the assumptions \hyperref[h1]{\textbf{(H1)}}, \hyperref[h2]{\textbf{(H2)}} and \hyperref[h2]{\textbf{(H3)}} hold, then we can get the LDP by applying Theorem \ref{thmldp}.

	\subsection{Proof of Theorem 3.9}
	
	
	
	\begin{proposition}\label{} [\textbf{(MDP)}$\bm{_{1}}$]
		For any given $m\in(0,\infty)$, let $\left\{\psi_n, n\in\mathbb{N}\right\}, \psi\in\mathcal{S}_m$ satisfy that $\psi_n\to\psi$ in $\mathcal{S}_m$ as $n\to\infty$, then
		$$\lim\limits_{n\to\infty}\sup_{t\in[0,T]}|\Gamma^0(\psi_n)(t)-\Gamma^0(\psi)(t)|=0.$$
		
		\begin{proof}
			Notice that $\nu^{\psi}=\Gamma^0(\psi)$ is the solution to (\ref{mdp1}), and $\nu^{\psi_n}=\Gamma^0(\psi_n)$ is the solution to (\ref{mdp1}) with $\psi$ replaced by $\psi_n$. It is equivalent to prove the following result
			$$\lim\limits_{n\to\infty}\sup_{t\in[0,T]}|\nu^{\psi_n}(t)-\nu^\psi(t)|=0.$$
			
			The proof is similar to that of Proposition \ref{ld1}.
			By Taylor formula, we have
			\begin{align*}
				|\nu^{\psi_n}(t)-\nu^{\psi}(t)|^2=&\,2\int_{0}^{t}\langle \nu^{\psi_n}_s-\nu^{\psi}_s, \nabla b(X^0_s,\mathcal{L}_{X^0_s})(\nu^{\psi_n}_s-\nu^{\psi}_s)\rangle\mathrm{d}s \\
				&+2\,\int_{0}^{t}\langle \nu^{\psi_n}_s-\nu^{\psi}_s,\sigma(X^0_s,\mathcal{L}_{X^0_s})(\psi_n(s)-\psi(s))\rangle\mathrm{d}s \\
				&-2\int_{0}^{t}\langle \nu^{\psi_n}_s-\nu^{\psi}_s, \mathrm{d}\hat{K}^{\psi_n}_s-\mathrm{d}\hat{K}^{\psi}_s\rangle \\
				:=&\,I_1(t)+I_2(t)+I_3(t).
			\end{align*}
			
			By property (\ref{mono}), we have $\sup_{s\in[0,t]}I_3(s)\leq 0$. For $I_1(t)$, it follows from (B1) that
			\begin{align*}
				\sup_{s\in[0,t]}I_1(s)\leq\,2L_{\nabla b}\int_{0}^{t}|\nu^{\psi_n}_s-\nu^{\psi}_s|^2\mathrm{d}s.
			\end{align*}

			Combining the estimates above, we get
			\begin{align*}
				\sup_{s\in[0,t]}|\nu^{\psi_n}_s-\nu^{\psi}_s|^2\leq\,C\int_{0}^{t}|\nu^{\psi_n}_r-\nu^{\psi}_r|^2\mathrm{d}r+\sup_{r\in[0,t]}I_2(r).
			\end{align*}
			Then by Gronwall's inequality we have
			\begin{align*}
				\sup_{s\in[0,t]}|\nu^{\psi_n}_s-\nu^{\psi}_s|^2\leq \,e^{CT}\sup_{r\in[0,t]}I_2(r).
			\end{align*}
			
			To estimate $I_2(t)$, we set
			\begin{align*}
				f_n(t):=\int_{0}^{t}\sigma(X_s^0,\mathcal{L}_{X_s^0})[\psi_n(s)-\psi(s)]\mathrm{d}s.
			\end{align*}
			
			Since $\psi_n(t)$ converges to $\psi(t)$ weakly in $L^2([0,T],\mathbb{R}^d)$ as $n\to\infty$, by the similar way in Proposition \ref{ld1},  we obtain
			\begin{align*}
				\lim_{n\to\infty}\sup_{t\in[0,T]}|f_n(t)|=0.
			\end{align*}
			By applying Taylor formula to $\langle |\nu^{\psi_n}_t-\nu^{\psi}_t|,f_n(t)\rangle$  we get
			\begin{align*}
				\lim_{n\to\infty}\sup_{t\in[0,T]}|\nu^{\psi_n}_s-\nu^{\psi}_s|=0,
			\end{align*}
			which completes the proof.

		\end{proof}
	\end{proposition}

	To verify \textbf{(MDP)}$\bm{_{2}}$, we will need the following two lemmas.
	
	\begin{lemma}\label{ep2}
		There exist two positive constant $\epsilon_{0}>0$ and $C_T>0$ such that
		$$\mathbb{E}\left(\sup_{t\in[0,T]}\lvert X_{t}^{\epsilon}-X_{t}^{0}\rvert^{2}\right)\leq C_T(\epsilon+\epsilon\rho_{\sigma,\epsilon}^2+\rho_{b,\epsilon}^2),\quad \forall\epsilon\in \left(0,\epsilon_{0}\right]$$
		where $\rho_{b,\epsilon}$ and $\rho_{\sigma,\epsilon}$ are the constants given in \hyperref[h2]{\textbf{(H2)}}.
	\end{lemma}
	
	\begin{proof}
		We follow the similar way in Lemma \ref{ep0}. By applying It$\hat{\text{o}}$'s formula we have
		\begin{align*}
			&\lvert X_{t}^{\epsilon}-X_{t}^{0}\rvert^{2} \\
			=&2\int_{0}^{t}\left\langle b_{\epsilon}\left(X_{r}^{\epsilon},\mathcal{L}_{X_{r}^{\epsilon}}\right)-b\left(X_{r}^{0},\mathcal{L}_{X_{r}^{0}}\right), X_{r}^{\epsilon}-X_{r}^{0}\right\rangle\mathrm{d}r
			+2\sqrt{\epsilon}\int_{0}^{t}\left\langle\sigma_{\epsilon}\left(X_{r}^{\epsilon},\mathcal{L}_{X_{r}^{\epsilon}}\right), X_{r}^{\epsilon}-X_{r}^{0}\right\rangle\mathrm{d}W_s \\
			&+\epsilon\int_{0}^{t}\lVert\sigma_{\epsilon}\left(X_{r}^{\epsilon},\mathcal{L}_{X_{r}^{\epsilon}}\right)\rVert^{2}\mathrm{d}r+\int_{0}^{t}\left\langle X_{r}^{\epsilon}-X_{r}^{0},\mathrm{d}K_{r}^{0}-\mathrm{d}K_{r}^{\epsilon}\right\rangle \\
			=:&J_{1}^{\epsilon}(t)+J_{2}^{\epsilon}(t)+J_{3}^{\epsilon}(t)+J_{4}^{\epsilon}(t).
		\end{align*}
		
		For $J_{1}^{\epsilon}(t)$,
		by (\hyperref[b0]{\textbf{B0}}),(\hyperref[b1]{\textbf{B1}}) and Remark \ref{W2},
		\begin{align}\label{461}
			\mathbb{E}\left(\sup_{t\in [0,T]}\lvert J_{1}^{\epsilon}(t)\rvert\right)
			\leq (CL'+1)\mathbb{E}\int_{0}^{T}\lvert X_{s}^{\epsilon}-X_{s}^{0}\rvert^{2}\mathrm{d}s+\rho_{b,\epsilon}^2T.
		\end{align}
		
		For $J_{3}^{\epsilon}(t)$,
		by (\hyperref[b0]{\textbf{B0}}),(\hyperref[b3]{\textbf{B3}}) and Remark \ref{W2} again,
		\begin{align}\label{462}
			\mathbb{E}\left(\sup_{t\in [0,T]}\lvert J_{3}^{\epsilon}(t)\rvert\right)
			=&\, \epsilon\,\mathbb{E}\int_{0}^{T}\lVert\sigma_{\epsilon}\left(X_{r}^{\epsilon},\mathcal{L}_{X_{r}^{\epsilon}}\right)\rVert^{2}\,\mathrm{d}r \nonumber\\
			\leq&\, C\epsilon\rho_{\sigma,\epsilon}^2T+\,CL'T\epsilon \,\mathbb{E}\left(\sup_{r\in[0,T]}\lvert X_{r}^{\epsilon}-X_{r}^{0}\rvert^{2}\right)+\,C\epsilon.
		\end{align}
		
		For $J_{2}^{\epsilon}(t)$,
		by BDG's inequality and Young's inequality, we have
		\begin{align}\label{463}
			\mathbb{E}\left(\sup_{t\in [0,T]}\lvert J_{2}^{\epsilon}(t)\rvert\right)
			\leq& C\sqrt{\epsilon}\,\mathbb{E}\left[\int_{0}^{T}\lVert\sigma_{\epsilon}\left(X_{r}^{\epsilon},\mathcal{L}_{X_{r}^{\epsilon}}\right)\rVert^{2}\cdot\lvert X_{r}^{\epsilon}-X_{r}^{0}\rvert^{2}\,\mathrm{d}r\right]^{\frac{1}{2}} \nonumber\\
			\leq& \left(\frac{1}{4}+CL'T\epsilon\right)\,\mathbb{E}\left(\sup_{r\in[0,T]}\lvert X_{r}^{\epsilon}-X_{r}^{0}\rvert^{2}\right)+C\epsilon\rho_{\sigma,\epsilon}^{2}T+\,C\epsilon.
		\end{align}
		
		With (\ref{mono}) we get
		\begin{align}\label{464}
			\sup_{t\in [0,T]} J_{4}^{\epsilon}(t)\leq 0.
		\end{align}
		
		Combining (\ref{461})-(\ref{464}) together, we have
		\begin{align*}
			&\left(\frac{3}{4}-CL'T\epsilon\right)\,\,\mathbb{E}\left(\sup_{r\in[0,T]}\lvert X_{r}^{\epsilon}-X_{r}^{0}\rvert^{2}\right) \\
			\leq& CL'\mathbb{E}\int_{0}^{T}\lvert X_{s}^{\epsilon}-X_{s}^{0}\rvert^{2}\,\mathrm{d}s+\,C\epsilon+\rho_{b,\epsilon}^{2}T+C\epsilon\rho_{\sigma,\epsilon}^{2}T.
		\end{align*}
		
		Since there exists $\epsilon_0>0$ small enough such that for any $\epsilon\in(0,\epsilon_0],$
		$$\frac{3}{4}-C{L'}^{2}T\epsilon\geq\frac{1}{4},$$
		Finally by Gronwall's inequility, for any $\epsilon\in \left(0,\epsilon_{0}\right]$, we obtain the desired result
		$$\mathbb{E}\left(\sup_{t\in[0,T]}\lvert X_{t}^{\epsilon}-X_{t}^{0}\rvert^{2}\right)\leq C_T(\epsilon+\epsilon\rho_{\sigma,\epsilon}^2+\rho_{b,\epsilon}^2),\quad \forall\epsilon\in \left(0,\epsilon_{0}\right].$$
		
	\end{proof}

	\begin{lemma}\label{mdp2lemma}
		Let $M^{\epsilon,\psi_\epsilon}$ be the solution to (\ref{mdp2m}), then there exists some $\kappa_0>0$ such that for any $\psi_{\epsilon}\in\mathcal{D}_m$, then
		\begin{align}\label{minfty}
			\sup_{\epsilon\in(0,\kappa_0]}\mathbb{E}\left(\sup_{t\in[0,T]}|M^{\epsilon,\psi_{\epsilon}}_t|^2\right)<\infty.
		\end{align}
		\begin{proof}
			By applying It$\hat{\text{o}}$'s formula we have for any $t\in[0,T]$,
			\begin{align}
				&|M^{\epsilon,\psi_{\epsilon}}_t|^2  \nonumber\\
				=&\frac{2}{\lambda(\epsilon)}\int_{0}^{t}\left\langle M^{\epsilon,\psi_{\epsilon}}_s,b_{\epsilon}(\lambda(\epsilon)M^{\epsilon,\psi_{\epsilon}}_s+X^0_s,\mathcal{L}_{\bar{X}^\epsilon_s})-b(X^0_s,\mathcal{L}_{X^0_s})\right\rangle\mathrm{d}s  \nonumber\\
				&+\frac{2\sqrt{\epsilon}}{\lambda(\epsilon)}\int_{0}^{t}\left\langle M^{\epsilon,\psi_{\epsilon}}_s,\sigma_{\epsilon}(\lambda(\epsilon)M^{\epsilon,\psi_{\epsilon}}_s+X^0_s,\mathcal{L}_{\bar{X}^\epsilon_s})\mathrm{~d}W_s\right\rangle \nonumber+\frac{\epsilon}{\lambda^{2}(\epsilon)}\int_{0}^{t}\lVert \sigma_{\epsilon}(\lambda(\epsilon)M^{\epsilon,\psi_{\epsilon}}_s+X^0_s,\mathcal{L}_{\bar{X}^\epsilon_s})\rVert^2\mathrm{~d}s  \nonumber\\
				&+2\int_{0}^{t}\left\langle \sigma_{\epsilon}(\lambda(\epsilon)M^{\epsilon,\psi_{\epsilon}}_s+X^0_s,\mathcal{L}_{\bar{X}^\epsilon_s})\psi_\epsilon(s),M^{\epsilon,\psi_{\epsilon}}_s\right\rangle\mathrm{d}s  \nonumber-2\int_{0}^{t}\left\langle M^{\epsilon,\psi_{\epsilon}}_s,\mathrm{d}\hat{K}^{\epsilon,\psi_\epsilon}_s\right\rangle \nonumber\\
				=:&I_{1}(t)+I_{2}(t)+I_{3}(t)+I_{4}(t)+I_{5}(t).
			\end{align}
			
			
			By (\hyperref[b0]{\textbf{B0}})-(\hyperref[b3]{\textbf{B3}}), there exists $\epsilon_{1}>0$ such that
			\begin{align}\label{ep1}
				\frac{\epsilon}{\lambda^2(\epsilon)}\vee \lambda(\epsilon) \vee \rho_{b,\epsilon} \vee \rho_{\sigma,\epsilon} \vee \frac{\rho_{b,\epsilon}}{\lambda(\epsilon)}\in (0,\frac{1}{2}], \quad \forall\epsilon\in(0,\epsilon_{1}].
			\end{align}
			Let $\epsilon_{2}=\epsilon_{0}\wedge\epsilon_{1}\wedge\frac{1}{2}$, where $\epsilon_{0}$ is the constant given in Lemma \ref{ep0}. Denote by $C$ some generic constant which may change from line to line and is independent of $\epsilon$.
			
			
			Since $A_\epsilon$ is monotone, by Proposition (\ref{multi}) and H\"older's inequality we have
			\begin{align}\label{mi5}
				\sup_{t\in[0,T]}I_5(t)&
				\leq \sup_{t\in[0,T]}\left(-|\hat{K}^{\epsilon,\psi_{\epsilon}}|_0^t\right)+C\int_{0}^{T}|M^{\epsilon,\psi_{\epsilon}}_s|\mathrm{d}s+C
				\leq C\int_{0}^{t}\sup_{r\in[0,s]}|M^{\epsilon,\psi_{\epsilon}}_r|^2\mathrm{d}s+C.
			\end{align}
			
			Due to (\hyperref[b0]{\textbf{B0}}), (\hyperref[b1]{\textbf{B1}}), Lemma \ref{ep2} and (\ref{slim}), for any $\epsilon\in(0,\epsilon_{2}]$,
			\begin{align}\label{mi1}
				I_1(t)=&\,\,\frac{2}{\lambda(\epsilon)}\int_{0}^{t}\left\langle b_{\epsilon}(\lambda(\epsilon)M^{\epsilon,\psi_{\epsilon}}_s+X^0_s,\mathcal{L}_{\bar{X}^\epsilon_s})-b(X^0_s,\mathcal{L}_{X^0_s}),M^{\epsilon,\psi_{\epsilon}}_s\right\rangle\mathrm{d}s  \nonumber\\
				=&\,\,\frac{2}{\lambda(\epsilon)}\int_{0}^{t}\left\langle b_{\epsilon}(\lambda(\epsilon)M^{\epsilon,\psi_{\epsilon}}_s+X^0_s,\mathcal{L}_{\bar{X}^\epsilon_s})-b(\lambda(\epsilon)M^{\epsilon,\psi_{\epsilon}}_s+X^0_s,\mathcal{L}_{\bar{X}^\epsilon_s}),M^{\epsilon,\psi_{\epsilon}}_s\right\rangle\mathrm{d}s  \nonumber\\
				&+\frac{2}{\lambda(\epsilon)}\int_{0}^{t}\left\langle b(\lambda(\epsilon)M^{\epsilon,\psi_{\epsilon}}_s+X^0_s,\mathcal{L}_{\bar{X}^\epsilon_s})-b(X^0_s,\mathcal{L}_{\bar{X}^\epsilon_s}),M^{\epsilon,\psi_{\epsilon}}_s\right\rangle\mathrm{d}s  \nonumber\\
				&+\frac{2}{\lambda(\epsilon)}\int_{0}^{t}\left\langle b(X^0_s,\mathcal{L}_{\bar{X}^{\epsilon}_s})-b(X^0_s,\mathcal{L}_{X^0_s}),M^{\epsilon,\psi_{\epsilon}}_s\right\rangle\mathrm{d}s \nonumber\\
				\leq&\,\,\frac{2\rho_{b,\epsilon}}{\lambda(\epsilon)}\int_{0}^{t}|M^{\epsilon,\psi_{\epsilon}}_s|\mathrm{d}s+2L'\int_{0}^{t}|M^{\epsilon,\psi_{\epsilon}}_s|^2\mathrm{d}s+\frac{2L'}{\lambda(\epsilon)}\int_{0}^{t}|M^{\epsilon,\psi_{\epsilon}}_s|\mathbb{W}_2(\mathcal{L}_{\bar{X}^{\epsilon}_s},\mathcal{L}_{X^0_s})\mathrm{d}s  \nonumber\\
				\leq&\,\,(2L'+\frac{2L'C\sqrt{\epsilon}}{\lambda(\epsilon)})\int_{0}^{t}|M^{\epsilon,\psi_{\epsilon}}_s|^2\mathrm{d}s+\frac{2\rho_{b,\epsilon}}{\lambda(\epsilon)}\int_{0}^{t}|M^{\epsilon,\psi_{\epsilon}}_s|\mathrm{d}s
				\leq C\int_{0}^{t}|M^{\epsilon,\psi_{\epsilon}}_s|^2\mathrm{d}s+C.
			\end{align}
			
			For $I_3(t)$, we have
			\begin{align}\label{mi3}
				I_3(t)=&\,\,\frac{\epsilon}{\lambda^{2}(\epsilon)}\int_{0}^{t}\lVert \sigma_{\epsilon}(\lambda(\epsilon)M^{\epsilon,\psi_{\epsilon}}_s+X^0_s,\mathcal{L}_{\bar{X}_s^\epsilon})\rVert^2\mathrm{~d}s  \nonumber\\
				\leq&\,\, \frac{\epsilon}{\lambda^{2}(\epsilon)}\int_{0}^{t}\lVert \sigma_{\epsilon}(\lambda(\epsilon)M^{\epsilon,\psi_{\epsilon}}_s+X^0_s,\mathcal{L}_{\bar{X}_s^\epsilon})-\sigma(\lambda(\epsilon)M^{\epsilon,\psi_{\epsilon}}_s+X^0_s,\mathcal{L}_{\bar{X}_s^\epsilon})\rVert^2\mathrm{~d}s  \nonumber\\
				&+\frac{\epsilon}{\lambda^{2}(\epsilon)}\int_{0}^{t}\lVert \sigma(\lambda(\epsilon)M^{\epsilon,\psi_{\epsilon}}_s+X^0_s,\mathcal{L}_{\bar{X}_s^\epsilon})-\sigma(X^0_s,\mathcal{L}_{\bar{X}_s^\epsilon})\rVert^2\mathrm{~d}s  \nonumber\\
				&+\frac{\epsilon}{\lambda^{2}(\epsilon)}\int_{0}^{t}\lVert \sigma(X^0_s,\mathcal{L}_{\bar{X}_s^\epsilon})-\sigma(X^0_s,\mathcal{L}_{X^0_s})\rVert^2\mathrm{~d}s+\frac{\epsilon}{\lambda^{2}(\epsilon)}\int_{0}^{t}\lVert \sigma(X^0_s,\mathcal{L}_{X^0_s})\rVert^2\mathrm{~d}s \nonumber\\
				\leq&\,\, \frac{C\epsilon\rho^2_{\sigma,\epsilon}}{\lambda^{2}(\epsilon)}+
				C\epsilon\int_{0}^{t}|M^{\epsilon,\psi_{\epsilon}}_s|^2\mathrm{d}s+\frac{\epsilon}{\lambda^{2}(\epsilon)}\int_{0}^{t}\mathbb{W}_2^2(\mathcal{L}_{\bar{X}^{\epsilon}_s},\mathcal{L}_{X^0_s})\mathrm{d}s+\frac{\epsilon}{\lambda^{2}(\epsilon)}\int_{0}^{t}\lVert \sigma(X^0_s,\mathcal{L}_{X^0_s})\rVert^2\mathrm{~d}s  \nonumber \\
				\leq&\,\, C\int_{0}^{t}|M^{\epsilon,\psi_{\epsilon}}_s|^2\mathrm{d}s+C.
			\end{align}
			
			For $I_4(t)$,  we have
			\begin{align}\label{mi44}
				I_4(t)=&\,2\int_{0}^{t}\left\langle \sigma_{\epsilon}(\lambda(\epsilon)M^{\epsilon,\psi_{\epsilon}}_s+X^0_s,\mathcal{L}_{X_s^0})\psi_\epsilon(s),M^{\epsilon,\psi_{\epsilon}}_s\right\rangle\mathrm{d}s  \nonumber\\
				=&\,2\int_{0}^{t}\left\langle \left[\sigma_{\epsilon}(\lambda(\epsilon)M^{\epsilon,\psi_{\epsilon}}_s+X^0_s,\mathcal{L}_{\bar{X}_s^\epsilon})-\sigma(\lambda(\epsilon)M^{\epsilon,\psi_{\epsilon}}_s+X^0_s,\mathcal{L}_{\bar{X}_s^\epsilon})\right]\psi_\epsilon(s),M^{\epsilon,\psi_{\epsilon}}_s\right\rangle\mathrm{d}s  \nonumber\\
				&+2\int_{0}^{t}\left\langle \left[\sigma(\lambda(\epsilon)M^{\epsilon,\psi_{\epsilon}}_s+X^0_s,\mathcal{L}_{\bar{X}_s^\epsilon})-\sigma(X^0_s,\mathcal{L}_{X^0_s})\right]\psi_\epsilon(s),M^{\epsilon,\psi_{\epsilon}}_s\right\rangle\mathrm{d}s  \nonumber \\
				&+2\int_{0}^{t}\left\langle \sigma(X^0_s,\mathcal{L}_{X^0_s})\psi_\epsilon(s),M^{\epsilon,\psi_{\epsilon}}_s\right\rangle\mathrm{d}s  \nonumber\\
				\leq&\,\,2\rho_{\sigma,\epsilon}\int_{0}^{t}|M^{\epsilon,\psi_{\epsilon}}_s||\psi_\epsilon(s)|\mathrm{d}s+C\int_{0}^{t}(\lambda(\epsilon)|M^{\epsilon,\psi_{\epsilon}}_s|+\mathbb{W}_2(\mathcal{L}_{\bar{X}_s^\epsilon},\mathcal{L}_{X^0_s}))|M^{\epsilon,\psi_{\epsilon}}_s||\psi_\epsilon(s)|\mathrm{d}s  \nonumber\\
				&+\,C\int_{0}^{t}\lVert \sigma(X^0_s,\mathcal{L}_{X^0_s})\rVert|M^{\epsilon,\psi_{\epsilon}}_s||\psi_\epsilon(s)|\mathrm{d}s.
			\end{align}
			By H{\"o}lder's inequality and Young's inequality, we get
			\begin{align}\label{mi4}	
				I_4(t)
				\leq&\,\, \left(1+C\left(\mathbb{E}(\sup_{s\in[0,T]}|\bar{X}^\epsilon_s-X^0_s|^2
				)\right)^{\frac{1}{2}}\right)\int_{0}^{t}|M^{\epsilon,\psi_{\epsilon}}_s||\psi_\epsilon(s)|\mathrm{d}s+C\lambda(\epsilon)\int_{0}^{t}|M^{\epsilon,\psi_{\epsilon}}_s|^2|\psi_\epsilon(s)|\mathrm{d}s   \nonumber\\
				&+\, C\int_{0}^{t}\lVert \sigma(X^0_s,\mathcal{L}_{X^0_s})\rVert|M^{\epsilon,\psi_{\epsilon}}_s||\psi_\epsilon(s)|\mathrm{d}s  \nonumber\\
				\leq&\,\, C\int_{0}^{t}|M^{\epsilon,\psi_{\epsilon}}_s|^2\mathrm{d}s+C\int_{0}^{t}|\psi_\epsilon(s)|^2\mathrm{d}s+C\int_{0}^{t}|M^{\epsilon,\psi_{\epsilon}}_s|^2|\psi_\epsilon(s)|^2\mathrm{d}s+C\int_{0}^{t}\lVert \sigma(X^0_s,\mathcal{L}_{X^0_s})\rVert^2\mathrm{d}s \nonumber\\
				\leq&\,\, C\int_{0}^{t}\left(1+|\psi_\epsilon(s)|^2\right)|M^{\epsilon,\psi_{\epsilon}}_s|^2\mathrm{d}s+C.
			\end{align}
			
			By combining the above estimates (\ref{mi5})-(\ref{mi4}) together and  applying Gronwall's inequality, we obtain  for any $\epsilon\in(0,\epsilon_{2}]$ and $t\in[0,T]$,
			\begin{align}
				\sup_{t\in[0,T]}|M^{\epsilon,\psi_{\epsilon}}_t|^2\leq e^{C\int_{0}^{T}(4+|\psi_\epsilon(s)|^2)\mathrm{d}s}\left\{C+\sup_{t\in[0,T]}|I_2(t)|\right\}.
			\end{align}
			
			Since $\psi_\epsilon\in\mathcal{S}_m$ $P$-a.s., then for any $\epsilon\in(0,\epsilon_{2}]$ we have
			\begin{align}
				\frac{1}{2}\int_{0}^{T}|\psi_\epsilon(s)|^2\mathrm{d}s\leq m, \quad P\text{-a.s.}.
			\end{align}
			
			Therefore, there exists a constant $\zeta>0$ such that for any $\epsilon\in(0,\epsilon_{2}]$,
			\begin{align}\label{mgron}
				\mathbb{E}\left(\sup_{t\in[0,T]}|M^{\epsilon,\psi_\epsilon}_t|^2\right)\leq \zeta\left\{1+\mathbb{E}\left(\sup_{t\in[0,T]}|I_2(t)|\right)\right\}.
			\end{align}
			
			For $I_2(t)$, by BDG’s inequality, Young's inequality, (\hyperref[b3]{\textbf{B3}}), Lemma \ref{ep2}, (\ref{slim}) and (\ref{mi3}), we have for any $\epsilon\in(0,\epsilon_{2}]$,
			\begin{align}\label{mi2}
				\mathbb{E}\left(\sup_{t\in[0,T]}|I_2(t)|\right) \nonumber
				\leq&\,\, \frac{C\sqrt{\epsilon}}{\lambda(\epsilon)}\mathbb{E}\left(\int_{0}^{T}|M^{\epsilon,\psi_\epsilon}_s|^2\lVert\sigma_{\epsilon}(\lambda(\epsilon)M^{\epsilon,\psi_\epsilon}_s+X^0_s,\mathcal{L}_{\bar{X}^\epsilon_s})\rVert^2\mathrm{d}s\right)^{\frac{1}{2}} \nonumber\\
				\leq&\,\, \frac{C\sqrt{\epsilon}}{\lambda(\epsilon)}\mathbb{E}\left(\sup_{s\in[0,T]}|M^{\epsilon,\psi_\epsilon}_s|^2\right)+\frac{C\sqrt{\epsilon}}{\lambda(\epsilon)}\mathbb{E}\int_{0}^{T}\lVert\sigma_{\epsilon}(\lambda(\epsilon)M^{\epsilon,\psi_\epsilon}_s+X^0_s,\mathcal{L}_{\bar{X}^\epsilon_s})\rVert^2\mathrm{d}s  \nonumber\\
				\leq&\,\, \frac{C\sqrt{\epsilon}}{\lambda(\epsilon)}\mathbb{E}\left(\sup_{s\in[0,T]}|M^{\epsilon,\psi_\epsilon}_s|^2\right)+\frac{C\sqrt{\epsilon}}{\lambda(\epsilon)}\left(C\rho_{\sigma,\epsilon}^2+C\lambda^2(\epsilon)\mathbb{E}\int_{0}^{T}|M^{\epsilon,\psi_\epsilon}_s|^2\mathrm{d}s\right)  \nonumber\\
				&+\, \frac{C\sqrt{\epsilon}}{\lambda(\epsilon)}\left(\int_{0}^{T}\mathbb{W}_2^2(\mathcal{L}_{\bar{X}^\epsilon_s},\mathcal{L}_{X_s^0})\mathrm{d}s+\int_{0}^{T}\lVert\sigma(X^0_s,\mathcal{L}_{X^0_s})\rVert^2\mathrm{d}s\right)  \nonumber \\
				\leq&\,\, C\left(\frac{\sqrt{\epsilon}}{\lambda(\epsilon)}+\sqrt{\epsilon}\lambda(\epsilon)\right)\mathbb{E}\left(\sup_{s\in[0,T]}|M^{\epsilon,\psi_\epsilon}_s|^2\right)+C.
			\end{align}
			
			By substituting (\ref{mi2}) back into (\ref{mgron}), we obtain for any $\epsilon\in(0,\epsilon_{2}]$,
			\begin{align}
				\left(1-\frac{C\sqrt{\epsilon}}{\lambda(\epsilon)}-C\sqrt{\epsilon}\lambda(\epsilon)\right)\mathbb{E}\left(\sup_{s\in[0,T]}|M^{\epsilon,\psi_\epsilon}_s|^2\right)\leq C.
			\end{align}
			
			Then there exists a constant $0\leq\kappa_0\leq \epsilon_2$ such that for any $\epsilon\in(0,\kappa_0]$,
			\begin{align*}
				\left(1-\frac{C\sqrt{\epsilon}}{\lambda(\epsilon)}-C\sqrt{\epsilon}\lambda(\epsilon)\right)\geq \frac{1}{4}.
			\end{align*}
			Thus we have
			\begin{align*}
				\sup_{\epsilon\in(0,\kappa_0]}\mathbb{E}\left(\sup_{t\in[0,T]}|M^{\epsilon,\psi_{\epsilon}}_t|^2\right)<\infty,
			\end{align*}
			which completes the proof.
		\end{proof}
		
	\end{lemma}
	
	By adapting the same procedure in Lemma \ref{yozfin} and Lemma \ref{yosifin}, we have the following results.
	\begin{lemma}
		Assume that \hyperref[h0]{\textbf{(H0)}}, (\hyperref[b0]{\textbf{B0}})-\hyperref[b3]{\textbf{(B3)}} and \hyperref[h4]{\textbf{(H4)}} hold. For any $p\geq 1$ and $0\in\overline{D(A)}$, there exists a constant $C>0$ such that for any $\epsilon\in(0,1)$ and $\alpha>0$,
		\begin{align*}
			\mathbb{E}\sup_{t\in[0,T]}|M^{\epsilon,\psi_\epsilon,\alpha}_t|^{2p}\leq C,
		\end{align*}
		where $C$ may depend on $p,T$ and $m$, but not $\epsilon$.
	\end{lemma}
	
	\begin{lemma}
		Assume that \hyperref[h0]{\textbf{(H0)}}, (\hyperref[b0]{\textbf{B0}})-\hyperref[b3]{\textbf{(B3)}} and \hyperref[h4]{\textbf{(H4)}} hold. For any $p\geq 1$ and $0\in\overline{D(A)}$, there exists a constant $C>0$ such that for any $\epsilon\in(0,1)$,
		\begin{equation*}
			\mathbb{E}\sup_{t\in[0,T]}|M^{\epsilon,\psi_\epsilon}_t|^{2p}+\mathbb{E}|\hat{K}^{\epsilon,u_{\epsilon}}|^T_0\leq C,
		\end{equation*}
		where $C$ may depend on $p,T$ and $m$, but not $\epsilon$.
	\end{lemma}
	
	Similar to the proof of Lemma \ref{afin}, we deduce the following important lemma.
	\begin{lemma}\label{mfin}
		Assume that \hyperref[h0]{\textbf{(H0)}}, (\hyperref[b0]{\textbf{B0}})-\hyperref[b3]{\textbf{(B3)}} and \hyperref[h4]{\textbf{(H4)}} hold. For $0\in\overline{D(A)}$, we have
		\begin{equation*}
			\lim\limits_{\alpha\to 0}\mathbb{E}\sup_{t\in[0,T]}|M^{\epsilon,u_\epsilon,\alpha}_t-M^{\epsilon,u_\epsilon}_t|^2=0
		\end{equation*}
		uniformly in $\epsilon>0$.
	\end{lemma}
	
	Consider the related Yosida approximation of equation (\ref{mdp1})
	\begin{align*}
		\begin{cases}
			\mathrm{d}\nu^{\psi,\alpha}(t)=\nabla b(X^0_t,\mathcal{L}_{X^0_t})\nu^{\psi,\alpha}(t)\mathrm{d}t+\sigma(X^0_t,\mathcal{L}_{X^0_t})\psi(t)\mathrm{d}t-A^\alpha_\epsilon(\nu^{\psi,\alpha}_t)\mathrm{d}t, \\
			\nu^{\psi,\alpha}(0)=0.
		\end{cases}
	\end{align*}
	
	Analogously, we can prove the following lemma.
	\begin{lemma}\label{nfin}
		Assume that \hyperref[h0]{\textbf{(H0)}}, (\hyperref[b0]{\textbf{B0}})-\hyperref[b3]{\textbf{(B3)}} and \hyperref[h4]{\textbf{(H4)}} hold. For any $p\geq 1$ and $0\in\overline{D(A)}$, there exists $C>0$ such that for any $\alpha>0$,
		\begin{equation*}
			\sup_{t\in[0,T]}|\nu^{\psi,\alpha}_t|^{2p}\leq C,
		\end{equation*}
		where $C$ may depend on $p,T$ and $m$, but not $\epsilon$ and $\alpha$.
		
		Moreover, for $0\in\overline{D(A)}$, we have
		\begin{equation*}
			\lim\limits_{\alpha\to 0}\sup_{t\in[0,T]}|\nu^{\psi,\alpha}_t-\nu^{\psi}_t|^2=0.
		\end{equation*}
	\end{lemma}
	
	For the sake of brevity, we omit the proofs of the lemmas above. Now we are in the position to verify \textbf{(MDP)}$\bm{_{2}}$.
	
	\begin{proposition}\label{}[\textbf{(MDP)}$\bm{_{2}}$]
		For any given $m\in(0,\infty)$, $\psi_\epsilon\in\mathcal{D}_m$ and $\delta>0$, we have
		\begin{align*}
			\lim\limits_{\epsilon\to 0}P\left(\sup_{t\in[0,T]}|M^{\epsilon,\psi_\epsilon}_t-\nu^{\psi_\epsilon}_t|\geq\delta\right)=0.
		\end{align*}
		
		\begin{proof}	
			Let $\epsilon>0$ be fixed. By triangle inequality we have
			\begin{align}\label{twoo}
				\sup_{t\in[0,T]}|M^{\epsilon,\psi_\epsilon}_t-\nu^{\psi_\epsilon}_t|\leq& \,\sup_{t\in[0,T]}|M^{\epsilon,\psi_\epsilon}_t-M^{\epsilon,\psi_\epsilon,\alpha}_t|
				+\sup_{t\in[0,T]}|M^{\epsilon,\psi_\epsilon,\alpha}_t-\nu^{\psi_\epsilon,\alpha}_t| \nonumber\\
				&+\,\sup_{t\in[0,T]}|\nu^{\psi_\epsilon,\alpha}_t-\nu^{\psi_\epsilon}_t|.
			\end{align}
			Next we deal with the above three terms separately.
			For any $\varpi>0$, by Lemma \ref{mfin} and Lemma \ref{nfin}, there exists some constant $\alpha$ independent of $\epsilon$ such that
			\begin{align}\label{441}
				P\left(\sup_{t\in[0,T]}|M^{\epsilon,\psi_\epsilon}_t-M^{\epsilon,\psi_\epsilon,\alpha}_t|\geq\frac{\delta}{3}\right)\leq \frac{9\mathbb{E}\sup_{t\in[0,T]}|M^{\epsilon,\psi_\epsilon}_t-M^{\epsilon,\psi_\epsilon,\alpha}_t|^2}{\delta^2}\leq \frac{\varpi}{3}
			\end{align}
			and
			\begin{align}\label{442}
				P\left(\sup_{t\in[0,T]}|\nu^{\psi_\epsilon}_t-\nu^{\psi_\epsilon,\alpha}_t|\geq\frac{\delta}{3}\right)\leq \frac{9\mathbb{E}\sup_{t\in[0,T]}|\nu^{\psi_\epsilon}_t-\nu^{\psi_\epsilon,\alpha}_t|^2}{\delta^2}\leq \frac{\varpi}{3}.
			\end{align}
			
			It is left to estimate the second term in the right hand side of \eqref{twoo}. For any $\alpha>0$ and $\psi_\epsilon\in\mathcal{D}_m$, let $Q^\epsilon(t):=M^{\epsilon,\psi_\epsilon,\alpha}(t)-\nu^{\psi_\epsilon,\alpha}(t),$ $\forall t\in[0,T]$, then
			\begin{align} \mathrm{d}Q^\epsilon_t=&\left(\frac{1}{\lambda(\epsilon)}\left(b_\epsilon(\lambda(\epsilon)M^{\epsilon,\psi_\epsilon,\alpha}_t+X^0_t,\mathcal{L}_{\bar{X}^\epsilon_t})-b(X^0_t,\mathcal{L}_{X^0_t})\right)-\nabla b(X^0_t,\mathcal{L}_{X^0_t})\nu^{\psi_\epsilon,\alpha}_t\right)\mathrm{d}t   \nonumber\\
				&+\left(\sigma_{\epsilon}(\lambda(\epsilon)M^{\epsilon,\psi_\epsilon}_t+X^0_t,\mathcal{L}_{\bar{X}^\epsilon_t})-\sigma(X^0_t,\mathcal{L}_{X^0_t})\right)\psi_\epsilon(t)\mathrm{d}t  \nonumber \\
				&+\frac{\sqrt{\epsilon}}{\lambda(\epsilon)}\sigma_{\epsilon}(\lambda(\epsilon)M^{\epsilon,\psi_\epsilon,\alpha}_t+X^0_t,\mathcal{L}_{\bar{X}^\epsilon_t})\mathrm{d}W_t  \nonumber-\left(\mathrm{d}\hat{K}_t^{\epsilon,\psi_\epsilon}-\mathrm{d}\hat{K}_t^{\psi_\epsilon}\right).
			\end{align}
			
			For any $a\in\mathbb{N}$, define the following stopping time
			\begin{align}
				\tau_\epsilon^a=\inf\big\{t\geq 0:|M^{\epsilon,\psi_\epsilon}(t)|\geq a\big\}\wedge T.
			\end{align}
			Due to Lemma \ref{mdp2lemma} and Markov's inequality, we get
			\begin{align}\label{cheby}
				P(\tau_\epsilon^a<T)\leq\frac{\mathbb{E}(\sup_{t\in[0,T]}|M^{\epsilon,\psi_\epsilon}(t)|^2)}{a^2}\leq \frac{C}{a^2}, \quad \forall\epsilon\in(0,\kappa_0].
			\end{align}
			
			By It$\hat{\text{o}}$'s formula, we get
			\begin{align}\label{mdpq}
				|Q^\epsilon_{t\wedge\tau_\epsilon^a}|^2=&2\int_{0}^{t\wedge\tau_\epsilon^a}\left\langle \frac{1}{\lambda(\epsilon)}\left(b_\epsilon(\lambda(\epsilon)M^{\epsilon,\psi_\epsilon,\alpha}_s+X^0_s,\mathcal{L}_{\bar{X}^\epsilon_s})-b(X^0_s,\mathcal{L}_{X^0_s})\right)-\nabla b(X^0_s,\mathcal{L}_{X^0_s})\nu^{\psi_\epsilon,\alpha}_s,Q^\epsilon_s\right\rangle\mathrm{d}s  \nonumber\\
				&+\frac{2\sqrt{\epsilon}}{\lambda(\epsilon)}\int_{0}^{t\wedge\tau_\epsilon^a}\left\langle Q^\epsilon_s,\sigma_{\epsilon}(\lambda(\epsilon)M^{\epsilon,\psi_\epsilon,\alpha}_s+X^0_s,\mathcal{L}_{\bar{X}^\epsilon_s})\mathrm{d}W_s\right\rangle  \nonumber\\
				&+\frac{\epsilon}{\lambda^2(\epsilon)}\int_{0}^{t\wedge\tau_\epsilon^a}\lVert\sigma_{\epsilon}(\lambda(\epsilon)M^{\epsilon,\psi_\epsilon,\alpha}_s+X^0_s,\mathcal{L}_{\bar{X}^\epsilon_s})\rVert^2\mathrm{d}s  \nonumber\\
				&+2\int_{0}^{t\wedge\tau_\epsilon^a}\left\langle \left(\sigma_{\epsilon}(\lambda(\epsilon)M^{\epsilon,\psi_\epsilon,\alpha}_s+X^0_s,\mathcal{L}_{\bar{X}^\epsilon_s})-\sigma(X^0_s,\mathcal{L}_{X^0_s})\right)\psi_\epsilon(s),Q^\epsilon_s\right\rangle\mathrm{d}s  \nonumber\\
				&+2\int_{0}^{t\wedge\tau_\epsilon^a}\left\langle Q^\epsilon_s,A^\alpha(\nu^{\psi_\epsilon,\alpha}_s)-A^\alpha(M^{\epsilon,\psi_\epsilon,\alpha}_s)\right\rangle\mathrm{d}s \nonumber \\
				&+2\int_{0}^{t\wedge\tau_\epsilon^a}\left\langle Q^\epsilon_s,A^\alpha(M^{\epsilon,\psi_\epsilon,\alpha}_s)-A^\alpha_\epsilon(M^{\epsilon,\psi_\epsilon,\alpha}_s)\right\rangle\mathrm{d}s  \nonumber\\
				=:&\,\,\hat{I}_1(t)+\hat{I}_2(t)+\hat{I}_3(t)+\hat{I}_4(t)+\hat{I}_5(t)+\hat{I}_6(t).
			\end{align}
			
			For $\hat{I}_{5}(t)$, by (\ref{mono}) we have for each $\epsilon\in(0,\epsilon_{3}]$,
			\begin{align}\label{hati5}
				\hat{I}_{5}(t)=2\int_{0}^{t\wedge\tau_\epsilon^a}\left\langle Q^\epsilon_s,A^\alpha(\nu^{\psi_\epsilon,\alpha}_s)-A^\alpha(M^{\epsilon,\psi_\epsilon,\alpha}_s)\right\rangle\mathrm{d}s\leq 0.
			\end{align}
			
			By (\ref{nulim}) and $\psi_\epsilon\in\mathcal{D}_m$, there exists some $\Omega^0\in\mathcal{F}$ with $P(\Omega^0)=1$ such that
			\begin{align}\label{gamma}
				\gamma\coloneqq\sup_{\epsilon\in(0,\kappa_0]}\sup_{\omega\in\Omega^0, t\in[0,T]}|\nu^{\psi_\epsilon,\alpha}_t(\omega)|<\infty.
			\end{align}
			
			For $\hat{I}_{6}(t)$, by the same way in Proposition \ref{ld2}, we have
			\begin{align}\label{hati6}
				\hat{I}_{6}(t)\leq C\int_{0}^{t\wedge\tau_\epsilon^a}|Q^{\epsilon,\psi_\epsilon}_s|^2\mathrm{d}s+C\int_{0}^{t\wedge\tau_\epsilon^a}|A_\epsilon^\alpha(M^{\epsilon,\psi_\epsilon,\alpha}_s)-A^\alpha(M^{\epsilon,\psi_\epsilon,\alpha}_s)|^2\mathrm{d}t\leq (a+\gamma)^2T.
			\end{align}
			
			For $\hat{I}_1(t)$, we have
			\begin{align}\label{hati1}
				\hat{I}_1(t)=&\frac{2}{\lambda(\epsilon)}\int_{0}^{t\wedge\tau_\epsilon^a}\left\langle b_\epsilon(\lambda(\epsilon)M^{\epsilon,\psi_\epsilon,\alpha}_s+X^0_s,\mathcal{L}_{\bar{X}^\epsilon_s})-b(\lambda(\epsilon)M^{\epsilon,\psi_\epsilon,\alpha}_s+X^0_s,\mathcal{L}_{\bar{X}^\epsilon_s}),Q^{\epsilon,\psi_\epsilon}_s\right\rangle\mathrm{d}s  \nonumber\\
				&+\frac{2}{\lambda(\epsilon)}\int_{0}^{t\wedge\tau_\epsilon^a}\left\langle b(\lambda(\epsilon)M^{\epsilon,\psi_\epsilon,\alpha}_s+X^0_s,\mathcal{L}_{\bar{X}^\epsilon_s})-b(\lambda(\epsilon)M^{\epsilon,\psi_\epsilon,\alpha}_s+X^0_s,\mathcal{L}_{X^0_s}),Q^{\epsilon,\psi_\epsilon}_s\right\rangle\mathrm{d}s  \nonumber\\
				&+2\int_{0}^{t\wedge\tau_\epsilon^a}\left\langle \frac{1}{\lambda(\epsilon)}\left(b(\lambda(\epsilon)M^{\epsilon,\psi_\epsilon,\alpha}_s+X^0_s,\mathcal{L}_{X^0_s})-b(X^0_s,\mathcal{L}_{X^0_s})\right)-\nabla b(X^0_s,\mathcal{L}_{X^0_s})\nu^{\psi_\epsilon,\alpha}_s,Q^{\epsilon,\psi_\epsilon}_s\right\rangle\mathrm{d}s  \nonumber\\
				=:&\,\,\hat{I}_{1,1}(t)+\hat{I}_{1,2}(t)+\hat{I}_{1,3}(t).
			\end{align}
			Due to (\hyperref[b0]{\textbf{B0}}), (\hyperref[b1]{\textbf{B1}}) and Lemma \ref{ep2}, we have
			\begin{align}\label{hati112}
				\hat{I}_{1,1}(t)+\hat{I}_{1,2}(t)\leq&\, \frac{2\rho_{b,\epsilon}}{\lambda(\epsilon)}\int_{0}^{t\wedge\tau_\epsilon^a}|Q^{\epsilon,\psi_\epsilon}_s|\mathrm{d}s+\frac{2L}{\lambda(\epsilon)}\int_{0}^{t\wedge\tau_\epsilon^a}\left(\mathbb{E}|X^\epsilon_s-X^0_s|^2\right)^{\frac{1}{2}}|Q^{\epsilon,\psi_\epsilon}_s|\mathrm{d}s  \nonumber\\
				\leq& \, \frac{2\rho_{b,\epsilon}+2L(\epsilon+\rho_{b,\epsilon}^2+\epsilon\rho_{\sigma,\epsilon}^2)^{\frac{1}{2}}}{\lambda(\epsilon)}(a+\gamma)T.
			\end{align}
			Let $\epsilon_{3}=\kappa_0\wedge\epsilon_{2}$, then for any $\epsilon\in(0,\epsilon_{3}]$, by using the mean value theorem, (\hyperref[b0]{\textbf{B0}}) and (\hyperref[b1]{\textbf{B1}}), we obtain that there exists $\theta_{\epsilon}(s)\in[0,1]$ such that
			\begin{align*}
				\hat{I}_{1,3}(t)=&\,\,2\int_{0}^{t\wedge\tau_{\epsilon}^{a}}\left\langle \frac{b(\lambda(\epsilon)M^{\epsilon,\psi_\epsilon,\alpha}_s+X^{0}_s,\mathcal{L}_{X^{0}_s})-b(X^{0}_s,\mathcal{L}_{X^{0}_s})}{\lambda(\epsilon)M^{\epsilon,\psi_\epsilon,\alpha}_s}M^{\epsilon,\psi_\epsilon,\alpha}_s-\nabla b(X^{0}_s,\mathcal{L}_{X^{0}_s})\nu^{\psi_\epsilon,\alpha}_s,Q^{\epsilon,\psi_\epsilon}_s\right\rangle\mathrm{d}s \nonumber\\
				\leq&\,\,2\int_{0}^{t\wedge\tau_{\epsilon}^{a}}\left\langle \nabla b(\lambda(\epsilon)M^{\epsilon,\psi_\epsilon,\alpha}_s\theta_{\epsilon}(s)+X^{0}_s,\mathcal{L}_{X^0_s})M^{\epsilon,\psi_\epsilon,\alpha}_s-\nabla b(X^0_s,\mathcal{L}_{X^0_s})\nu^{\psi_\epsilon,\alpha}_s,Q^{\epsilon,\psi_\epsilon}_s\right\rangle\mathrm{d}s  \nonumber\\
				=&\, 2\int_{0}^{t\wedge\tau_{\epsilon}^{a}}\left\langle \nabla b(\lambda(\epsilon)M^{\epsilon,\psi_\epsilon,\alpha}_s\theta_{\epsilon}(s)+X^0_s,\mathcal{L}_{X^0_s})M^{\epsilon,\psi_\epsilon,\alpha}_s-\nabla b(X^0_s,\mathcal{L}_{X^0_s})M^{\epsilon,\psi_\epsilon,\alpha}_s,Q^{\epsilon,\psi_\epsilon}_s\right\rangle\mathrm{d}s  \nonumber\\
				&+\, 2\int_{0}^{t\wedge\tau_{\epsilon}^{a}}\left\langle \nabla b(X^0_s,\mathcal{L}_{X^0_s})M^{\epsilon,\psi_\epsilon,\alpha}_s-\nabla b(X^0_s,\mathcal{L}_{X^0_s})\nu^{\psi_\epsilon}_s,Q^{\epsilon,\psi_\epsilon}_s\right\rangle\mathrm{d}s  \nonumber\\
				\leq&\,\, 2\int_{0}^{t\wedge\tau_{\epsilon}^{a}}| \nabla b(\lambda(\epsilon)M^{\epsilon,\psi_\epsilon,\alpha}_s\theta_{\epsilon}(s)+X^0_s,\mathcal{L}_{X^0_s})-\nabla b(X^0_s,\mathcal{L}_{X^0_s})||M^{\epsilon,\psi_\epsilon,\alpha}_s||Q^{\epsilon,\psi_\epsilon}_s|\mathrm{d}s   \nonumber\\
				&+\, 2\int_{0}^{t\wedge\tau_{\epsilon}^{a}}|\nabla b(X^0_s,\mathcal{L}_{X^0_s})||Q^{\epsilon,\psi_\epsilon}_s|^{2}\mathrm{d}s \nonumber\\
				\leq&\,\, L'\lambda(\epsilon)\int_{0}^{t\wedge\tau_{\epsilon}^{a}}\left(1+|X^0_s|^{q'}+|\lambda(\epsilon)M^{\epsilon,\psi_\epsilon,\alpha}_s+X^0_s|^{q'}\right)|M^{\epsilon,\psi_\epsilon,\alpha}_s|^{2}|Q^{\epsilon,\psi_\epsilon}_s|\mathrm{d}s  \nonumber\\
				&+\, 2\int_{0}^{t\wedge\tau_{\epsilon}^{a}}|\nabla b(X^0_s,\mathcal{L}_{X^0_s})||Q^{\epsilon,\psi_\epsilon}_s|^{2}\mathrm{d}s.
			\end{align*}
			
			Denote $C_a=L'\left(1+\sup_{s\in[0,T]}|X_{s}^{0}|^{q'}+|a+\sup_{s\in[0,T]}|X_{s}^{0}||^{q'}\right)a^{2}(a+\gamma)T$.  Notice that $C_a$ is independent of $\epsilon$, thus we get
			\begin{align}\label{hati132}
				\hat{I}_{1,3}(t)\leq\, C_a\lambda(\epsilon)+2\int_{0}^{t}|\nabla b(X^0_{s\wedge\tau_\epsilon^a},\mathcal{L}_{X^0_{s\wedge\tau_\epsilon^a}})||Q^{\epsilon,\psi_\epsilon}_{s\wedge\tau_\epsilon^a}|^2\mathrm{d}s.
			\end{align}
			
			Substituting (\ref{hati5}), (\ref{hati6}), (\ref{hati112}) and (\ref{hati132}) back into (\ref{mdpq}), and by Gronwall's inequality we have for any $\epsilon\in(0,\epsilon_{3}],\,\, t\in[0,T]$,
			
			\begin{align}\label{hatgron}
				\sup_{t\in[0,T]}|Q^{\epsilon,\psi_\epsilon}_{t\wedge\tau_\epsilon^a}|^2\leq e^{2\int_{0}^{T}\lVert\nabla b(X^0_s,\mathcal{L}_{X^0_s})\rVert\mathrm{d}s}\left[C_a(\frac{\rho_{b,\epsilon}+\lambda^2(\epsilon)+(\epsilon+\rho_{b,\epsilon}^2
					+\epsilon\rho_{\sigma,\epsilon})^{\frac{1}{2}}}{\lambda(\epsilon)})+\sum_{i=2}^{4}\sup_{t\in[0,T]}\hat{I}_i(t)\right].
			\end{align}
			By (\hyperref[b1]{\textbf{B1}}) we have
			\begin{align}
				\Delta:=e^{2\int_{0}^{T}\lVert\nabla b(X^0_s,\mathcal{L}_{X^0_s})\rVert\mathrm{d}s}<\infty.
			\end{align}
			
			For $\hat{I}_{2}(t)$ and $\hat{I}_{3}(t)$, by BDG’s inequality and (\ref{slim}), we obtain that
			\begin{align}\label{hati23}
				&\mathbb{E}\left(\sup_{t\in[0,T]}|\hat{I}_{2}(t)|\right)+\mathbb{E}\left(\sup_{t\in[0,T]}|\hat{I}_{3}(t)|\right)  \nonumber\\
				\leq&\,\frac{2\sqrt{\epsilon}}{\lambda(\epsilon)}\mathbb{E}\left(\int_{0}^{T\wedge\tau_\epsilon^a} |Q^{\epsilon,\psi_\epsilon}_s|^2\lVert\sigma_{\epsilon}(\lambda(\epsilon)M^{\epsilon,\psi_\epsilon,\alpha}_s+X^0_s,\mathcal{L}_{\bar{X}^\epsilon_s})\rVert^2\mathrm{d}s\right)^\frac{1}{2}  \nonumber\\
				\leq&\,\, \frac{1}{4}\mathbb{E}\left(\sup_{s\in[0,T]}|Q^{\epsilon,\psi_\epsilon}_{s\wedge\tau_\epsilon^a}|^2\right)
				+\frac{C\epsilon}{\lambda^2(\epsilon)}\mathbb{E}\int_{0}^{T\wedge\tau_\epsilon^a}\lVert\sigma_{\epsilon}(\lambda(\epsilon)M^{\epsilon,\psi_\epsilon,\alpha}_s+X^0_s,\mathcal{L}_{\bar{X}^\epsilon_s})\rVert^2\mathrm{d}s \nonumber\\
				\leq&\,\, \frac{1}{4}\mathbb{E}\left(\sup_{s\in[0,T]}|Q^{\epsilon,\psi_\epsilon}_{s\wedge\tau_\epsilon^a}|^2\right)+\frac{C\epsilon\rho_{\sigma,\epsilon}^2}{\lambda^2(\epsilon)}+\frac{C\epsilon}{\lambda^2(\epsilon)}\mathbb{E}\int_{0}^{T\wedge\tau_\epsilon^a}|M^{\epsilon,\psi_\epsilon,\alpha}_s|^2\mathrm{d}s  \nonumber\\
				&+\,\, \frac{C\epsilon}{\lambda^2(\epsilon)}\mathbb{E}\int_{0}^{T\wedge\tau_\epsilon^a}\mathbb{E}(|\bar{X}^\epsilon_s-X^0_s|^2)\mathrm{d}s+\frac{C\epsilon}{\lambda^2(\epsilon)}\int_{0}^{T}\lVert\sigma(X^0_s,\mathcal{L}_{X^0_s})\rVert^2\mathrm{d}s  \nonumber\\
				\leq&\,\, \frac{1}{4}\mathbb{E}\left(\sup_{s\in[0,T]}|Q^{\epsilon,\psi_\epsilon}_{s\wedge\tau_\epsilon^a}|^2\right)+\frac{C_a\epsilon}{\lambda^2(\epsilon)}(\rho_{\sigma,\epsilon}^2+\epsilon+\epsilon\rho_{\sigma,\epsilon}^2+\rho_{b,\epsilon}^2+T).
			\end{align}
			
			For $\hat{I}_{4}(t)$, by \hyperref[b3]{\textbf{(B3)}}, (\ref{gamma}), Lemma \ref{ep2} and $\psi_\epsilon\in\mathcal{D}_m$, and the similar arguments in the proof of (\ref{mi4}), we have for any $\epsilon\in(0,\epsilon_{3}]$,
			\begin{align}\label{hati4}
				\mathbb{E}\left(\sup_{t\in[0,T]}|\hat{I}_{4}(t)|\right) \leq&\, C\rho_{\sigma,\epsilon}\mathbb{E}\int_{0}^{T\wedge\tau_\epsilon^a}|Q^{\epsilon,\psi_\epsilon}_{s}||\psi_\epsilon(s)|\mathrm{d}s  \nonumber\\
				&+\,C\mathbb{E}\int_{0}^{T\wedge\tau_\epsilon^a}\left(\lambda(\epsilon)|M^{\epsilon,\psi_\epsilon,\alpha}_{s}|+\mathbb{W}_2(\mathcal{L}_{\bar{X}^\epsilon_s},\mathcal{L}_{X^0_s})+\lVert\sigma(X^0_s,\mathcal{L}_{X^0_s})\rVert\right)|\psi_\epsilon(s)||Q^{\epsilon,\psi_\epsilon}_s|\mathrm{d}s  \nonumber\\
				\leq&\,\, CT(a+\gamma)^2\left(\rho_{\sigma,\epsilon}+\lambda(\epsilon)+(\epsilon+\rho_{b,\epsilon}^2+\epsilon\rho_{\sigma,\epsilon}^2)^{\frac{1}{2}}\right)\mathbb{E}\left(\int_{0}^{T}|\psi_\epsilon|^2\mathrm{d}s\right)^{\frac{1}{2}}   \nonumber\\
				\leq&\,\, C_a\left(\rho_{\sigma,\epsilon}+\lambda(\epsilon)+(\epsilon+\rho_{b,\epsilon}^2+\epsilon\rho_{\sigma,\epsilon}^2)^{\frac{1}{2}}\right).
			\end{align}
			
			Inserting the above inequalities (\ref{hati23}) and (\ref{hati4}) into (\ref{hatgron}), we deduce that for any $\epsilon\in(0,\epsilon_{3}]$,
			\begin{align}
				\frac{3}{4}\mathbb{E}\left(\sup_{t\in[0,T]}|Q^{\epsilon,\psi_\epsilon}_{t\wedge\tau_\epsilon^a}|^2\right)\leq C_a\left(\rho_{\sigma,\epsilon}+\lambda(\epsilon)+\frac{\rho_{b,\epsilon}}{\lambda(\epsilon)}+\frac{\epsilon}{\lambda^2(\epsilon)}+\frac{1+\lambda(\epsilon)}{\lambda(\epsilon)}(\epsilon+\rho_{b,\epsilon}^2+\epsilon\rho_{\sigma,\epsilon}^2)^{\frac{1}{2}}\right).
			\end{align}
			
			Due to (\hyperref[b2]{\textbf{B2}}), we  obtain
			\begin{align}
				\lim\limits_{\epsilon\to 0}\mathbb{E}\left(\sup_{t\in[0,T]}|M^{\epsilon,\psi_\epsilon,\alpha}_{t\wedge\tau_\epsilon^a}-\nu^{\psi_\epsilon,\alpha}_{t\wedge\tau_\epsilon^a}|^2\right)=\lim\limits_{\epsilon\to 0}\mathbb{E}\left(\sup_{t\in[0,T]}|Q^{\epsilon,\psi_\epsilon}_{t\wedge\tau_\epsilon^a}|^2\right)=0.
			\end{align}
			
			Thus for any $\delta>0$, $\epsilon\in(0,\epsilon_{3}]$, $a\in\mathbb{N}$, we have
			\begin{align}
				&P\left(\sup_{t\in[0,T]}|M^{\epsilon,\psi_\epsilon,\alpha}_t-\nu^{\psi_\epsilon,\alpha}_t|\geq\frac{\delta}{3}\right) \nonumber\\
				\leq&\, P\left(\bigg(\sup_{t\in[0,T]}|M^{\epsilon,\psi_\epsilon,\alpha}_{t\wedge\tau_\epsilon^a}-\nu^{\psi_\epsilon,\alpha}_{t\wedge\tau_\epsilon^a}|\geq\frac{\delta}{3}\bigg)\cap(\tau_\epsilon^a\geq T)\right)+P(\tau_\epsilon^a< T) \nonumber\\
				\leq&\,\, \frac{9}{\delta^2}\mathbb{E}\left(\sup_{t\in[0,T]}|M^{\epsilon,\psi_\epsilon,\alpha}_{t\wedge\tau_\epsilon^a}-\nu^{\psi_\epsilon,\alpha}_{t\wedge\tau_\epsilon^a}|^2\right)+\frac{C}{a^2}.
			\end{align}
			
			For any $\alpha$ fixed, we can choose $\epsilon$ and $a\to\infty$ such that for any $\varpi>0$
			\begin{align}\label{443}
				P\left(\sup_{t\in[0,T]}|M^{\epsilon,\psi_\epsilon,\alpha}_t-\nu^{\psi_\epsilon,\alpha}_t|\geq\frac{\delta}{3}\right)\leq \frac{\varpi}{3}.
			\end{align}
			
			Then combining (\ref{441}), (\ref{442}) and (\ref{443}) we get the desired result
			\begin{align*}
				\lim\limits_{\epsilon\to 0}P\left(\sup_{t\in[0,T]}|M^{\epsilon,\psi_\epsilon}_t-\nu^{\psi_\epsilon}_t|\geq\delta\right)=0.
			\end{align*}
		\end{proof}
		
	\end{proposition}

	\section{Functional Iterated Logarithm Law}
	
	
	Let $ W $ be a k-dimensional Brownian motion and $T$ be fixed. Let $H$ be an open set of $\mathbb{R}^d$ and denote the set of continuous path $\hbar:[0,T]\to H$ by $C([0,T];H)$, endowed with the norm
	$\lVert \hbar\rVert=\sup_{t\in[0,T]}|\hbar(t)|.$
	For $\hbar_1,\,\hbar_2\in C([0,T];H)$ and $B\in \mathcal{B}(C([0,T];H))$, we denote
	\begin{align*}
		d(\hbar_1,\hbar_2)=\lVert \hbar_1-\hbar_2\rVert,\quad d(\hbar_1,B)=\inf_{\hbar_2\in B}d(\hbar_1,\hbar_2).
	\end{align*}
	
	Let
	\begin{align*}
		\mathcal{H}:=\left\{\hbar_.=\int_{0}^{.}h(s)\mathrm{d}s\in C([0,T];\mathbb{R}^d);\hbar(0)=0,\int_{0}^{T}|h(s)|^2\mathrm{d}s<\infty\right\},
	\end{align*}
	
	then $\mathcal{H}$ is a Hilbert space endowed with the scalar product $\langle \hbar_1, \hbar_2\rangle:=\int_{0}^{T}\langle h_1(s),h_2(s)\rangle\mathrm{d}s $.

	\subsection{The Large Time Functional Iterated Logarithm Law}

	For $u>e$, define $$\psi(u):=\sqrt{u\log\log u}\,.$$
	
	Suppose that $\tilde{A}$ is a multivalued maximal monotone operator with $0\in Int(D(\tilde{A}))$,  $\tilde{b}:\overline{D(A)}\times\mathcal{P}_2\to\mathbb{R}^d$ and $\tilde{\sigma}:\overline{D(A)}\times\mathcal{P}_2\to \mathbb{R}^d\otimes\mathbb{R}^{d}$ are progressively measurable functions such that there exists a unique solution to
	\begin{align*}\label{IILS}
		\begin{cases}
			\mathrm{d}Y_t\in\tilde{b}(Y_t,\mathcal{L}_{Y_t})\mathrm{d}t+\tilde{\sigma}(Y_t,\mathcal{L}_{Y_t})\mathrm{d}W_t-\tilde{A}(Y_t)\mathrm{d}t, \\
			Y_0=x\in\overline{D(A)}.
		\end{cases}
	\end{align*}
	
	To show our main result, we first give the following definition introduced by Baldi in \cite{Ba1986}.
	
	\begin{definition}\label{gadef}
		Let $U$ be an open set of $\mathbb{R}^d$. For any $a\in\mathbb{R}^{+}$, let $\Gamma_a:U\to U$ be a $C^2$ bijective transformation having continuous derivatives up to order 2. The family $\Gamma_a=\left\{\Gamma_a\right\}_{a>0}$ is said to be a system of contractions centered at x if
		\begin{enumerate}[(a)]
			\item $\Gamma_a(x)=x$ for every $a\in\mathbb{R}^{+},$
			\item if $ a\geq\theta $ then $ |\Gamma_a(y)-\Gamma_a(z)|\leq |\Gamma_\theta(y)-\Gamma_\theta(z)| $ for every $ y,z\in U,$
			\item $ \Gamma_1 $ is identity and $ \Gamma^{-1}_a=\Gamma_{a^{-1}} $,
			\item for every compact set $\mathcal{K}$ of $U$, $f\in\mathcal{K}$ and $\epsilon>0$, there exists $\delta>0$, such that if $|a\theta-1|<\delta$ then
			$$\lvert\Gamma_a\circ\Gamma_\theta(f)-f\rvert<\epsilon,\quad \forall a,\theta\in\mathbb{R}^{+}.$$
		\end{enumerate}
	\end{definition}
	We assume $\overline{D(A)}\in U$ in the following.\\
	
	\textbf{Notation}. Let $u>e$, we set $\tilde{b}_u:\mathbb{R}^d\times \mathcal{P}_2\to\mathbb{R}^d$ and $\tilde{\sigma}_u:\mathbb{R}^d\times\mathcal{P}_2\to\mathbb{R}^d\otimes\mathbb{R}^d$ such that
	\begin{align*}
		&\tilde{b}_u(y,\mu):=u\tilde{L}(y,\mu)[\Gamma_{\psi(u)}](\Gamma_{\psi(u)}^{-1}(y)), \\
		&\tilde{\sigma}_u(y,\mu):=\psi(u)\nabla[\Gamma_{\psi(u)}](\Gamma_{\psi(u)}^{-1}(y))^{T}\tilde{\sigma}(\Gamma_{\psi(u)}^{-1}(y),\mu\circ\Gamma_{\psi(u)}), \\
		&\tilde{A}_u(y):=u\nabla[\Gamma_{\psi(u)}](\Gamma_{\psi(u)}^{-1}(y))^T\tilde{A}(\Gamma_{\psi(u)}^{-1}(y)),
	\end{align*}
	where for $\tilde{a}=\tilde{\sigma}^T\tilde{\sigma}$, the operator $\tilde{L}$ is given as
	\begin{align*}
		\tilde{L}(y,\mu)[f](q)|_{q=\Gamma_{\psi(u)}^{-1}(y)}=&\sum_{i=1}^{d}\frac{\partial f}{\partial y_i}(\Gamma_{\psi(u)}^{-1}(y))\tilde{b}_i(\Gamma_{\psi(u)}^{-1}(y),\mu\circ\Gamma_{\psi(u)}) \\
		&+\frac{1}{2}\sum_{i,j=1}^{d}\tilde{a}_{i,j}(\Gamma_{\psi(u)}^{-1}(y),\mu\circ\Gamma_{\psi(u)})\frac{\partial^2f}{\partial y_i\partial y_j}(\Gamma_{\psi(u)}^{-1}(y)).
	\end{align*}

	\begin{assumption}\label{Cn}
		For $\forall y\in\mathbb{R}^d$ and $\forall \mu\in\mathcal{P}_2$, we make the following assumptions.
		\begin{enumerate}
			\item[\textbf{(C0)}] There exist $\hat{b}:\mathbb{R}^d\times\mathcal{P}_2\to\mathbb{R}^d$ and $\hat{\sigma}:\mathbb{R}^d\times\mathcal{P}_2\to\mathbb{R}^d\otimes\mathbb{R}^d$ such that
			\begin{align*}
				\lim\limits_{u\to\infty}\tilde{b}_u(y,\mu)=\hat{b}(y,\mu), \quad \lim\limits_{u\to\infty}\tilde{\sigma}_u(y,\mu)=\hat{\sigma}(y,\mu),
			\end{align*}
			where $\hat{b}$ and $\hat{\sigma}$ satisfy \hyperref[h1]{\textbf{(H1)}} and \hyperref[h2]{\textbf{(H2)}} respectively. \\
			\item[\textbf{(C1)}] There exists a multivalued maximal monotone operator $\hat{A}:\mathbb{R}^d\to 2^{\mathbb{R}^d}$ such that for every $\epsilon>0$
			\begin{align*}
				\lim\limits_{u\to\infty}(1+\epsilon\tilde{A}_u)^{-1}(y)=(1+\epsilon A)^{-1}(y).
			\end{align*}
			\item[\textbf{(C2)}] $\tilde{A}_u$, $\tilde{A}$ and $\hat{A}$ are maximal monotone operators with $\overline{D(\tilde{A}_u)}=\overline{D(A)}$ and nonempty interior. Moreover, for $a\in Int(D(\tilde{A}_u))$, $\tilde{A}_u$ is locally bounded at $a$ uniformly in $u$. \\
			\item[\textbf{(C3)}] If $b_\epsilon=\tilde{b}_{1/\epsilon}$, $\sigma_\epsilon=\tilde{\sigma}_{1/\epsilon}$ and $A_\epsilon=\tilde{A}_{1/\epsilon}$, the system of small random perturbation $(b_\epsilon,\sigma_\epsilon,A_\epsilon)$ satisfies \hyperref[h1]{\textbf{(H1)}}, \hyperref[h2]{\textbf{(H2)}} and \hyperref[h4]{\textbf{(H4)}}.
		\end{enumerate}
		
	\end{assumption}

	For $u>e$ and $t\in\mathbb{R}^+$, define
	\begin{align*}
		Q_u(t):=\Gamma_{\psi(u)}(Y_{ut}),
	\end{align*}
	recall that $Y_0=x$ and $\Gamma_{u}(x)=x$, by definition we have $\Gamma_{u}(0)=x$. Since $\Gamma_{\psi(u)}(\cdot)$ is twice differentiable, by It$\hat{\text{o}}$'s formula we have
	\begin{align}\label{qut}
		\mathrm{d}Q_u(t)&=\mathrm{d}\Gamma_{\psi(u)}(Y_{ut}) \nonumber\\
		&=\nabla[\Gamma_{\psi(u)}](Y_{ut})^T\mathrm{d}Y_{ut}+\frac{(\mathrm{d}Y_{ut})^T}{2}Hess[\Gamma_{\psi(u)}](Y_{ut})\mathrm{d}Y_{ut}.
	\end{align}
	Rewriting $Y_{ut}=\Gamma_{\psi(u)}^{-1}(Q_u(t))$ and substituting it in (\ref{qut}), we get
	\begin{align*}
		\mathrm{d}Q_u(t)=&u\sum_{i=1}^{d}\frac{\partial\Gamma_{\psi(u)}}{\partial y_i}(\Gamma_{\psi(u)}^{-1}(Q_u(t)))\tilde{b}_i(\Gamma_{\psi(u)}^{-1}(Q_u(t)),\mathcal{L}_{Q_u(t)}\circ\Gamma_{\psi(u)})\mathrm{d}t \\
		&+\frac{u}{2}\sum_{i,j=1}^{d}\frac{\partial^2\Gamma_{\psi(u)}}{\partial y_i\partial y_j}(\Gamma_{\psi(u)}^{-1}(Q_u(t)))\sum_{k=1}^{d}\tilde{\sigma}_{k,i}\tilde{\sigma}_{j,k}(\Gamma_{\psi(u)}^{-1}(Q_u(t)),\mathcal{L}_{Q_u(t)}\circ\Gamma_{\psi(u)})\mathrm{d}t \\
		&+\sum_{i=1}^{d}\frac{\partial\Gamma_{\psi(u)}}{\partial y_i}(\Gamma_{\psi(u)}^{-1}(Q_u(t)))\sum_{k=1}^{d}\tilde{\sigma}_{i,k}(\Gamma_{\psi(u)}^{-1}(Q_u(t)),\mathcal{L}_{Q_u(t)}\circ \Gamma_{\psi(u)})\mathrm{d}W^k_{ut} \\
		&-u\sum_{i=1}^{d}\frac{\partial\Gamma_{\psi(u)}}{\partial y_i}(\Gamma_{\psi(u)}^{-1}(Q_u(t)))\tilde{A}(\Gamma_{\psi(u)}^{-1}(Q_u(t)))\mathrm{d}t.
	\end{align*}
	By time change, $\mathcal{W}^{(u)}_t=\frac{W_{ut}}{\sqrt{u}}$ is a Brownian motion. With initial condition $Q_u(0)=x$, $Q_u(t)$ is the solution to the following multivalued McKean-Vlasov SDE:
	\begin{align*}
		\mathrm{d}Q_u(t)\in \tilde{b}_u(Q_u(t),\mathcal{L}_{Q_u(t)})\mathrm{d}t+\frac{1}{\sqrt{\log\log u}}\tilde{\sigma}_u(Q_u(t),\mathcal{L}_{Q_u(t)})\mathrm{d}\mathcal{W}^{(u)}_t-\tilde{A}_u(Q_u(t))\mathrm{d}t.
	\end{align*}
	Under Assumption \ref{Cn} and by Theorem \ref{thmldp}, we have for every closed subset $C$ of $\mathcal{B}(C([0,T],\mathbb{R}^d))$,
	\begin{align}\label{ldpc}
		\limsup\limits_{u\to\infty}\frac{1}{\log\log u}\log P(Q_u\in C)\leq -\inf_{q\in C} I(q),
	\end{align}
	and for every open subset $O$ of $\mathcal{B}(C([0,T],\mathbb{R}^d))$,
	\begin{align*}\label{ldpo}
		\liminf\limits_{u\to\infty}\frac{1}{\log\log u}\log P(Q_u\in O)\geq -\inf_{q\in O} I(q).
	\end{align*}
	
	Then we present the main theorem of this section.
	\begin{theorem}\label{ILL}
		Under Assumption \ref{Cn}, the family $\left\{Q_u\right\}_{u>e}$ is relatively compact and the set $\Lambda:=\left\{g;I(g)\leq 1\right\}$ is the limit set of $\left\{Q_u\right\}_{u>e}$ when $u\to\infty$ $a.s.$, where
		\begin{align*}
			I(g)=\frac{1}{2}\inf_{\hbar\in\mathcal{H};g=Y^h}\int_{0}^{T}|h(s)|^2\mathrm{d}s,
		\end{align*}
		and $Y^h$ solve the following multivalued McKean-Vlasov SDE:
		\begin{align*}
			\begin{cases}
				\mathrm{d}Y^h_t\in b(Y^h_t,\mathcal{L}_{Y^h_t})\mathrm{d}t+\sigma(Y^h_t,\mathcal{L}_{Y^h_t})h(t)\mathrm{d}t-A(Y^h_t)\mathrm{d}t, \\
				Y^h_0=x.
			\end{cases}
		\end{align*}
		
		\begin{proof}
			The proof of relatively compactness and the limit set will be given in Proposition \ref{ILLP1} and Proposition \ref{ILLP2} respectively in section \ref{ILLMP}.
		\end{proof}
		
	\end{theorem}

	\subsection{The Small Time Iterated Logarithm Law}
	
	For $0<u<e^{-1}$, define $$\varphi(u):=\sqrt{u\log\log u^{-1}}\,.$$
	
	Let $\tilde{b}:\overline{D(A)}\times\mathcal{P}_2\to\mathbb{R}^d$, $\tilde{\sigma}:\overline{D(A)}\times\mathcal{P}_2\to \mathbb{R}^d\otimes\mathbb{R}^{d}$ and $\tilde{A}$ be a multivalued maximal monotone operator with $0\in Int(D(\tilde{A}))$.
	
	\textbf{Notation}. Let $0<u<e^{-1}$, we also set $\tilde{b}_u:\mathbb{R}^d\times \mathcal{P}_2\to\mathbb{R}^d$ and $\tilde{\sigma}_u:\mathbb{R}^d\times\mathcal{P}_2\to\mathbb{R}^d\otimes\mathbb{R}^d$ such that
	\begin{align*}
		&\tilde{b}_u(y,\mu):=u\tilde{L}(y,\mu)[\Gamma_{\varphi(u)}](\Gamma_{\varphi(u)}^{-1}(y)), \\
		&\tilde{\sigma}_u(y,\mu):=\psi(u)\nabla[\Gamma_{\varphi(u)}](\Gamma_{\varphi(u)}^{-1}(y))^{T}\tilde{\sigma}(\Gamma_{\varphi(u)}^{-1}(y),\mu\circ\Gamma_{\varphi(u)}), \\
		&\tilde{A}_u(y):=u\nabla[\Gamma_{\varphi(u)}](\Gamma_{\varphi(u)}^{-1}(y))^T\tilde{A}(\Gamma_{\varphi(u)}^{-1}(y)),
	\end{align*}
	where for $\tilde{a}=\tilde{\sigma}^T\tilde{\sigma}$, the operator $\tilde{L}$ is given as
	\begin{align*}
		\tilde{L}(y,\mu)[f](q)|_{q=\Gamma_{\varphi(u)}^{-1}(y)}=&\sum_{i=1}^{d}\frac{\partial f}{\partial y_i}(\Gamma_{\varphi(u)}^{-1}(y))\tilde{b}_i(\Gamma_{\varphi(u)}^{-1}(y),\mu\circ\Gamma_{\varphi(u)}) \\
		&+\frac{1}{2}\sum_{i,j=1}^{d}\tilde{a}_{i,j}(\Gamma_{\varphi(u)}^{-1}(y),\mu\circ\Gamma_{\varphi(u)})\frac{\partial^2f}{\partial y_i\partial y_j}(\Gamma_{\varphi(u)}^{-1}(y)).
	\end{align*}
	
	
	\begin{assumption}\label{Cn1}
		For $\forall y\in\mathbb{R}^d$ and $\forall \mu\in\mathcal{P}_2$, we make the following assumption.
		\begin{enumerate}
			\item[\textbf{(C$'$0)}] There exists $\hat{b}:\mathbb{R}^d\times\mathcal{P}_2\to\mathbb{R}^d$ and $\hat{\sigma}:\mathbb{R}^d\times\mathcal{P}_2\to\mathbb{R}^d\otimes\mathbb{R}^d$ such that
			\begin{align*}
				\lim\limits_{u\to 0^+}\tilde{b}_u(y,\mu)=\hat{b}(y,\mu), \quad \lim\limits_{u\to 0^+}\tilde{\sigma}_u(y,\mu)=\hat{\sigma}(y,\mu),
			\end{align*}
			where $\hat{b}$ and $\hat{\sigma}$ satisfy \hyperref[h1]{\textbf{(H1)}} and \hyperref[h2]{\textbf{(H2)}} respectively. \\
			\item[\textbf{(C$'$1)}] There exists a multivalued maximal monotone operator $\hat{A}:\mathbb{R}^d\to 2^{\mathbb{R}^d}$ such that for every $\epsilon>0$
			\begin{align*}
				\lim\limits_{u\to 0^+}(1+\epsilon\tilde{A}_u)^{-1}(y)=(1+\epsilon A)^{-1}(y).
			\end{align*}
			\item[\textbf{(C$'$2)}] $\tilde{A}_u$, $\tilde{A}$ and $\hat{A}$ are maximal monotone operators with $\overline{D(\tilde{A}_u)}=\overline{D(A)}$ and nonempty interior. Moreover, for $a\in Int(D(\tilde{A}_u))$, $\tilde{A}_u$ is locally bounded at $a$ uniformly in $u$. \\
			\item[\textbf{(C$'$3)}] The system of small random perturbation $(\tilde{b}_u,\tilde{\sigma}_u,\tilde{A}_u)$ satisfies \hyperref[h1]{\textbf{(H1)}}, \hyperref[h2]{\textbf{(H2)}} and \hyperref[h4]{\textbf{(H4)}}.
		\end{enumerate}
		
	\end{assumption}
	
	Let $Y_t$ be the solution of multivalued McKean-Vlasov SDE (\ref{IILS}), and define for $0<u<e^{-1}$,
	\begin{align*}
		Q^{'}_u(t):=\Gamma_{\varphi(u)}(Y_{ut}).
	\end{align*}
	
	Then similarily we have the following theorem.
	
	\begin{theorem}
		Under Assumption \ref{Cn1}, the family $\left\{Q^{'}_u\right\}_{0<u<e^{-1}}$ is relatively compact and the set $\Lambda:=\left\{g;I(g)\leq 1\right.\}$ is the limit set of $\left\{Q^{'}_u\right\}_{0<u<e^{-1}}$ when $u\to 0^+$ $a.s.$, where
		\begin{align*}
			I(g)=\frac{1}{2}\inf_{\hbar\in\mathcal{H};g=Y^h}\int_{0}^{T}|h(s)|^2\mathrm{d}s,
		\end{align*}
		and $Y^h$ solve the following multivalued McKean-Vlasov SDE:
		\begin{align*}
			\begin{cases}
				\mathrm{d}Y^h_t\in b(Y^h_t,\mathcal{L}_{Y^h_t})\mathrm{d}t+\sigma(Y^h_t,\mathcal{L}_{Y^h_t})h(t)\mathrm{d}t-A(Y^h_t)\mathrm{d}t, \\
				Y^h_0=x.
			\end{cases}
		\end{align*}
		
		\begin{proof}
			The proof is analogous to that of Theorem 5.3, and is therefore omitted here.
		\end{proof}
	\end{theorem}

	\subsection{Proofs}\label{ILLMP}
	
	We first state the following continuous dependence result which plays an important role in the proof of Lemma \ref{ILL2}.
	
	\begin{lemma}\label{hcon}
		Suppose that $\tilde{A}$ is a maximal monotone multivalued operator with $0\in Int(D(\tilde{A}))$, $\tilde{b}$ and $\tilde{\sigma}$ are progressive measurable functions that satisfy \hyperref[h1]{\textbf{(H1)}} and \hyperref[h2]{\textbf{(H2)}}. Then $ B_R\ni \hbar\to Z^h $ is continuous with respect to the distance $ d $, where $ B_R:=\left\{\hbar\in\mathcal{H};\frac{1}{2}\int_{0}^{T}|h(s)|^2\mathrm{d}s\leq 2R\right\} $ and $(Z^h,K^h)$ solves the following multi-valued equation:
		\begin{align*}
			\begin{cases}
				\mathrm{d}Z^h_t\in \tilde{b}(Z^h_t,\mathcal{L}_{Z^h_t})\mathrm{d}t+\tilde{\sigma}(Z^h_t,\mathcal{L}_{Z^h_t})h(t)\mathrm{d}t-\tilde{A}(Z^h_t)\mathrm{d}t, \\
				Z^h_0=x\in\overline{D(\tilde{A})}.
			\end{cases}
		\end{align*}
		\begin{proof}
			For any fixed $\hbar_0\in B_R$, let $\hbar\in B_R$ such that $\sup_{t\in [0,T]}|\hbar(t)-\hbar_0(t)|\to 0,$ then
			\begin{align*}
				&|Z^h_t-Z^{h_0}_t|^2\\
				=&-2\,\int_{0}^{t}\langle Z^h_s-Z^{h_0}_s, \mathrm{d}K^h_s-\mathrm{d}K^{h_0}_s \rangle \mathrm{d}s+
				2\,\int_{0}^{t}\langle Z^h_s-Z^{h_0}_s, \tilde{b}(Z^h_s,\mathcal{L}_{Z^h_s})-\tilde{b}(Z^{h_0}_s,\mathcal{L}_{Z^{h_0}_s}) \rangle\mathrm{d}s \\
				& +2\,\int_{0}^{t}\langle Z^h_s-Z^{h_0}_s,  \tilde{\sigma}(Z^h_s,\mathcal{L}_{Z^h_s})h(s)-\tilde{\sigma}(Z^{h_0}_s,\mathcal{L}_{Z^{h_0}_s})h_0(s) \rangle\mathrm{d}s \\
				\leq &\,2\,\int_{0}^{t} \langle Z^h_s-Z^{h_0}_s, \tilde{b}(Z^h_s,\mathcal{L}_{Z^h_s})-\tilde{b}(Z^{h_0}_s,\mathcal{L}_{Z^{h_0}_s})  \rangle\mathrm{d}s   \\
				&+2\,\int_{0}^{t}\langle Z^h_s-Z^{h_0}_s,  \tilde{\sigma}(Z^h_s,\mathcal{L}_{Z^h_s})h(s)-\tilde{\sigma}(Z^{h_0}_s,\mathcal{L}_{Z^{h_0}_s})h(s) \rangle\mathrm{d}s \\
				&+2\,\int_{0}^{t}\langle Z^h_s-Z^{h_0}_s,  \tilde{\sigma}(Z^{h_0}_s,\mathcal{L}_{Z^{h_0}_s})h(s)-\tilde{\sigma}(Z^{h_0}_s,\mathcal{L}_{Z^{h_0}_s})h_0(s) \rangle\mathrm{d}s,
			\end{align*}
			where the last inequality follows from \eqref{mono}.
			
			By using Young's inequality, \textbf{(H1)} and \textbf{(H2)}, we have
			\begin{align*}
				&\sup_{s\in[0,t]}|Z^h_s-Z^{h_0}_s|^2 \\
				\leq&2\int_{0}^{t}\varrho(|Z^h_s-Z^{h_0}_s|^2)\mathrm{d}s+2\,\int_{0}^{t} |Z^h_s-Z^{h_0}_s|\cdot \lVert \tilde{\sigma}(Z^h_s,\mathcal{L}_{Z^h_s})-\tilde{\sigma}(Z^{h_0}_s,\mathcal{L}_{Z^{h_0}_s})\rVert\cdot|h(s)|\mathrm{d}s \\
				&+\sup_{t'\in[0,t]}2\int_{0}^{t'}\langle Z^h_s-Z^{h_0}_s,  \tilde{\sigma}(Z^{h_0}_s,\mathcal{L}_{Z^{h_0}_s})(h(s)-h_0(s)) \rangle\mathrm{d}s \\
				\leq & \,2\int_{0}^{t}\varrho(|Z^h_s-Z^{h_0}_s|^2)\mathrm{d}s+4R^{1/2}\left( \int_{0}^{t}|Z^h_s-Z^{h_0}_s|^2\cdot\lVert \tilde{\sigma}(Z^h_s,\mathcal{L}_{Z^h_s})-\tilde{\sigma}(Z^{h_0}_s,\mathcal{L}_{Z^{h_0}_s})\rVert^2 \mathrm{d}s\right)^{1/2} \\
				&+\sup_{t'\in[0,t]}2\int_{0}^{t'}\langle Z^h_s-Z^{h_0}_s,  \tilde{\sigma}(Z^{h_0}_s,\mathcal{L}_{Z^{h_0}_s})(h(s)-h_0(s)) \rangle\mathrm{d}s,
			\end{align*}
			which implies
			\begin{align}\label{511}
				\sup_{s\in[0,t]}|Z^h_s-Z^{h_0}_s|^2\leq & \,\frac{1}{2}\sup_{s\in[0,t]}|Z^h_s-Z^{h_0}_s|^2+(8R+2)\int_{0}^{t}\varrho(\sup_{t'\in[0,s]}|Z^h_{t'}-Z^{h_0}_{t'}|^2)\mathrm{d}s \nonumber\\
				&+ \sup_{t'\in[0,t]}2\int_{0}^{t'}\langle Z^h_s-Z^{h_0}_s,  \tilde{\sigma}(Z^{h_0}_s,\mathcal{L}_{Z^{h_0}_s})(h(s)-h_0(s)) \rangle\mathrm{d}s.
			\end{align}
			Here we denote the last term of (\ref{511}) as $J(t)$ and define
			\begin{align*}
				\nu(t):=\int_{0}^{t}\tilde{\sigma}(Z^{h_0}_s,\mathcal{L}_{Z^{h_0}_s})(h(s)-h_0(s))\mathrm{d}s.
			\end{align*}
			Using integration by parts, we have
			\begin{align*}
				&\int_{0}^{t}\langle\, Z^h_s-Z^{h_0}_s,  \tilde{\sigma}(Z^{h_0}_s,\mathcal{L}_{Z^{h_0}_s})(h(s)-h_0(s)) \rangle\mathrm{d}s \\
				=&\langle\, \nu(t), Z^h_t-Z^{h_0}_t\rangle-\int_{0}^{t}\langle \,\nu(s),\tilde{b}(Z^h_s,\mathcal{L}_{Z^h_s})-\tilde{b}(Z^{h_0}_s,\mathcal{L}_{Z^{h_0}_s})\rangle\mathrm{d}s \\
				&-\int_{0}^{t}\langle \,\nu(s),\tilde{\sigma}(Z^h_s,\mathcal{L}_{Z^h_s})h(s)-\tilde{\sigma}(Z^{h_0}_s,\mathcal{L}_{Z^{h_0}_s})h_0(s)\rangle\mathrm{d}s +\int_{0}^{t}\langle \,\nu(s),\mathrm{d}K^h_s-\mathrm{d}K^{h_0}_s\rangle.
			\end{align*}
			Following the proof of Lemma \ref{yosifin}, we have for $h,h_0\in B_R$ and $p\geq 1$
			\begin{align*}
				\sup_{t\in[0,T]}|Z^h_t|^p+|K^h|^0_T\leq C_R \text{ and }\sup_{t\in[0,T]}|Z^{h_0}_t|^p+|K^{h_0}|^0_T\leq C_R.
			\end{align*}
			By \textbf{(H1)} and \textbf{(H2)} we can deduce that
			\begin{align*}
				\sup_{t\in[0,T]} \int_{0}^{t}\langle\, Z^h_s-Z^{h_0}_s,  \tilde{\sigma}(Z^{h_0}_s,\mathcal{L}_{Z^{h_0}_s})(h(s)-h_0(s)) \rangle\mathrm{d}s\leq C_R\sup_{t\in[0,T]}|\nu(t)|.
			\end{align*}
			
			Next we will prove that
			\begin{align*}
				\sup_{t\in[0,T]}|\nu(t)|\to 0\quad\textit{as } \sup_{s\in[0,T]}|\hbar(s)-\hbar_0(s)|\to 0.
			\end{align*}
			Let $\varphi$ be a mollifier with $\textit{supp}\,\varphi\subset(-1,1),\,\,\varphi\in C^{\infty}$ and $\int_{-1}^{1}\varphi(t)\mathrm{d}t=1$. Consider the following function
			\begin{align*}
				\rho_n(t)=\int_{0}^{T}\tilde{\sigma}(Z^{h_0}_{u},\mathcal{L}_{Z^{h_0}_{u}})\varphi_n(t-u)\mathrm{d}u,
			\end{align*}
			where $\varphi_n(u):=n\varphi(nu)$.
			We can easily get that $t\to\rho_n(t)$ is differentiable and
			\begin{align*}
				\int_{0}^{T}\lVert \tilde{\sigma}(Z^{h_0}_t,\mathcal{L}_{Z^{h_0}_t})-\rho_n(t)\rVert^2\mathrm{d}t\to 0\quad\textit{as} \,\,\, n\to\infty.
			\end{align*}
			Hence
			\begin{align*}
				\sup_{t\in[0,T]} |\nu(t)|=&\sup_{t\in[0,T]}\lvert \int_{0}^{t}\tilde{\sigma}(Z^{h_0}_s,\mathcal{L}_{Z^{h_0}_s})(h(s)-h_0(s))\mathrm{d}s \rvert \\
				\leq& \sup_{t\in[0,T]}\lvert\int_{0}^{t}(\tilde{\sigma}(Z^{h_0}_s,\mathcal{L}_{Z^{h_0}_s})-\rho_n(s))(h(s)-h_0(s))\mathrm{d}s \rvert+\sup_{t\in[0,T]}\lvert\int_{0}^{t}\rho_n(s)(h(s)-h_0(s))\mathrm{d}s \rvert \\
				=:&J_1+J_2.
			\end{align*}
			By using H\"older's inequality, for any $\delta>0$, there exists $n_0\in\mathbb{N}$ such that
			\begin{align*}
				J_1\leq&\left(\int_{0}^{T}\lVert \tilde{\sigma}(Z^{h_0}_s,\mathcal{L}_{Z^{h_0}_s})-\rho_{n_0}(s)\rVert^2\mathrm{d}s\right)^{1/2}\left( \int_{0}^{T}|h(s)-h_0(s)|^2\mathrm{d}s \right)^{1/2} \\
				\leq& \,\,4R^{1/2}\delta^{1/2}.
			\end{align*}
			For $J_2$, denote the derivative of $\rho_{n_0}(t)$ by $\rho'_{n_0}(t)$, using integration by parts, we obtain
			\begin{align*}
				J_2\leq \sup_{t\in[0,T]}|\hbar(t)-\hbar_0(t)|\left( \sup_{t\in[0,T]}\lVert \rho_{n_0}(t)\rVert+\int_{0}^{T}\lVert\rho'_{n_0}(t)\rVert\mathrm{d}t \right),
			\end{align*}
			Thus we obtain that $ J_2\to 0 $ as $\sup_{t\in[0,T]}|\hbar(t)-\hbar_0(t)|\to 0$. By Lemma \ref{bhr} we get
			\begin{align*}
				\sup_{t\in[0,T]}|Z^h_t-Z^{h_0}_t|\to 0 \quad \textit{as}\quad \sup_{t\in[0,T]}|\hbar(t)-\hbar_0(t)|\to 0.
			\end{align*}
		\end{proof}
	\end{lemma}

	The following two results will play an important role in the proof of Proposition 5.7.
	
	\begin{lemma}\label{ILL2}
		Assume that \hyperref[h0]{\textbf{(H0)}}-\hyperref[h4]{\textbf{(H4)}} hold. Let $a$ be a fixed positive number, then for every $\beta>0$ and $R>0$, there exist $\epsilon_{0},\eta_0>0$ such that for any $0<\epsilon<\epsilon_{0}$ and $0<\eta<\eta_0$,
		\begin{align*}
			P\left\{d(\sqrt{\epsilon} W,\hbar)<\eta,\,d(X^\epsilon,g)>\beta\right\}\leq \exp(-R/\epsilon)
		\end{align*}
		holds for every $\hbar$ with $I(g)\leq a$ and $g=Y^h$.
		
		\begin{proof}
			Consider the diffusion
			\begin{align*}
				Z^\epsilon_t:=
				\begin{pmatrix}
					\sqrt{\epsilon} W_t \\
					X^\epsilon_t
				\end{pmatrix} ,
			\end{align*}
			then $Z^\epsilon_t$ is the solution to the following MMVSDE:
			\begin{align*}
				\begin{cases}
					\mathrm{d}Z^\epsilon_t\in \tilde{b}_\epsilon(Z^\epsilon_t,\mathcal{L}_{Z^\epsilon_t})\mathrm{d}t+\sqrt{\epsilon}\tilde{\sigma}_\epsilon(Z^\epsilon_t,\mathcal{L}_{Z^\epsilon_t})\mathrm{d}W_t-\tilde{A}_\epsilon(Z^\epsilon_t)\mathrm{d}t, \\
					Z^\epsilon_0=\left(^0_x\right),
				\end{cases}
			\end{align*}
			where
			\begin{align*}
				&\tilde{b}_\epsilon(x',\mu')=\tilde{b}_\epsilon\left(\begin{pmatrix}y \\ x \end{pmatrix},\begin{pmatrix}\nu \\ \mu \end{pmatrix}\right)=\begin{pmatrix}0 \\ b_\epsilon(x,\mu) \end{pmatrix}:\mathbb{R}^{d+d}\times \mathcal{P}_2(\mathbb{R}^{d+d})\to \mathbb{R}^{d+d}, \\ &\tilde{\sigma}_\epsilon(x',\mu')=\tilde{\sigma}_\epsilon\left(\begin{pmatrix}y \\ x \end{pmatrix},\begin{pmatrix}\nu \\ \mu \end{pmatrix}\right)=\begin{pmatrix}I \\ \sigma_\epsilon(x,\mu) \end{pmatrix}:\mathbb{R}^{d+d}\times \mathcal{P}_2(\mathbb{R}^{d+d})\to \mathbb{R}^{d+d}\otimes \mathbb{R}^d
			\end{align*}
			and
			\begin{align*}
				\tilde{A}_\epsilon(x')=\tilde{A}_\epsilon\left(\begin{pmatrix}y \\ x \end{pmatrix}\right)=\begin{pmatrix}0 \\ A_\epsilon(x) \end{pmatrix}:\mathbb{R}^{d+d}\to 2^{\mathbb{R}^{d+d}}
			\end{align*}
			for any $x,y\in\mathbb{R}^d$ and $\mu,\nu\in\mathcal{P}_2(\mathbb{R}^d)$.
			
			For each $g\in C([0,T];\mathbb{R}^{d+d})$, define
			\begin{align*}
				\tilde{I}(g)\coloneqq\frac{1}{2}\inf_{\hbar\in\mathcal{H};g=Z^h}\int_{0}^{T}|h(t)|^2\mathrm{d}t,
			\end{align*}
			where $(Z^h,K^h)$ is the solution to the following equation:
			\begin{align}\label{izh}
				\begin{cases}
					\mathrm{d}Z^h_t\in \tilde{b}(Z^h_t,\mathcal{L}_{Z^h_t})\mathrm{d}t+\tilde{\sigma}(Z^h_t,\mathcal{L}_{Z^h_t})h(t)\mathrm{d}t-\tilde{A}(Z^h_t)\mathrm{d}t, \\
					Z^h_0=\left(^0_x\right).
				\end{cases}
			\end{align}
			Here $\tilde{b}, \tilde{\sigma}$ and $\tilde{A}$ are defined similarly as before.		
			
			It is obvious to note that $\tilde{b},\tilde{b}_\epsilon,\tilde{\sigma},\tilde{\sigma}_\epsilon,\tilde{A}$ and $\tilde{A}_\epsilon$ also satisfy \textbf{(H0)}-\textbf{(H4)}. Thus (\ref{izh}) is decomposed into two systems, i.e.
			\begin{align*}
				Z^h=\begin{pmatrix}
					\hbar \\
					Y^h
				\end{pmatrix},
			\end{align*}
			where $Y^h$ is the solution to the following equation:
			\begin{align*}
				\begin{cases}
					\mathrm{d}Y^h_t\in b(Y^h,\mathcal{L}_{Y^h_t})\mathrm{d}t+\sigma(Y^h,\mathcal{L}_{Y^h_t})h(t)\mathrm{d}t-A(Y^h_t)\mathrm{d}t, \\
					Y^h_0=x.
				\end{cases}
			\end{align*}
			
			Fix $\beta>0$ and $R>0$. Then
			\begin{align*}
				P\left\{d(\sqrt{\epsilon} W,\hbar)<\eta,\,d(X^\epsilon,Y^h)>\beta\right\}=P(Z^\epsilon\in \digamma_\eta),
			\end{align*}
			where
			\begin{align*}
				\digamma_\eta:=\left\{\chi=\begin{pmatrix} \chi_1 \\ \chi_2 \end{pmatrix}\in C([0,T];\mathbb{R}^{d+d});\,d(\chi_1,\hbar)<\eta,\,d(\chi_2,Y^h)>\beta\right\}.
			\end{align*}
			By Lemma \ref{hcon} we know that $\hbar\to Z^h$ is continuous with respect the distance $d$ when $\hbar\in B_R$. Thus there exists $\eta>0$ such that if $\chi_1\in B_R$ and $d(\chi_1,\hbar)\leq \eta$, then
			\begin{align*}
				d(Y^{\chi_1},Y^h)\leq \beta.
			\end{align*}
			This implies that for such $\eta$, there does not exist any trajectory $\chi\in \digamma_\eta$ such that
			\begin{align*}
				\frac{1}{2}\int_{0}^{T}|\chi_1(t)|^2\mathrm{d}t\leq 2R,\quad  \chi=\begin{pmatrix} \chi_1 \\  Y^{\chi_1} \end{pmatrix}
			\end{align*}
			which means
			\begin{align*}
				\inf_{g\in\digamma_\eta} \tilde{I}(g)\geq 2R.
			\end{align*}
			Thus by Theorem \ref{thmldp} and Definition \ref{ldpdefin}, we get the desired result.
			
		\end{proof}

	\end{lemma}

	\begin{lemma}\label{item1}
		For $\forall c>1$ and $\forall \epsilon>0$, there exists a positive integer $j_0(\omega)$ almost surely finite such that $\forall j>j_0$,
		\begin{align*}
			d(Q_{c^j},\Lambda)<\epsilon, \quad \text{where}\,\, d(x,B)=\inf_{y\in B}\lVert x-y\rVert.
		\end{align*}
		
		\begin{proof}
			Consider the set $\Lambda_\epsilon:=\left\{g;d(g,\Lambda)\geq \epsilon\right\}$. By definition and the lower semicontinuity of $I(g)$ we have that $\inf_{g\in \Lambda_\epsilon}I(g)>1$, so there exists a positive number $\eta$ such that $\inf_{g\in  \Lambda_\epsilon}I(g)>1+\eta$. Therefore, by (\ref{ldpc}) we have
			\begin{align*}
				\limsup\limits_{u\to\infty} \frac{1}{\log\log u}P(Q_u\in\Lambda_\epsilon)\leq -(1+\eta).
			\end{align*}
			Thus for $j$ large enough and for every $c>1$, there exists a constant $C>0$ such that
			\begin{align*}
				P(Q_{c^j}\in\Lambda_\epsilon)\leq \exp\left\{-(1+\eta)\log\log(c^j)\right\}=\frac{C}{j^{1+\eta}}.
			\end{align*}
			Since $\sum_{j=1}^{\infty}P(Q_{c^j}\in\Lambda_\epsilon)<\infty$, by Borel-Cantelli lemma we get $P(\limsup_j[d(Q_{c^j},\Lambda)\geq\epsilon])=0$, i.e. $\lim_{j\to\infty}d(Q_{c^j},\Lambda)=0\,$ a.s..
		\end{proof}
		
	\end{lemma}

	We are now able to present the proof of Theorem 5.3, which will be divided into two propositions.

	\begin{proposition}[Relatively Compactness]\label{ILLP1}
		Under Assumption \ref{Cn}, for every $\epsilon>0$ there exists almost surely a positive number $u_0$ such that for every $u>u_0$,
		\begin{align*}
			d(Q_u,\Lambda)<\epsilon.
		\end{align*}

		\begin{proof}
			Let $c> 1$. For any $u>1$, there exists some $j\in\mathbb{N}$ sufficiently large such that $c^{j-1}\leq u\leq c^j$. Then
			\begin{align*}
				d(Q_u,\Lambda)\leq&\, d(Q_{c^j},\Lambda)+d(Q_u,Q_{c^j}) \\
				\leq&\, d(Q_{c^j},\Lambda)+\lVert Q_u-\Gamma_{\psi(u)}\circ \Gamma_{\psi(c^j)}^{-1}(Q_{c^j})\rVert \\
				&+\lVert \Gamma_{\psi(u)}\circ \Gamma_{\psi(c^j)}^{-1}(Q_{c^j})-Q_{c^j}\rVert \\
				=:& J^d_1+J^d_2+J^d_3.
			\end{align*}
			For arbitrarily small $\epsilon>0$, Lemma \ref{item1} ensures that $J^d_1$ is bounded by ${\epsilon}$ and $Q_{c^j}$ is almost surely bounded with $j$ large enough. Moreover, for every $u\in[c^{j-1},c^j]$, $\forall \delta>0$
			\begin{align*}
				1\geq\frac{\psi(u)}{\psi(c^j)}\geq\frac{\psi(c^{j-1})}{\psi(c^j)}\geq \frac{1-\delta}{\sqrt{c}},
			\end{align*}
			so that if we choose $c>1$ small enough, then by using $(d)$ of Definition \ref{gadef} we get
			\begin{align*}
				J^d_3=\lVert \Gamma_{\psi(u)}\circ \Gamma_{\psi(c^j)}^{-1}(Q_{c^j})-Q_{c^j}\rVert\leq\epsilon.
			\end{align*}
			By definition of $Q_u(t)$ and $(b)$ of Definition \ref{gadef}, we have
			\begin{align*}
				J^d_2=&\sup_{t\in[0,T]}|\Gamma_{\psi(u)}(Y_{ut})-\Gamma_{\psi(u)}(Y_{c^jt})| \\
				\leq& \sup_{t\in[0,T]}|\Gamma_{\psi(c^{j-1})}(Y_{ut})-\Gamma_{\psi(c^{j-1})}(Y_{c^jt})|.
			\end{align*}
			Thus we rewrite
			\begin{align*}
				d(Q_u,\Lambda)<&\,2\epsilon+\sup_{t\in[0,T]}|\Gamma_{\psi(c^{j-1})}(Y_{ut})-\Gamma_{\psi(c^{j-1})}(Y_{c^jt})| \\
				\leq&\, 2\epsilon+\sup_{\substack{0\leq s\leq T \\ s/c\leq t\leq s}}|\Gamma_{\psi(c^{j-1})}(Y_{c^jt})-\Gamma_{\psi(c^{j-1})}(Y_{c^js})|.
			\end{align*}
			By adding and substracting the terms $Q_{c^{j}}(t)$ and $Q_{c^{j}}(s)$ and using (d) of Definition \ref{Cn}, we obtain
			\begin{align*}
				d(Q_u,\Lambda)<&\,2\epsilon+\sup_{\substack{0\leq s\leq T}}|\Gamma_{\psi(c^{j-1})}\circ\Gamma_{\psi(c^j)}^{-1}(Q_{c^j}(s))-Q_{c^j}(s)| \\
				&+\,\sup_{\substack{0\leq t\leq T}} |\Gamma_{\psi(c^{j-1})}\circ\Gamma_{\psi(c^j)}^{-1}(Q_{c^j}(t))-Q_{c^j}(t)|+\,\sup_{\substack{0\leq s\leq T \\ s/c\leq t\leq s}} |Q_{c^j}(s)-Q_{c^j}(t)| \\
				\leq &\,4\epsilon+\sup_{\substack{0\leq s\leq T \\ s/c\leq t\leq s}} |Q_{c^j}(s)-Q_{c^j}(t)|.
			\end{align*}
			Let $g\in\Lambda$ such that $\lVert Q_{c^j}-g \rVert<d(Q_{c^j},\Lambda)+\epsilon$. Then in particular, $\lVert Q_{c^j}-g\rVert <2\epsilon$ and we have
			\begin{align*}
				d(Q_u,\Lambda)\leq 8\epsilon+\sup_{g\in\Lambda}\sup_{\substack{0\leq s\leq T \\ s/c\leq t\leq s}}|g(t)-g(s)|.
			\end{align*}
			Since $\Lambda$ is compact, using Ascoli-Arzel\`a's theorem we can deduce that there exists a constant $c>1$ such that
			\begin{align*}
				\sup_{g\in\Lambda}\sup_{\substack{0\leq s\leq T \\ s/c\leq t\leq s}}|g(t)-g(s)|\leq \epsilon.
			\end{align*}
			Therefore we get the desired result
			\begin{align*}
				d(Q_u,\Lambda)<\epsilon.
			\end{align*}
			
		\end{proof}
		
	\end{proposition}

	\begin{proposition}[The set of limit points]\label{ILLP2}
		Let $g\in \Lambda$ be such that $I(g)<1$, Then for any $\epsilon>0$ there exists $c_\epsilon>1$ such that for every $c>c_\epsilon$,
		\begin{align*}
			P\left\{d(Q_{c^j},g)<\epsilon\,\,\, i.o.\right\}=1.
		\end{align*}
		
		\begin{proof}
			For any $g\in \Lambda$ with $I(g)<1$, let $\hbar_\cdot=\int_0^\cdot h(s)ds$ with $g=Z^h$. Recall that $\mathcal{W}^{(u)}_t=\frac{W_{ut}}{\sqrt{u}}$.  For any $\epsilon>0$ and $\eta>0$, define
			\begin{align*}
				E_j=\left\{\big\lVert \frac{\mathcal{W}^{(c^j)}}{\sqrt{\log\log c^j}}-\hbar\big\rVert< \eta\right\},\quad
				F_j=\left\{ \lVert Q_{c^j}-g \rVert<\epsilon\right\}.
			\end{align*}
			Using Lemma \ref{ILL2} we have that for $j$ large enough and $\eta$ small enough, there exists some constant $C>0$ such that
			\begin{align}\label{lp1}
				P(E_j)-P(F_j)=P(E_j\cap F_j^c)\leq \exp\left\{-2\log\log c^j\right\}=\frac{C}{j^2}.
			\end{align}
			Since $P(\limsup_{j\to\infty} E_j)=1$ by Strassen's law. Therefore, $\sum_{j}P(E_j)=\infty$. However, by (\ref{lp1}) we also have
			\begin{align*}
				\sum_j(P(E_j)-P(F_j))< \infty.
			\end{align*}
			Thus we obtain $\sum_j P(F_j)=\infty$, by Borel-Cantelli lemma we get our result
			\begin{align*}
				P\left\{\lVert Q_{c^j}-g\rVert<\epsilon\,\,\, i.o.\right\}=1.
			\end{align*}
			
		\end{proof}
		
	\end{proposition}

	{\bf Acknowledgements}
	
	Cheng L. is supported by Natural Science Foundation of China (12001272, 12171173). Liu W. is supported by Natural Science Foundation of China (No. 12471143, 12131019). Qiao H. is supported by Natural Science Foundation of China (No. 12071071) and  the Jiangsu Provincial Scientific Research Center of Applied Mathematics (No. BK20233002).

\end{document}